\documentclass[english]{article}
\usepackage[T1]{fontenc}
\usepackage[latin9]{inputenc}
\usepackage{geometry}
\usepackage{babel}
\usepackage{mathrsfs}
\usepackage{amsmath}
\usepackage{amsthm}
\usepackage{amssymb}
\usepackage[all]{xy}
\usepackage[unicode=true,pdfusetitle,
 bookmarks=true,bookmarksnumbered=false,bookmarksopen=false,
 breaklinks=false,pdfborder={0 0 1},backref=false,colorlinks=false]
 {hyperref}

\makeatletter
\theoremstyle{plain}
\newtheorem{thm}{\protect\theoremname}[section]
\theoremstyle{definition}
\newtheorem{defn}[thm]{\protect\definitionname}
\theoremstyle{plain}
\newtheorem{prop}[thm]{\protect\propositionname}
\theoremstyle{remark}
\newtheorem{rem}[thm]{\protect\remarkname}
\theoremstyle{definition}
\newtheorem{xca}[thm]{\protect\exercisename}
\theoremstyle{plain}
\newtheorem{lem}[thm]{\protect\lemmaname}
\theoremstyle{plain}
\newtheorem{conjecture}[thm]{\protect\conjecturename}
\theoremstyle{definition}
\newtheorem{example}[thm]{\protect\examplename}
\newenvironment{lyxlist}[1]
	{\begin{list}{}
		{\settowidth{\labelwidth}{#1}
		 \setlength{\leftmargin}{\labelwidth}
		 \addtolength{\leftmargin}{\labelsep}
		 }}
	{\end{list}}
\theoremstyle{plain}
\newtheorem{cor}[thm]{\protect\corollaryname}

\newcommand{\etalchar}[1]{$^{#1}$}
\providecommand{\bysame}{\leavevmode\hbox to3em{\hrulefill}\thinspace}
\providecommand{\MR}{\relax\ifhmode\unskip\space\fi MR }

\providecommand{\href}[2]{#2}

\usepackage{xy}
\usepackage{mathrsfs}
\usepackage{hyperref}
\usepackage{color}
\usepackage{xcolor}
\usepackage{amsmath}
\hypersetup{colorlinks=true,linkcolor=purple,citecolor=magenta}

\numberwithin{thm}{subsection}
\usepackage{xspace}
\newcommand{\bR}{\ensuremath{\mathbf{R}}\xspace}
\newcommand{\sX}{\ensuremath{\mathcal{X}}\xspace}
\newcommand{\sY}{\ensuremath{\mathcal{Y}}\xspace}

\renewcommand{\ge}{\geqslant}
\renewcommand{\le}{\leqslant}
\DeclareMathOperator*{\fibprod}{\times}
\DeclareMathOperator{\codim}{\operatorname{codim}}
\newcommand{\virdim}{\operatorname{vdim}}
\newcommand{\fibprodR}{\fibprod^\bR}
\newcommand{\cl}{{\mathrm{cl}}}

\usepackage{mathtools}
\DeclarePairedDelimiter\abs{\lvert}{\rvert}

\makeatother

\providecommand{\conjecturename}{Conjecture}
\providecommand{\corollaryname}{Corollary}
\providecommand{\definitionname}{Definition}
\providecommand{\examplename}{Example}
\providecommand{\exercisename}{Exercise}
\providecommand{\lemmaname}{Lemma}
\providecommand{\propositionname}{Proposition}
\providecommand{\remarkname}{Remark}
\providecommand{\theoremname}{Theorem}

\begin{document}
\title{Beijing notes on the categorical local Langlands conjecture}
\author{David Hansen\\
with an appendix by Adeel Khan}

\maketitle
\noindent\begin{minipage}[t]{1\columnwidth}%
\begin{flushright}
\emph{A fool can ask more questions in a minute$\;\;\;\;\;\;$ }\\
\emph{than a wise man can answer in an hour.}
\par\end{flushright}%
\end{minipage}

\tableofcontents{}

\subsection*{Acknowledgements}

These are expanded notes written for a lecture series at the Morningside
Center of Mathematics in July 2023. I'm very grateful for the opportunity
to give these lectures, which significantly pushed me to clarify a
number of vague ideas and write them out carefully. Hopefully this
written account will be helpful to others in the field. On a personal
level, I'm especially grateful to Shizhang Li for his extremely kind
and generous hospitality during my visit to Beijing.

It's a pleasure to thank Pramod Achar, Alexander Bertoloni Meli, Jean-Fran\c{c}ois
Dat, Laurent Fargues, Jessica Fintzen, Wee Teck Gan, Linus Hamann,
Eugen Hellmann, Christian Johansson, Tasho Kaletha, Adeel Khan, Teruhisa
Koshikawa, Lucas Mann, Sam Raskin, Peter Scholze, Sug Woo Shin, Sandeep
Varma, Eva Viehmann, Xinwen Zhu, Yihang Zhu, and Konrad Zou for some
very helpful and inspiring discussions. I'd particularly like to thank
Adeel Khan for proving Proposition \ref{prop:soap} and agreeing to
include it as an appendix; Linus Hamann for teaching me about Eisenstein
series and helping me get the signs right; and Alexander Bertoloni
Meli for many discussions and valuable encouragement as I worked out
these ideas. I also got crucial inspiration from beautiful lectures
by Ian Gleason, Teruhisa Koshikawa,  Jo\~{a}o Louren\c{c}o, and Xinwen Zhu.
Finally, Christian Johansson pointed out in January 2023 that the
paper \cite{Bez} might be worth looking at; as always, his advice
was on the mark.

\subsection*{What these notes are trying to do}

Fix a prime $p$, and a nonarchimedean local field $E$ with uniformizer
$\varpi$ and residue field $\mathbf{F}_{q}$ of characteristic $p$.
Let $G/E$ be a connected reductive group, which we assume for simplicity
is quasisplit. The \emph{local Langlands conjecture, }in its vaguest
form, seeks to parametrize the irreducible smooth representations
of $G(E)$ in terms of Galois-theoretic data. A huge amount of effort
has been expended on making this conjecture precise, and proving it
in many cases. In its modern form, the conjecture proposes that irreducible
smooth representations of $G(E)$ can be parametrized by (suitable)
pairs $(\phi,\rho)$ where $\phi:W_{E}\times\mathrm{SL}_{2}\to\phantom{}^{L}G$
is an \emph{L-}parameter and $\rho$ is an irreducible algebraic representation
of a group closely related to the centralizer group $S_{\phi}=\mathrm{Cent}_{\hat{G}}(\phi)$.
Moreover, this parametrization should depend only on a choice of Whittaker
datum for $G$, and should satisfy many good properties. We refer
to \cite{KalN} for a precise formulation.

In recent years, the subject has undergone a major transformation,
thanks to Fargues's amazing discovery that ideas from \emph{geometric
Langlands }can be imported into the study of local Langlands, by reinterpreting
the basic structures in local Langlands in terms of $G$-bundles on
the Fargues--Fontaine curve \cite{Far,FarguesMSRI}. The foundations
for this rebuilding of the field were then laid in the revolutionary
manuscripts \cite{FS,SchD}. Among other things, they succeed in constructing
a canonical map $\pi\mapsto\varphi_{\pi}$ from irreducible smooth
representations towards semisimple \emph{L}-parameters, which should
be the semisimplification of the ``true'' local Langlands parametrization.

In a parallel stream, the idea emerged that the derived category of
all smooth $G(E)$-representations should embed fully faithfully into
ind-coherent sheaves on the stack of \emph{L}-parameters. This was
made precise (independently) by Hellmann \cite{Hel}, Zhu \cite{Zhu},
and Ben-Zvi--Chen--Helm--Nadler. It is then natural to conjecture
that some form of the geometric Langlands conjecture holds in this
setting, namely that the category of all sheaves on $\mathrm{Bun}_{G}$
should be equivalent to the category of ind-coherent sheaves on the
stack of \emph{L}-parameters. A formal conjecture along these lines
was proposed by Fargues--Scholze \cite[Conjecture I.10.2]{FS}, without
however pinning down the functor which should give the equivalence.
An alternative formulation was proposed by Zhu, working instead with
sheaves on the stack of $G$-isocrystals \cite{Zhu}.

These notes should be regarded as a minor supplement to the works
mentioned above. More concretely, we are guided by the following motivating
questions, in the setting of Fargues-Scholze:
\begin{enumerate}
\item Can we give an unconditional formulation of the categorical local
Langlands conjecture?
\item What properties and compatibilities should the categorical conjecture
enjoy? What additional conjectures do these properties suggest?
\item How does the categorical conjecture encode the classical local Langlands
conjecture?
\item How do the non-basic strata in $\mathrm{Bun}_{G}$ fit into the picture?
\end{enumerate}
These notes are \emph{not }intended as a detailed introduction to
the structures required to formulate these questions. More specifically,
we will assume the reader is somewhat familiar with the modern expectations
regarding the local Langlands correspondence, as outlined in (say)
\cite{KalN} and \cite{KalICM}. We will also assume some familiarity
with \cite{FS}, and with the philosophy of the classical geometric
Langlands program.

There are exercises scattered throughout the text. They vary widely
in difficulty, but none of them are open problems. Many of the harder
exercises will be fully solved in \cite{HM}.

\subsection*{Notation and conventions}

We will introduce various notations throughout the text. Here we briefly
mention several perhaps non-standard conventions. If $X$ is a disjoint
union of finite type (derived) Artin stacks, we write $\mathrm{Coh}(X)$
for the category of bounded coherent complexes on $X$ with quasicompact
support. On the other hand, we write $\mathrm{Perf}(X)$ for all dualizable
objects in $\mathrm{QCoh}(X)$, so there is no support condition.
This leads to the slightly unfortunate fact that $\mathrm{Perf}$
may not be contained in $\mathrm{Coh}$.

If $X$ is an object in a triangulated or stable $\infty$-category
$\mathcal{D}$, we say $X$ admits a filtration with graded pieces
$A_{1},\dots,A_{n}$ if there exists a sequence of maps $0=X_{0}\to X_{1}\to X_{2}\to\cdots\to X_{n}=X$
with $A_{j}\simeq\mathrm{cone}(X_{j-1}\to X_{j})$ for all $1\leq j\leq n$.
Similar remarks apply to functors between stable $\infty$-categories.

\section{The categorical conjecture}

Fix a prime $p$, and a nonarchimedean local field $E$ with uniformizer
$\varpi$ and residue field $\mathbf{F}_{q}$ of char. $p$. Fix also
a prime $\ell\neq p$. Let $\Lambda$ be a $\mathbf{Z}_{\ell}$-algebra
containing a fixed choice of $\sqrt{q}$.

Let $G$ be a connected reductive group over $E$. We often assume
that $G$ is quasisplit, both for simplicity and because the categorical
conjecture requires this assumption. Let $Q$ be the maximal finite
quotient of $W_{E}$ which acts faithfully on $\hat{G},$ and put
$^{L}G=\hat{G}\rtimes Q$, which we regard as a linear algebraic group
over $\mathbf{Z}_{\ell}[\sqrt{q}]$. We write $\mathrm{Rep}_{\Lambda}(^{L}G)$
and $\mathrm{Rep}_{\Lambda}(\hat{G})$ for the evident tensor categories
of algebraic representations on finite projective $\Lambda$-modules. 

\subsection{The automorphic side}

The basic geometric object is the stack $\mathrm{Bun}_{G}$ of $G$-bundles
on the Fargues-Fontaine curve. We will not review this in detail here,
since it is now a classical object.\footnote{I am joking.} The key
feature of its geometry is that it has a Harder-Narasimhan stratification
by locally closed substacks $\mathrm{Bun}_{G}^{b}$ indexed by elements
of the Kottwitz set $b\in B(G)$ \cite{FarguesGtorsor}. Each stratum
is a classifying stack for a certain group v-sheaf $\tilde{G}_{b}$,
which is an extension of the constant group sheaf $\underline{G_{b}(E)}$
by a connected unipotent group. By a deep theorem of Viehmann \cite{Vie},
the topology on $|\mathrm{Bun}_{G}|\cong B(G)$ is exactly the Newton
partial order topology on $B(G)$. We write $b\preceq b'$ if $b'\in\overline{\{b\}}$,
so the minimal elements are exactly the basic $b$.

On the automorphic side, the key player is the category $D(\mathrm{Bun}_{G},\Lambda)=D_{\mathrm{lis}}(\mathrm{Bun}_{G},\Lambda)$
of sheaves on $\mathrm{Bun}_{G}$. This is glued semi-orthogonally
from analogous categories of sheaves attached to each Harder-Narasimhan
stratum, which turn out to be much simpler. More precisely, for each
stratum we have a canonical t-exact tensor equivalence
\[
D(\mathrm{Bun}_{G}^{b},\Lambda)\cong D(G_{b}(E),\Lambda),
\]
where the right-hand side is simply the derived category of the abelian
category of smooth $\Lambda[G_{b}(E)]$-modules. This equivalence
is induced by the functors $s_{b\natural}$ and $s_{b}^{\ast}$ associated
with the tautological map $s_{b}:[\ast/\underline{G_{b}(E)}]\to\mathrm{Bun}_{G}^{b}$,
which turn out to be mutually inverse t-exact tensor equivalences.
We will \emph{always }identify $D(\mathrm{Bun}_{G}^{b},\Lambda)$
and $D(G_{b}(E),\Lambda)$ in this way.

For any $b\in B(G)$, there are functors $i_{b}^{\ast}:D(\mathrm{Bun}_{G},\Lambda)\to D(G_{b}(E),\Lambda)$
and $i_{b!}:D(G_{b}(E),\Lambda)\to D(\mathrm{Bun}_{G},\Lambda)$,
which do ``exactly what you think they do.'' The functor $i_{b!}$
has a right adjoint $i_{b}^{!}$, while $i_{b}^{\ast}$ has a right
adjoint $i_{b\ast}$ and also a left adjoint $i_{b\sharp}$. This
left adjoint is a special feature of the situation. The three pushforwards
are linked by natural transformations
\[
i_{b\sharp}\to i_{b!}\to i_{b\ast}
\]
which induce the identity after applying $i_{b}^{\ast}$. Note that
$i_{b\ast}A$ can only have nonzero stalks at points which are specializations
of $b$, while (more strangely) $i_{b\sharp}A$ can only have nonzero
stalks at points which are generizations of $A$.

Let us briefly recall the construction of these functors, especially
the pushforwards $i_{b\sharp}$ and $i_{b!}$ and the natural transformations
mentioned above. For any $b\in B(G)$, Fargues-Scholze construct an
auxiliary diagram of small v-stacks
\[
\xymatrix{\mathcal{M}_{b}\ar[d]^{q_{b}}\ar[r]^{\pi_{b}} & \mathrm{Bun}_{G}\\{}
[\ast/\underline{G_{b}(E)}]
}
\]
where $\pi_{b}$ is cohomologically smooth and surjects onto the open
substack $\mathrm{Bun}_{G}^{\preceq b}$ of points which are ``more
semistable'' than $b$. Moreover, the map $q_{b}$ has a canonical
section given by a closed immersion $[\ast/\underline{G_{b}(E)}]\to\mathcal{M}_{b}$.
Puncturing $\mathcal{M}_{b}$ along this section we get an auxiliary
diagram
\[
\xymatrix{\mathcal{M}_{b}^{\circ}\ar[d]^{q_{b}^{\circ}}\ar[r]^{\pi_{b}^{\circ}} & \mathrm{Bun}_{G}\\{}
[\ast/\underline{G_{b}(E)}]
}
\]
together with a natural open immersion $j_{b}:\mathcal{M}_{b}^{\circ}\to\mathcal{M}_{b}$.
The functor $i_{b\sharp}$ turns out to be given by the formula $i_{b\sharp}=\pi_{b\natural}q_{b}^{\ast}$.
The natural transformation $j_{b\natural}j_{b}^{\ast}\to\mathrm{id}$
then induces a natural transformation $\pi_{b\natural}^{\circ}q_{b}^{\circ\ast}\to\pi_{b\natural}q_{b}^{\ast}$,
and we define $i_{b!}$ as the cofiber of this map.\footnote{This definition may seem surprisingly complicated, but the issue is
that in the $D_{\mathrm{lis}}$ formalism, we only have a priori access
to $\natural$-pushforwards along cohomologically smooth maps. Consequently,
we are forced to construct $i_{b!}$ indirectly since $i_{b}$ itself
is not cohomologically smooth.} From this construction the map $i_{b\sharp}\to i_{b!}$ is tautological.
An easy application of base change shows that $i_{b}^{\ast}i_{b!}\cong\mathrm{id}$,
so by adjunction we get the claimed transformation $i_{b!}\to i_{b\ast}$.
Finally, $i_{b}^{!}$ is defined as the right adjoint of $i_{b!}$.

For our purposes, it will be very convenient to ``renormalize''
these functors. To make the relevant definition, we note that $G_{b}$
is an inner form of a Levi subgroup $M_{\{b\}}\subset G$, namely
the centralizer of the Newton cocharacter $\nu_{b}$. Let $P_{\{b\}}$
be the dynamic parabolic of $\nu_{b}$, and let $\delta_{b}:M_{\{b\}}(E)\to q^{\mathbf{Z}}$
be the usual modulus character. Since $G_{b}$ and $M_{\{b\}}$ have
canonically isomorphic cocenters, we can regard $\delta_{b}$ as a
cocharacter $G_{b}(E)\to q^{\mathbf{Z}}$. Since we have fixed a choice
of $\sqrt{q}$ in our coefficient ring, we can take the square root
to get a character $\delta_{b}^{1/2}:G_{b}(E)\to\Lambda^{\times}$.
\begin{defn}[Renormalized functors]
For $?\in\{\sharp,!,\ast\}$, we define $i_{b?}^{\mathrm{ren}}:D(G_{b}(E),\Lambda)\to D(\mathrm{Bun}_{G},\Lambda)$
by the formula
\[
i_{b?}^{\mathrm{ren}}A=i_{b?}(A\otimes\delta_{b}^{-1/2})[-\left\langle 2\rho_{G},\nu_{b}\right\rangle ].
\]
Similarly, for $?\in\{\ast,!\}$ we define $i_{b}^{?\mathrm{ren}}A=(\delta_{b}^{1/2}\otimes i_{b}^{?}A)[\left\langle 2\rho_{G},\nu_{b}\right\rangle ]$.
\end{defn}

Note that by design, we still have adjunctions $i_{b\sharp}^{\mathrm{ren}}\vdash i_{b}^{\ast\mathrm{ren}}\vdash i_{b\ast}^{\mathrm{ren}}$
and $i_{b!}^{\mathrm{ren}}\vdash i_{b}^{!\mathrm{ren}}$, as well
as natural transformations $i_{b\sharp}^{\mathrm{ren}}\to i_{b!}^{\mathrm{ren}}\to i_{b\ast}^{\mathrm{ren}}$
which induce the identity after applying $i_{b}^{\ast\mathrm{ren}}$.
Note also that when $b$ is basic, we have not changed the functors
at all. The importance of this renormalization will become clear later.
For now let us just remark that the renormalized functors interact
cleanly with duality (Proposition \ref{prop:dualitypush}), have good
(semi)perversity properties (Exercise \ref{exer:perversesemiexact}),
are closely related to geometric Eisenstein series (Remark \ref{rem:Eisensteinpush}),
and preserve \emph{L-}parameters (Theorem \ref{thm:ibcompatibleparameters}).
\begin{prop}
\label{prop:easyvanishing}One has the following a priori vanishing
results:

\emph{i. }$R\mathrm{Hom}(i_{b\sharp}A,i_{b'\sharp}A')=0$ unless $b\preceq b'$,

\emph{ii. }$R\mathrm{Hom}(i_{b!}A,i_{b'!}A')=0$ unless $b'\preceq b$,

\emph{iii. }$R\mathrm{Hom}(i_{b\sharp}A,i_{b'!}A')=0$ unless $b=b'$,

\emph{iv. $R\mathrm{Hom}(i_{b\sharp}A,i_{b'\ast}A')=0$ }unless\emph{
$b'\preceq b$.}\\
The obvious variants hold for renormalized pushforwards.
\end{prop}

\begin{proof}
Parts i., iii. and iv. are exercises using various adjunctions together
with the support properties of $i_{b\ast}$, $i_{b!}$ and $i_{b\sharp}$
mentioned above, and ii. can be deduced from i. by duality.
\end{proof}

\subsubsection{Finiteness and duality}

There are two natural finiteness conditions one can impose on objects
of $D(G_{b}(E),\Lambda)$. On one hand, we can consider the \emph{compact}
\emph{objects} in this category. It turns out that any compact object
is built from basic compact objects, i.e. representations of the form
$\mathrm{ind}_{K}^{G_{b}(E)}\Lambda$ where $K\subset G_{b}(E)$ is
a pro-$p$ compact open subgroup and $\mathrm{ind}_{K}^{G_{b}(E)}$
denotes compact induction, by finitely many shifts, cones and retracts.
It is not hard to see that $D(G_{b}(E),\Lambda)$ is compactly generated. 
\begin{rem}
When $\Lambda$ is a field of characteristic zero, a difficult theorem
of Bernstein asserts that $A\in D(G_{b}(E),\Lambda)$ is compact iff
it has bounded cohomological amplitude and each $H^{n}(A)$ is a finitely
generated representation. As a particular consequence of this theorem,
we note that the standard truncation functors on $D(G_{b}(E),\Lambda)$
preserve compact objects when $\Lambda$ is a field of characteristic
zero. We will use this fact later in the construction of the hadal
t-structure.
\end{rem}

On the other hand, we can consider \emph{ULA objects}, namely those
objects $A\in D(G_{b}(E),\Lambda)$ for which $A^{K}$ is a perfect
complex of $\Lambda$-modules for all pro-$p$ open compacts $K$.
Note that if $\Lambda$ is a field and $A$ has bounded cohomological
amplitude, $A$ is ULA exactly when each $H^{n}(A)$ is an admissible
representation in the usual sense. One should think of the ULA condition
as a very mild variant on admissibility, which is somehow more conceptual.

These two finiteness conditions come with two associated dualities.
On the subcategory $D(G_{b}(E),\Lambda)^{\mathrm{ULA}}$ of ULA objects,
we have the operation of \emph{smooth duality} $\mathbf{D}_{\mathrm{sm}}$
sending $A$ to $R\mathscr{H}\mathrm{om}(A,\Lambda)$, where $R\mathscr{H}\mathrm{om}(-,\Lambda)$
is the internal hom towards the trivial representation in $D(G_{b}(E),\Lambda)$.
This operation makes sense for any $A$, but on the subcategory of
ULA objects it restricts to an involutive anti-equivalence. When $\Lambda$
is a field and $A=\pi$ is an admissible representation concentrated
in degree zero, $\mathbf{D}_{\mathrm{sm}}\pi$ is just the usual smooth
dual $\pi^{\vee}$, i.e. the smooth vectors in the abstract dual $\mathrm{Hom}_{\Lambda}(\pi,\Lambda)$. 

On the subcategory $D(G_{b}(E),\Lambda)^{\omega}$ of compact objects,
we also have Bernstein's \emph{cohomological duality $\mathbf{D}_{\mathrm{coh}}$},
which by definition sends $A$ to the total external derived hom $R\mathrm{Hom}(A,\mathcal{C}_{c}^{\infty}(G_{b}(E),\Lambda))$.
It is a fun exercise to check that $\mathbf{D}_{\mathrm{coh}}$ sends
any basic compact object $\mathrm{ind}_{K}^{G_{b}(E)}\Lambda$ to
itself, but noncanonically so (the isomorphism depending on a choice
of $\Lambda$-valued Haar measure on $G_{b}(E)$), and therefore defines
an involutive anti-equivalence on compact objects.

These finiteness conditions and their associated dualities have perfect
adaptations to sheaves on $\mathrm{Bun}_{G}$. Namely, smooth duality
matches with Verdier duality, which is given by exactly the same formula,
namely $\mathbf{D}_{\mathrm{Verd}}(A)=R\mathscr{H}\mathrm{om}(A,\Lambda)$
where the internal hom is now computed in $D(\mathrm{Bun}_{G},\Lambda)$.
This operation makes sense for all sheaves, and on the subcategory
$D(\mathrm{Bun}_{G},\Lambda)^{\mathrm{ULA}}$ of ULA sheaves it induces
an involutive antiequivalence. We omit the \emph{definition }of ULA
sheaves on $\mathrm{Bun}_{G}$, and simply note that they are (non-definitionally)
characterized by the property that their $\ast$-restriction to each
stratum is ULA in $D(G_{b}(E),\Lambda)$.

Cohomological duality matches with the more exotic operation of \emph{Bernstein-Zelevinsky
duality}, denoted $\mathbf{D}_{\mathrm{BZ}}$. This is only defined
on compact sheaves, where it again yields an involutive anti-equivalence.
This duality is characterized by the formula $R\mathrm{Hom}(\mathbf{D}_{\mathrm{BZ}}(A),B)=\pi_{\natural}(A\otimes B)$
where $\pi:\mathrm{Bun}_{G}\to\ast$ is the structure map and $\pi_{\natural}$
is the left adjoint of $\pi^{\ast}$. Compact sheaves also have a
useful concrete characterization, by the property that their $\ast$-restriction
to each stratum is compact and is identically zero on all but finitely
many strata.

The finiteness conditions and dualities on $\mathrm{Bun}_{G}$ and
on the individual strata are linked as follows.
\begin{prop}
\emph{\label{prop:dualitypush}o. }The functors $i_{b!}$, $i_{b\sharp}$,
and $i_{b}^{\ast}$ preserve compact objects. The functors $i_{b!}$,
$i_{b\ast}$, and $i_{b}^{!}$ preserve ULA objects. Likewise for
the renormalized functors.

\emph{i. }For any $A\in D(G_{b}(E),\Lambda)$ we have $\mathbf{D}_{\mathrm{Verd}}i_{b!}^{\mathrm{ren}}A\simeq i_{b\ast}^{\mathrm{ren}}\mathbf{D}_{\mathrm{sm}}A$,
and if $A$ is ULA we also have $\mathbf{D}_{\mathrm{Verd}}i_{b\ast}^{\mathrm{ren}}A\simeq i_{b!}^{\mathrm{ren}}\mathbf{D}_{\mathrm{sm}}A$.

\emph{ii. }For any $A\in D(G_{b}(E),\Lambda)^{\omega}$ we have $\mathbf{D}_{\mathrm{BZ}}i_{b!}^{\mathrm{ren}}A\simeq i_{b\sharp}^{\mathrm{ren}}\mathbf{D}_{\mathrm{coh}}A$
and $\mathbf{D}_{\mathrm{BZ}}i_{b\sharp}^{\mathrm{ren}}A\simeq i_{b!}^{\mathrm{ren}}\mathbf{D}_{\mathrm{coh}}A$.

\emph{iii. }For any $B\in D(\mathrm{Bun}_{G},\Lambda)$ we have $\mathbf{D}_{\mathrm{sm}}i_{b}^{\ast\mathrm{ren}}B\simeq i_{b}^{!\mathrm{ren}}\mathbf{D}_{\mathrm{Verd}}B$,
and if $B$ is ULA we also have $\mathbf{D}_{\mathrm{sm}}i_{b}^{!\mathrm{ren}}B\simeq i_{b}^{\ast\mathrm{ren}}\mathbf{D}_{\mathrm{Verd}}B$
\end{prop}

Note that the duality compatibilies here would be much uglier to state
with the naive pushforwards.
\begin{proof}[Sketch]
 The key point, via some calesthenics with the appropriate diagram
\[
\xymatrix{[\ast/\underline{G_{b}(E)}]\ar[d]^{s_{b}}\ar[dr]^{i_{b}'}\\
\mathrm{Bun}_{G}^{b}\ar[r]^{i_{b}}\ar[d]^{f_{b}} & \mathrm{Bun}_{G}\ar[dl]^{\pi}\\
\ast
}
\]
is the observation that $i_{b}^{!}\Lambda\simeq f_{b}^{!}\Lambda\simeq\delta_{b}^{-1}[-2\left\langle 2\rho_{G},\nu_{b}\right\rangle ]$.
Getting the shift correct here is not hard, but pinning down the twist
is much harder, and I originally figured this out by a very indirect
argument using \cite[Proposition IX.5.3 and Theorem IX.7.2]{FS}.
Subsequently, Hamann and Imai gave a very nice direct proof which
covers all cases \cite{HI}.
\end{proof}
\begin{rem}
\label{rem:dualitycomparisons}It is instructive to recall that if
$\Lambda$ is a field of characteristic zero and $A=\pi$ is an irreducible
admissible $G(E)$-representation concentrated in degree zero, then
$\pi$ is both compact \emph{and }ULA, so both dualities make sense.\footnote{Again, compactness is not obvious, and follows from the theorem of
Bernstein mentioned earlier.} The smooth dual of $\pi$ is again an irreducible admissible representation,
and the operation of smooth duality is expected to interact very cleanly
with the local Langlands parametrization \cite{KalGen}. The effect
of cohomological duality is much less obvious: it follows from deep
work of Aubert/Bernstein/Schneider-Stuhler/Zelevinsky that $\mathbf{D}_{\mathrm{coh}}(\pi)\simeq\mathrm{Zel}(\pi)[-d_{\pi}]$
for some nonnegative integer $d_{\pi}$ and some irreducible admissible
representation $\mathrm{Zel}(\pi)$. The operation $\pi\mapsto\mathrm{Zel}(\pi)$
is the Aubert-Zelevinsky involution.\footnote{Some authors would say that $\mathrm{Zel}(\pi)^{\vee}$ is the Aubert-Zelevinsky
involution.} The integer $d_{\pi}$ is just the dimension of the component of
the Bernstein variety containing $\pi$, but the representation $\mathrm{Zel}(\pi)$
is mysterious in general, and it is hard to say much about it beyond
the simple observation that it has the same supercuspidal support
as $\pi^{\vee}$.\footnote{It is at least true that $\mathrm{Zel}(\pi)=\pi^{\vee}$ when $\pi$
is supercuspidal.} 

As a sample of what can happen, we recall that $\mathrm{Zel}(-)$
interchanges the trivial representation $\mathbf{1}$ with the Steinberg
representation $\mathrm{St}$. Although $\mathrm{St}$ has the same
semisimple $L$-parameter as the trivial representation, the associated
monodromy operators are very different (in fact maximally different).
In particular, cohomological duality interacts in a somewhat mysterious
way with the local Langlands parametrization. Nevertheless, it will
transpire that the categorical local Langlands equivalence should
intertwine Bernstein-Zelevinsky duality on $\mathrm{Bun}_{G}$ with
(twisted) Grothendieck-Serre duality on the stack of \emph{L}-parameters.
\end{rem}

\begin{xca}
Show that in parts i. and iii. of Proposition \ref{prop:dualitypush},
the ULA condition in the second claim can be dropped. (Hint: Use \cite[Prop. VII.7.7]{FS}.)
\end{xca}

\subsection{t-structures}

First we briefly recall the perverse t-structure, which makes sense
for any coefficient ring $\Lambda$. The existence of this t-structure
has been well-understood by experts for years (see e.g. \cite[Remark I.10.3]{FS}),
although it doesn't seem to be recorded in the literature.
\begin{prop}
There is a perverse t-structure $^{p}D^{\leq0},\,^{p}D^{\geq0}$ on
$D(\mathrm{Bun}_{G},\Lambda)$, with abelian heart denoted $\mathrm{Perv}(\mathrm{Bun}_{G},\Lambda)$,
uniquely characterized by the condition that $A\in D(\mathrm{Bun}_{G},\Lambda)$
lies in $^{p}D^{\leq0}$ (resp. $^{p}D^{\geq0}$) if $i_{b}^{\ast}A$
sits in cohomological degrees $\leq\left\langle 2\rho_{G},\nu_{b}\right\rangle $
for all $b$ (resp. if $i_{b}^{!}A$ sits in cohomological degrees
$\geq\left\langle 2\rho_{G},\nu_{b}\right\rangle $ for all $b$).
When $\Lambda$ is cohomologically regular in the sense that $\mathrm{Perf}(\Lambda)$
is preserved by standard truncation, the perverse truncation functors
preserve the ULA property. When $\Lambda$ is a field, Verdier duality
exchanges $^{p}D_{\mathrm{}}^{\leq0}\cap D^{\mathrm{ULA}}$ and $\,^{p}D^{\geq0}\cap D^{\mathrm{ULA}}$.
\end{prop}

Here it is again cleaner to think in terms of the renormalized functors:
$A\in D(\mathrm{Bun}_{G},\Lambda)$ lies in $^{p}D^{\leq0}$ (resp.
$^{p}D^{\geq0}$) if and only if $i_{b}^{\ast\mathrm{ren}}A$ (resp.
$i_{b}^{!\mathrm{ren}}A$) sits in cohomological degrees $\leq0$
(resp. $\geq0$) for all $b$.
\begin{xca}
\label{exer:perversesemiexact}i. Check that $i_{b!}^{\mathrm{ren}}$
sends $D^{\leq0}(G_{b}(E),\Lambda)$ into $^{p}D^{\leq0}(\mathrm{Bun}_{G},\Lambda)$,
and that $i_{b\ast}^{\mathrm{ren}}$ sends $D^{\geq0}(G_{b}(E),\Lambda)$
into $^{p}D^{\geq0}(\mathrm{Bun}_{G},\Lambda)$. Check that $^{p}D^{\leq0}(\mathrm{Bun}_{G},\Lambda)$
is generated by all $i_{b!}^{\mathrm{ren}}D^{\leq0}(G_{b}(E),\Lambda)$
under extensions and colimits.

ii. Prove that if $\Lambda$ is a field, there is a natural bijection
between irreducible objects $A\in\mathrm{Perv}(\mathrm{Bun}_{G},\Lambda)$
and pairs $(b,\pi)$ where $b\in B(G)$ and $\pi\in\Pi(G_{b})$ is
an irreducible smooth representation, defined by sending a pair $(b,\pi)$
to the intermediate extension sheaf
\[
i_{b!\ast}^{\mathrm{ren}}\pi\overset{\mathrm{def}}{=}\mathrm{im}\left(^{p}H^{0}(i_{b!}^{\mathrm{ren}}\pi)\to\,^{p}H^{0}(i_{b\ast}^{\mathrm{ren}}\pi)\right).
\]
Prove that $\mathbf{D}_{\mathrm{Verd}}(i_{b!\ast}^{\mathrm{ren}}\pi)\simeq i_{b!\ast}^{\mathrm{ren}}\pi^{\vee}$.
\end{xca}

Now we come to the first really new construction in these notes. 

\textbf{Warning. }In what follows, we assume that $\Lambda$ is a
field \emph{of characteristic zero}. However, the following construction
works verbatim for any coefficient ring $\Lambda$ with the property
that the standard truncation functors on $D(G_{b}(E),\Lambda)$ preserve
the subcategory of compact objects for all $b\in B(G)$.\footnote{For any fixed $G$, this condition holds with $\Lambda=\overline{\mathbf{F}_{\ell}}$
or $\overline{\mathbf{Z}_{\ell}}$ for all but finitely many $\ell$.
It would be interesting to explicate this finite set of bad primes
when $G=\mathrm{GL}_{n}$.}

Define $^{h}D^{\leq0}(\mathrm{Bun}_{G},\Lambda)^{\omega}$ to be the
full subcategory of $D(\mathrm{Bun}_{G},\Lambda)^{\omega}$ generated
under finite extensions by objects of the form $i_{b\sharp}^{\mathrm{ren}}A$
with $A\in D^{\leq0}(G_{b}(E),\Lambda)^{\omega}$. Likewise, define
$^{h}D^{\geq0}(\mathrm{Bun}_{G},\Lambda)^{\omega}$ to be the full
subcategory of $D(\mathrm{Bun}_{G},\Lambda)^{\omega}$ generated under
finite extensions by objects of the form $i_{b!}^{\mathrm{ren}}A$
with $A\in D^{\geq0}(G_{b}(E),\Lambda)^{\omega}$.
\begin{thm}
\label{thm:hadal-t-structure}The pair
\[
\left(^{h}D^{\leq0}(\mathrm{Bun}_{G},\Lambda)^{\omega},{}^{h}D^{\geq0}(\mathrm{Bun}_{G},\Lambda)^{\omega}\right)
\]
defines a t-structure on $D(\mathrm{Bun}_{G},\Lambda)^{\omega}$,
called the \emph{hadal}\footnote{Pronouned ``HEY-dull''. See \href{https://en.wikipedia.org/wiki/Hadal_zone}{https://en.wikipedia.org/wiki/Hadal\_zone}.
Unlike perverse sheaves, which propagate downwards from the maximal
points in their support, hadal sheaves begin their life deep in the
Newton strata of $\mathrm{Bun}_{G}$, and then bubble up to the surface
(i.e. the basic $b$'s).}\emph{ t-structure. }We write $^{h}\tau^{\leq n}$ and $^{h}\tau^{\geq n}$
for the truncation functors associated with this t-structure, and
we write $\mathrm{Had}(\mathrm{Bun}_{G},\Lambda)$ for the abelian
category of \emph{hadal sheaves }defined as its heart.
\end{thm}

Of course we also write 
\[
^{h}H^{n}={}^{h}\tau^{\leq n}\circ{}^{h}\tau^{\geq n}\cong{}^{h}\tau^{\geq n}\circ{}^{h}\tau^{\leq n}:D(\mathrm{Bun}_{G},\Lambda)^{\omega}\to\mathrm{Had}(\mathrm{Bun}_{G},\Lambda)
\]
for the $n$th hadal cohomology functor.
\begin{proof}
If $U\subset\mathrm{Bun}_{G}$ is any open substack, set $D(U):=D(U,\Lambda)^{\omega}$
for brevity. Define $^{h}D^{\leq0}(U)$ and $^{h}D^{\geq0}(U)$ inside
$D(U)$ analogously with the definition for $U=\mathrm{Bun}_{G}$,
but only allowing $b\in|U|$ in the specification of the generators.
It is clear that if $j:U\to V$ is any inclusion of open substacks
of $\mathrm{Bun}_{G}$, $j_{!}$ carries $^{h}D^{\leq0}(U)$ fully
faithfully into $^{h}D^{\leq0}(V)$, and likewise for $^{h}D^{\geq0}$.
It is also clear from the definitions that $^{h}D^{\leq0}(U)[1]\subset{}^{h}D^{\leq0}(U)$
and $^{h}D^{\geq0}(U)[-1]\subset{}^{h}D^{\geq0}(U)$. Finally, we
observe that $^{h}D^{\geq1}(U)$ is exactly the right orthogonal of
$^{h}D^{\leq0}(U)$ inside $D(U)$. This is an easy but very enlightening
exercise with the definitions, which we leave to the reader.

It is now clearly enough to prove that if $U\subset\mathrm{Bun}_{G}$
is a \emph{quasicompact }open substack, the pair 
\[
\left(^{h}D^{\leq0}(U),{}^{h}D^{\geq0}(U)\right)
\]
defines a t-structure on $D(U)$. We will prove this by induction
on the maximal length of any chain of specializations inside $|U|$.
When there are no nontrivial such chains, $U$ is a finite disjoint
union of open strata $[\ast/\underline{G_{b}(E)}]$ with $b$ basic.
In this case the result is clear, since for basic $b$
\[
\left(^{h}D^{\leq0}([\ast/\underline{G_{b}(E)}]),{}^{h}D^{\geq0}([\ast/\underline{G_{b}(E)}])\right)
\]
is just the standard t-structure on $D([\ast/\underline{G_{b}(E)}])\cong D(G_{b}(E),\Lambda)^{\omega}$.

We now proceed by induction. Fix a quasicompact open substack $U$,
and let $b\in|U|$ be the closed point in a chain of maximal length.
Let $j:V\to U$ be the inclusion of the open substack with $|V|=|U|\smallsetminus b$,
and let $i_{b}:\mathrm{Bun}_{G}^{b}\to U$ be the inclusion of the
closed substack associated with $b$ as usual. By the induction hypothesis,
we already have access to the hadal t-structure and its truncation
functors $^{h}\tau_{V}^{\leq n}$ and $^{h}\tau_{V}^{\geq n}$ on
$D(V)$. By the observations in the first paragraph of the proof,
our only remaining task is to show that any given $A\in D(U)$ can
be fit into a distinguished triangle
\[
E\to A\to C\overset{[1]}{\to}
\]
with $E\in\,^{h}D^{\leq0}(U)$ and $C\in\,^{h}D^{\geq1}(U)$. For
this, we define objects $B,C\in D(U)$ inductively by requiring that
they sit in distinguished triangles
\[
i_{b\sharp}^{\mathrm{ren}}\tau^{\leq0}i_{b}^{\ast\mathrm{ren}}A\to A\to B\overset{[1]}{\to}
\]
and
\[
j_{!}\,^{h}\tau_{V}^{\leq0}j^{\ast}B\to B\to C\overset{[1]}{\to}.
\]
By the octahedral axiom, we then get an object $E\in D(U)$ together
with a diagram
\[
\xymatrix{i_{b\sharp}^{\mathrm{ren}}\tau^{\leq0}i_{b}^{\ast\mathrm{ren}}A\ar[d]\\
E\ar[d]\ar[r] & A\ar[r] & C\ar[r]^{[1]} & \phantom{}\\
j_{!}\,^{h}\tau_{V}^{\leq0}j^{\ast}B\ar[d]^{[1]}\\
\\
}
\]
where the row and column are both distinguished triangles. It is clear
that $E$ sits in $^{h}D^{\leq0}(U)$, since it is an extension of
two objects in this category, so we just need to check that $C$ sits
in $^{h}D^{\geq1}(U)$, which we've already noted is the right orthogonal
of $^{h}D^{\leq0}(U)$. It is therefore enough to check that $\mathrm{Hom}(i_{b'\sharp}^{\mathrm{ren}}F,C)=0$
for any $b'\in|U|$ and any $F\in D^{\leq0}(G_{b'}(E),\Lambda)^{\omega}$.
To check this we divide into two disjoint cases:

\textbf{Case 1:} $b'=b$. In this case
\begin{align*}
\mathrm{Hom}(i_{b\sharp}^{\mathrm{ren}}F,C) & =\mathrm{Hom}(i_{b\sharp}^{\mathrm{ren}}F,B)\\
 & \simeq\mathrm{Hom}(F,i_{b}^{\ast\mathrm{ren}}B)\\
 & =0
\end{align*}
where the first isomorphism follows from the triangle defining $C$
and the fact that $\mathrm{Hom}(i_{b\sharp}^{\mathrm{ren}}F,j_{!}-)=0$,
and the final vanishing follows from the fact that $i_{b}^{\ast\mathrm{ren}}B\simeq\tau^{\geq1}i_{b}^{\ast\mathrm{ren}}A$
by consideration of the triangle defining $B$.

\textbf{Case 2:} $b'\in|V|$. In this case $i_{b'\sharp}^{\mathrm{ren}}F=j_{!}i_{b'\sharp}^{\mathrm{ren}}F\in j_{!}\,^{h}D^{\leq0}(V)$,
with the evident abuse of notation, so
\[
\mathrm{Hom}_{U}(i_{b'\sharp}^{\mathrm{ren}}F,C)=\mathrm{Hom}_{V}(i_{b'\sharp}^{\mathrm{ren}}F,j^{\ast}C).
\]
But $j^{\ast}C\simeq\,^{h}\tau_{V}^{\geq1}j^{\ast}B$ by consideration
of the triangle defining $C$, and $i_{b'\sharp}^{\mathrm{ren}}F\in\,^{h}D^{\leq0}(V)$,
so we get the desired vanishing.
\end{proof}
Our next goal is Theorem \ref{thm:irreducible-hadal-sheaves}, which
gives an explicit classification of irreducible hadal sheaves. This
requires several preparatory lemmas.
\begin{lem}
\label{lem:hadal-sub-quotient-prep}Fix $b\in B(G)$, and let $\pi$
be any finitely generated smooth $G_{b}(E)$-representation.

\emph{i. }The hadal sheaf $^{h}H^{0}(i_{b\sharp}^{\mathrm{ren}}\pi)$
does not have any nonzero hadal quotient supported on $\mathrm{Bun}_{G}^{\prec b}$.

\emph{ii. }The hadal sheaf $^{h}H^{0}(i_{b!}^{\mathrm{ren}}\pi)$
does not have any nonzero hadal subobject supported on $\mathrm{Bun}_{G}^{\prec b}$.
\end{lem}

\begin{proof}
We prove the first claim; the second is analogous. Let $j:\mathrm{Bun}_{G}^{\prec b}\to\mathrm{Bun}_{G}$
be the evident open immersion. Let $F$ be a hadal sheaf supported
on $\mathrm{Bun}_{G}^{\prec b}$, so $F\cong j_{!}j^{\ast}F$. Since
$i_{b\sharp}^{\mathrm{ren}}$ is right t-exact for the hadal t-structure,
we compute that
\begin{align*}
\mathrm{Hom}({}^{h}H^{0}(i_{b\sharp}^{\mathrm{ren}}\pi),j_{!}j^{\ast}F) & \cong\mathrm{Hom}(i_{b\sharp}^{\mathrm{ren}}\pi,j_{!}j^{\ast}F)\\
 & \cong\mathrm{Hom}(\pi,i_{b}^{\ast\mathrm{ren}}j_{!}j^{\ast}F)\\
 & =0,
\end{align*}
since $i_{b}^{\ast}j_{!}=0$.
\end{proof}
\begin{lem}
\label{lem:hadal-t-exact-stalk-at-special-point}Let $A\in\mathrm{Had}(\mathrm{Bun}_{G},\Lambda)$
be a hadal sheaf, and let $b$ be a maximally special point in the
support of $A$, i.e. we assume that $i_{b}^{\ast}A\not\simeq0$ and
that $i_{b'}^{\ast}A=0$ for all $b\prec b'$ in the Newton partial
order. Then $i_{b}^{\ast\mathrm{ren}}A\in D(G_{b}(E),\Lambda)$ is
concentrated in degree zero.
\end{lem}

\begin{proof}
Since $i_{b}^{\ast\mathrm{ren}}$ is left t-exact with respect to
the hadal t-structure on $\mathrm{Bun}_{G}$ and the standard t-structure
on $D(G_{b}(E),\Lambda)$, we already know that $i_{b}^{\ast\mathrm{ren}}A$
is concentrated in degrees $\geq0$. To see that it is also concentrated
in degrees $\leq0$, pick a quasicompact open substack $U\subset\mathrm{Bun}_{G}$
containing $\mathrm{supp}A$, and such that $b\in U$ is a closed
point in $U$. Write $V=U\smallsetminus\mathrm{Bun}_{G}^{b}$ as in
the proof of Theorem \ref{thm:hadal-t-structure}, whose notation
we will refer to in what follows. As in the proof of Theorem \ref{thm:hadal-t-structure},
we may regard $A$ as a hadal sheaf in $D(U)$. Since $A\simeq\,^{h}\tau_{U}^{\leq0}A$
by assumption, the constructions in that proof give a distinguished
triangle
\[
i_{b\sharp}^{\mathrm{ren}}\tau^{\leq0}i_{b}^{\ast\mathrm{ren}}A\to A\to j_{!}\,^{h}\tau_{V}^{\leq0}j^{\ast}B\overset{[1]}{\to},
\]
upon noting that the object $E$ we constructed there is exactly the
connective truncation of $A$. Then $i_{b}^{\ast\mathrm{ren}}j_{!}=0$
and $i_{b}^{\ast\mathrm{ren}}i_{b\sharp}^{\mathrm{ren}}\cong\mathrm{id}$,
so applying $i_{b}^{\ast\mathrm{ren}}$ to this triangle gives
\[
i_{b}^{\ast\mathrm{ren}}A\simeq\tau^{\leq0}i_{b}^{\ast\mathrm{ren}}A,
\]
as desired.
\end{proof}
\begin{lem}
\label{lem:hadal-H-zero-conservative}Fix any $b\in B(G)$. Then the
functor on finitely generated $G_{b}(E)$-representations given by
$M\mapsto i_{b}^{\ast\mathrm{ren}}\,^{h}H^{0}(i_{b\sharp}^{\mathrm{ren}}M)$
is naturally isomorphic to the identity functor.

In particular, if $M\to N$ is a nonzero map of finitely generated
$G_{b}(E)$-representations, the induced map $\,^{h}H^{0}(i_{b\sharp}^{\mathrm{ren}}M)\to\,^{h}H^{0}(i_{b\sharp}^{\mathrm{ren}}N)$
is necessarily nonzero.
\end{lem}

\begin{proof}
Write $U=\mathrm{Bun}_{G}^{\preceq b}$ and $V=\mathrm{Bun}_{G}^{\prec b}$
as in the proof of Theorem \ref{thm:hadal-t-structure}, whose notation
we will partly reuse. Since $A:=i_{b\sharp}^{\mathrm{ren}}M$ is connective
for the hadal t-structure, $^{h}H^{0}(A)\simeq\,^{h}\tau_{U}^{\geq0}A$.
Rerunning the proof of Theorem \ref{thm:hadal-t-structure} with the
appropriate shifts, we inductively define distinguished triangles
\[
i_{b\sharp}^{\mathrm{ren}}\tau^{\leq-1}i_{b}^{\ast\mathrm{ren}}A\to A\to B\overset{[1]}{\to}
\]
and
\[
j_{!}\,^{h}\tau_{V}^{\leq-1}j^{\ast}B\to B\to{}^{h}H^{0}(A)\overset{[1]}{\to}.
\]
Applying $i_{b}^{\ast\mathrm{ren}}$ to the second triangle we get
$i_{b}^{\ast\mathrm{ren}}{}^{h}H^{0}(A)\simeq i_{b}^{\ast\mathrm{ren}}B$,
so then using this and applying $i_{b}^{\ast\mathrm{ren}}$ to the
first triangle we get a distinguished triangle
\[
i_{b}^{\mathrm{ren}\ast}i_{b\sharp}^{\mathrm{ren}}\tau^{\leq-1}i_{b}^{\ast\mathrm{ren}}A\to i_{b}^{\ast\mathrm{ren}}A\to i_{b}^{\ast\mathrm{ren}}{}^{h}H^{0}(A)\overset{[1]}{\to}.
\]
But $i_{b}^{\ast\mathrm{ren}}A\cong M$, so the first term vanishes
identically, giving the result.
\end{proof}
\begin{thm}
\label{thm:irreducible-hadal-sheaves}For every pair $(b,\pi)$ with
$b\in B(G)$ and $\pi$ an irreducible smooth $G_{b}(E)$-representation,
there is a unique \emph{irreducible} hadal sheaf $\mathscr{G}_{b,\pi}$
characterized by the requirements that\emph{ $\mathrm{supp}\,\mathscr{G}_{b,\pi}\subseteq\mathrm{Bun}_{G}^{\preceq b}$}and
\emph{$i_{b}^{\ast\mathrm{ren}}\mathscr{G}_{b,\pi}\simeq\pi$}, and
given explicitly by the formula
\[
\mathscr{G}_{b,\pi}\simeq\mathrm{im}\left(\,^{h}H^{0}(i_{b\sharp}^{\mathrm{ren}}\pi)\to\,^{h}H^{0}(i_{b!}^{\mathrm{ren}}\pi)\right).
\]
Moreover, every irreducible hadal sheaf arises from a uniquely associated
pair $(b,\pi)$ in this way.
\end{thm}

We will sometimes write $i_{b\sharp!}^{\mathrm{ren}}\pi$ for the
sheaf $\mathscr{G}_{b,\pi}$, in analogy with the notation $i_{b!\ast}^{\mathrm{ren}}$
for intermediate extension of perverse sheaves.
\begin{proof}
Let $A$ be an irreducible hadal sheaf. Pick $b$ a maximally special
point in the support of $A$, so $i_{b}^{\ast\mathrm{ren}}A$ is concentrated
in degree zero by Lemma \ref{lem:hadal-t-exact-stalk-at-special-point}.
Then $i_{b\sharp}^{\mathrm{ren}}i_{b}^{\ast\mathrm{ren}}A$ is connective
for the hadal t-structure, so via the adjunction $i_{b\sharp}^{\mathrm{ren}}i_{b}^{\ast\mathrm{ren}}\to\mathrm{id}$
we get a distinguished triangle
\[
i_{b\sharp}^{\mathrm{ren}}i_{b}^{\ast\mathrm{ren}}A\to A\to K\overset{[1]}{\to}
\]
of connective objects, such that $b\notin\mathrm{supp}K$. Since the
hadal truncation functors are compatible with $!$-extension along
open substacks, taking hadal cohomology gives a long exact sequence
\[
0\to\,^{h}H^{-1}(K)\to\,^{h}H^{0}(i_{b\sharp}^{\mathrm{ren}}i_{b}^{\ast\mathrm{ren}}A)\overset{\alpha}{\to}A\to\,^{h}H^{0}(K)\to0
\]
of hadal sheaves whose outer terms have support disjoint from $b$.
Then $\alpha$ is nonzero (e.g. by noting that it induces an isomorphism
between nonzero objects after applying $i_{b}^{\ast\mathrm{ren}}$),
so the irreducibility hypothesis on $A$ implies that $\alpha$ is
surjective. Since the source of $\alpha$ has support contained in
$\mathrm{Bun}_{G}^{\preceq b}$, we then deduce additionally that
$\mathrm{supp}A\subseteq\mathrm{Bun}_{G}^{\preceq b}$.

Next, we show that the point $b$ is uniquely determined. Let $b'$
be any other maximally special point in the support of $A$. By a
dual version of the argument in the first paragraph of the proof,
we get an injective map
\[
A\overset{\alpha'}{\to}\,^{h}H^{0}(i_{b'!}^{\mathrm{ren}}i_{b'}^{\ast\mathrm{ren}}A).
\]
Composing this with the surjective map $\alpha$ constructed in the
first paragraph, we get a nonzero map
\[
\,^{h}H^{0}(i_{b\sharp}^{\mathrm{ren}}i_{b}^{\ast\mathrm{ren}}A)\overset{\alpha'\circ\alpha}{\longrightarrow}\,^{h}H^{0}(i_{b'!}^{\mathrm{ren}}i_{b'}^{\ast\mathrm{ren}}A).
\]
But $i_{b\sharp}^{\mathrm{ren}}i_{b}^{\ast\mathrm{ren}}A$ is connective
and $i_{b'!}^{\mathrm{ren}}i_{b'}^{\ast\mathrm{ren}}A$ is coconnective,
by two applications of Lemma \ref{lem:hadal-t-exact-stalk-at-special-point},
so this is the same as the datum of a nonzero map
\[
i_{b\sharp}^{\mathrm{ren}}i_{b}^{\ast\mathrm{ren}}A\to i_{b'!}^{\mathrm{ren}}i_{b'}^{\ast\mathrm{ren}}A.
\]
But if such a nonzero map exists, then necessarily $b=b'$ by Proposition
\ref{prop:easyvanishing}.iii.

Next we show that $i_{b}^{\ast\mathrm{ren}}A$, which is concentrated
in degree zero by Lemma \ref{lem:hadal-t-exact-stalk-at-special-point},
is actually an irreducible smooth representation. Let $B\subseteq i_{b}^{\ast\mathrm{ren}}A$
be any subrepresentation. Consider the maps
\[
\,^{h}H^{0}(i_{b\sharp}^{\mathrm{ren}}B)\overset{\beta}{\to}\,^{h}H^{0}(i_{b\sharp}^{\mathrm{ren}}i_{b}^{\ast\mathrm{ren}}A)\overset{\alpha}{\to}A.
\]
By the surjectivity of $\alpha$ and some general nonsense, we get
an exact sequence
\[
0\to\ker\beta\to\ker\alpha\circ\beta\to\ker\alpha\to\mathrm{coker}\,\beta\to\mathrm{coker}\,\alpha\circ\beta\to0.
\]
From the first paragraph of the proof, we already know that $\ker\alpha$
has support contained in $\mathrm{Bun}_{G}^{\prec b}$, so then also
any quotient of $\ker\alpha$ has the same support property. Now,
since $A$ is irreducible, $\mathrm{coker}\,\alpha\circ\beta$ is
either $\simeq0$ or $\simeq A$. If it is $\simeq0$, then 
\[
\mathrm{coker}\beta\simeq\,^{h}H^{0}(i_{b\sharp}^{\mathrm{ren}}(i_{b}^{\ast\mathrm{ren}}A/B))
\]
is also quotient of $\ker\alpha$ and thus is supported in $\mathrm{Bun}_{G}^{\prec b}$,
so it must vanish identically, which then forces $B=i_{b}^{\ast\mathrm{ren}}A$
by Lemma \ref{lem:hadal-H-zero-conservative}. If it is $\simeq A$,
then $\mathrm{im}\beta\subseteq\ker\alpha$, and we get an exact sequence
\[
0\to\ker\beta\to\,^{h}H^{0}(i_{b\sharp}^{\mathrm{ren}}B)\overset{\tau}{\to}\ker\alpha\to\mathrm{coker}\,\beta\to A\to0.
\]
But we already know that $\ker\alpha$ is supported in $\mathrm{Bun}_{G}^{\prec b}$,
while $\,^{h}H^{0}(i_{b\sharp}^{\mathrm{ren}}B)$ cannot have any
quotient with this support property by Lemma \ref{lem:hadal-sub-quotient-prep}.
This forces $\tau=0$ and then $\ker\beta\overset{\sim}{\to}\,^{h}H^{0}(i_{b\sharp}^{\mathrm{ren}}B)$,
so then $\beta$ is the zero map. But then the natural inclusion map
$B\to i_{b}^{\ast\mathrm{ren}}A$ giving rise to $\beta$ is also
the zero map by Lemma \ref{lem:hadal-H-zero-conservative}, and thus
$B\simeq0$. Therefore, either $B=0$ or $B=i_{b}^{\ast\mathrm{ren}}A$,
so $i_{b}^{\ast\mathrm{ren}}A$ is irreducible.

Summarizing our efforts so far, we have produced from the irreducible
hadal sheaf $A$ a canonical pair $(b,\pi)$ as in the statement of
the theorem, with $b$ the unique maximally special point in the support
of $A$, and with $\pi=i_{b}^{\ast\mathrm{ren}}A$ irreducible. To
reconstruct $A$ from this datum, observe that in the course of our
arguments, we obtained maps
\[
\,^{h}H^{0}(i_{b\sharp}^{\mathrm{ren}}\pi)\twoheadrightarrow A\hookrightarrow\,^{h}H^{0}(i_{b!}^{\mathrm{ren}}\pi)
\]
whose composite is the canonical map induced by the natural transformation
$i_{b\sharp}^{\mathrm{ren}}\to i_{b!}^{\mathrm{ren}}$. This implies
that the association
\[
(b,\pi)\mapsto\mathscr{G}_{b,\pi}\overset{\mathrm{def}}{=}\mathrm{im}\left(\,^{h}H^{0}(i_{b\sharp}^{\mathrm{ren}}\pi)\to\,^{h}H^{0}(i_{b!}^{\mathrm{ren}}\pi)\right)
\]
defines an inverse to our recipe for extracting the pair $(b,\pi)$
from $A$. 

Finally, we need to see that for \emph{every} pair $(b,\pi)$, the
sheaf $\mathscr{G}_{b,\pi}$ defined in the previous paragraph is
irreducible. Fix any such pair. Pick any irreducible hadal sheaf $A$
together with a nonzero map $f:A\to\mathscr{G}_{b,\pi}$. By our arguments
so far, we already know that $A\simeq\mathscr{G}_{b',\pi'}$ for some
pair $(b',\pi')$. Composing the maps
\[
\,^{h}H^{0}(i_{b'\sharp}^{\mathrm{ren}}\pi')\twoheadrightarrow\mathscr{G}_{b',\pi'}\simeq A\overset{f}{\to}\mathscr{G}_{b,\pi}\hookrightarrow\,^{h}H^{0}(i_{b!}^{\mathrm{ren}}\pi)
\]
gives a nonzero map $\,^{h}H^{0}(i_{b'\sharp}^{\mathrm{ren}}\pi')\to\,^{h}H^{0}(i_{b!}^{\mathrm{ren}}\pi)$.
Arguing as in the second paragraph of the proof, the datum of such
a nonzero map is equivalent to the datum of a nonzero map $i_{b'\sharp}^{\mathrm{ren}}\pi'\to i_{b!}^{\mathrm{ren}}\pi$.
But if such a nonzero map exists, then necessarily $b=b'$ by Proposition
\ref{prop:easyvanishing}.iii. The various adjunctions easily imply
a general isomorphism $\mathrm{Hom}(i_{b\sharp}^{\mathrm{ren}}A,i_{b!}^{\mathrm{ren}}B)\cong\mathrm{Hom}(A,B)$,
so now we also get a nonzero map $\pi'\to\pi$, and thus an isomorphism
$\pi\simeq\pi'$. Therefore, $\mathscr{G}_{b,\pi}\simeq\mathscr{G}_{b',\pi'}$
is irreducible.
\end{proof}
\begin{rem}
Just as the perverse t-structure interacts well with Verdier duality,
the hadal t-structure should interact cleanly with Bernstein-Zelevinsky
duality. However, the precise statement is subtle, and we refer to
Remark \ref{rem:bzdual-hadal-interaction} for a more detailed explanation
of what we expect in this direction. 

We also remark that one major advantage of the hadal t-structure over
the perverse t-structure is that for any basic $b$, $i_{b!}$ gives
an exact embedding of the abelian category $\mathrm{Rep}_{\mathrm{sm}}(G_{b}(E),\Lambda)_{\mathrm{fg}}$
into $\mathrm{Had}(\mathrm{Bun}_{G},\Lambda)$, and in fact realizes
it as a Serre subcategory of the latter.
\end{rem}

\begin{xca}
Check that the hadal t-structure on $D(\mathrm{Bun}_{G},\Lambda)^{\omega}$
is bounded and nondegenerate.
\end{xca}

\begin{conjecture}
The natural realization functor $D^{b}\mathrm{Had}(\mathrm{Bun}_{G},\Lambda)\to D(\mathrm{Bun}_{G},\Lambda)^{\omega}$
is an equivalence of categories.
\end{conjecture}

\subsection{The spectral side}

Here we give a very brief recollection on the stack of \emph{L-}parameters
and its coarse moduli, mostly just to set notation; we regard this
as the ``easy'' side of the categorical conjecture. For a much more
detailed treatment, see \cite{DHKM}, \cite{Zhu}, or \cite[Chapter VIII]{FS}.

Let $\Lambda$ be any $\mathbf{Z}_{\ell}[\sqrt{q}]$-algebra. We write
$\mathrm{Par}_{G,\Lambda}=Z^{1}(W_{E},\hat{G})_{\Lambda}/\hat{G}$
for the stack of $\ell$-adically continuous \emph{L}-parameters regarded
as an Artin stack locally of finite type over $\mathrm{Spec}\,\Lambda$,
with its tautological map $\tau_{G}:\mathrm{Par}_{G,\Lambda}\to B\hat{G}_{\Lambda}$.
Note that we can define an analogous stack for any linear algebraic
group, but for non-reductive groups the correct object turns out to
be a \emph{derived }Artin stack (see \cite[Section 2.3]{Zhu}). We
will only need this extra generality for parabolic subgroups, in the
definition of spectral Eisenstein series. 

We will primarily be interested in the case when $\Lambda=\overline{\mathbf{Q}_{\ell}}$,
in which case we will drop $\Lambda$ from the notation. When $\Lambda=\overline{\mathbf{Q}_{\ell}}$
we write $X_{G}^{\mathrm{spec}}$ for the coarse quotient of $\mathrm{Par}_{G}$,
with its canonical map $q:\mathrm{Par}_{G}\to X_{G}^{\mathrm{spec}}$.
The $\overline{\mathbf{Q}_{\ell}}$-points of $\mathrm{Par}_{G}$
parametrize isomorphism classes of $\ell$-adically continuous \emph{L-}parameters
$\phi:W_{E}\to\,^{L}G(\overline{\mathbf{Q}_{\ell}})$, which we will
simply call \emph{L}-parameters. If $\phi$ is any \emph{L}-parameter,
we write $S_{\phi}=\mathrm{Cent}_{\hat{G}}(\phi)$ for its centralizer
group, and we write $i_{\phi}:BS_{\phi}\to\mathrm{Par}_{G}$ for the
associated locally closed immersion. Note that by \cite[Lemma 3.1.8]{Zhu},
we can also regard $\mathrm{Par}_{G}$ as a moduli stack of Weil-Deligne
parameters.

The closed points of $X_{G}^{\mathrm{spec}}$ are in canonical bijection
with \emph{semisimple} \emph{L}-parameters $W_{E}\to\,^{L}G(\overline{\mathbf{Q}_{\ell}})$.
Here we say an \emph{L-}parameter is semisimple if it is Frobenius-semisimple
and has open kernel. Note that $X_{G}^{\mathrm{spec}}$ is a disjoint
union of affine varieties over $\mathrm{Spec}\,\overline{\mathbf{Q}_{\ell}}$,
and $\mathcal{O}(X_{G}^{\mathrm{spec}})\cong\mathcal{O}(\mathrm{Par}_{G})$.
Moreover, each component of $X_{G}^{\mathrm{spec}}$ is a quotient
of a torus by a finite group, and in particular is Cohen-Macaulay.
Note that any \emph{L}-parameter has a unique semisimplification,
corresponding to the image of the point $\phi\in\mathrm{Par}_{G}$
along $q$. Conversely, if $\phi$ is a semisimple parameter, the
fiber of $q$ over the associated closed point $x_{\phi}\in X_{G}^{\mathrm{spec}}$
is a moduli space of $L$-parameters with constant semisimplification
$\phi$. Each such fiber contains a unique \emph{closed} $\overline{\mathbf{Q}_{\ell}}$-point,
corresponding to the actual parameter $\phi$. In particular, $q$
induces a bijection from the \emph{closed }$\overline{\mathbf{Q}_{\ell}}$-points
of $\mathrm{Par}_{G}$ onto the closed points of $X_{G}^{\mathrm{spec}}$.

It is instructive to understand the fibers of $q$ more explicitly.
\begin{prop}
Fix a semisimple L-parameter $\phi$. The reduced fiber of $q$ over
the associated closed point $x_{\phi}\in X_{G}^{\mathrm{spec}}$ admits
the explicit presentation
\[
q^{-1}(x_{\phi})^{\mathrm{red}}\simeq\left\{ (u,N)\in\mathcal{U}_{S_{\phi}}\times\mathfrak{g}^{\mathrm{ad}\phi(I_{E})}\mid\mathrm{ad}\phi(\mathrm{Fr})\cdot N=q^{-1}N,\mathrm{ad}u\cdot N=N\right\} /S_{\phi}
\]
where $S_{\phi}$ acts by simultaneous conjugation.
\end{prop}

Here $\mathcal{U}_{S_{\phi}}$ denotes the variety of unipotent elements
in $S_{\phi}$. Note that the closed substack cut out by $u=1$ is
exactly the Vogan variety
\[
V_{\phi}=\left\{ N\in\mathfrak{g}^{\mathrm{ad}\phi(I_{E})}\mid\mathrm{ad}\phi(\mathrm{Fr})\cdot N=q^{-1}N\right\} /S_{\phi}
\]
parametrizing Frobenius-semisimple \emph{L}-parameters with semisimplification
$\phi$. We also note that the associated closed immersion $V_{\phi}\to q^{-1}(x_{\phi})^{\mathrm{red}}$
has a natural retraction, given by forgetting $u$. In some cases,
e.g. when $\phi(\mathrm{Fr})$ is regular semisimple, the Vogan variety
is the entire fiber. At the other extreme, if $\phi$ is the trivial
\emph{L}-parameter (so $S_{\phi}=\hat{G}$), then only $N=0$ can
occur, but \emph{any} $u$ can occur, and the fiber is the entire
quotient $\mathcal{U}_{\hat{G}}/\hat{G}$. In general, the geometry
of the fiber involves variation of both $u$ and $N$. Note that in
the ``classical'' local Langlands correspondence, only Frobenius-semisimple
\emph{L}-parameters are relevant.
\begin{xca}[Hellmann]
 Take $G=\mathrm{GL}_{4}$, and let $\phi$ be the semsimple parameter
which is trivial on inertia and with $\phi(\mathrm{Fr})=\mathrm{diag}(1,q,q,q^{2})$.
Explicate the finite topological space $|q^{-1}(x_{\phi})|$ as a
set, and draw all of the nontrivial specializations within it.
\end{xca}

\subsection{Categorical conjecture}

We can now start to put the two sides together. For simplicitly, let
$L$ be an algebraic extension of $\mathbf{Q}_{\ell}(\sqrt{q})$,
and let $\Lambda\in\left\{ L,\mathcal{O}_{L}\right\} $. In all that
follows, we assume either that $\Lambda=L$ or that $\ell$ is a very
good prime for $G$ in the sense of \cite{FS}.

The essential carriers of information in Fargues-Scholze are the Hecke
operators, which we briefly recall (see \cite[Section IX.2]{FS}).
\begin{thm}
For any $V\in\mathrm{Rep}_{\Lambda}(^{L}G)$, there is a naturally
associated functor
\[
T_{V}:D(\mathrm{Bun}_{G},\Lambda)\to D(\mathrm{Bun}_{G},\Lambda)^{BW_{E}}
\]
where $D(\mathrm{Bun}_{G},\Lambda)^{BW_{E}}$ denotes the appropriate
category of $W_{E}$-equivariant objects in $D(\mathrm{Bun}_{G},\Lambda)$.
More generally, for any finite set $I$ and any $V\in\mathrm{Rep}_{\Lambda}((^{L}G)^{I})$,
there is a naturally associated functor
\[
T_{V}:D(\mathrm{Bun}_{G},\Lambda)\to D(\mathrm{Bun}_{G},\Lambda)^{BW_{E}^{I}}.
\]
Composing with the forgetful functor to $D(\mathrm{Bun}_{G},\Lambda)$,
the induced endofunctor
\[
T_{V}:D(\mathrm{Bun}_{G},\Lambda)\to D(\mathrm{Bun}_{G},\Lambda)
\]
depends only on the restriction of $V$ to the diagonally embedded
copy of $\hat{G}$ in $^{L}G^{I}$. As an endofunctor of $D(\mathrm{Bun}_{G},\Lambda)$,
$T_{V}$ preserves compact and ULA objects, and has left and right
adjoint given by $T_{V^{\vee}}$.
\end{thm}

The first key construction in Fargues-Scholze linking the the spectral
and automorphic worlds is a natural map
\[
\mathcal{O}(\mathrm{Par}_{G,\Lambda})\to\mathfrak{Z}(D(\mathrm{Bun}_{G},\Lambda))
\]
which they construct using the Hecke operators together with V. Lafforgue's
excursion operator formalism \cite[Theorem IX.5.2]{FS}. Note that
for any $b\in B(G)$, the functor $i_{b!}^{\mathrm{ren}}$ defines
a fully faithful embedding $D(G_{b}(E),\Lambda)\to D(\mathrm{Bun}_{G},\Lambda)$,
which induces a map in the other direction on Bernstein centers. Post-composing
with the above map, we obtain a canonical map $\mathcal{O}(\mathrm{Par}_{G,\Lambda})\overset{\Psi_{G}^{b}}{\to}\mathfrak{Z}(G_{b}(E),\Lambda)$.\footnote{We write $\mathfrak{Z}(G_{b}(E),\Lambda)$ for the center of the category
of smooth $\Lambda[G_{b}(E)]$-modules. When $\Lambda=\overline{\mathbf{Q}_{\ell}}$,
this is the usual Bernstein center, and we drop $\Lambda$ from the
notation. In that case we also write $X_{G_{b}}$ for the Bernstein
variety, so $\mathcal{O}(X_{G_{b}})\cong\mathfrak{Z}(G_{b}(E))$.} When $b=1$, we simply write $\Psi_{G}$ for this map. The following
is a renormalized version of \cite[Theorem IX.7.2]{FS}.
\begin{thm}
\label{thm:ibcompatibleparameters}For any $G$ and any $b\in B(G)$,
the diagram
\[
\xymatrix{\mathcal{O}(\mathrm{Par}_{G,\Lambda})\ar[d]\ar[r]^{\Psi_{G}^{b}} & \mathfrak{Z}(G_{b}(E),\Lambda)\\
\mathcal{O}(\mathrm{Par}_{G_{b},\Lambda})\ar[ur]_{\Psi_{G_{b}}}
}
\]
commutes. Here the left vertical arrow is induced by the canonical
finite map $\mathrm{Par}_{G_{b},\Lambda}\to\mathrm{Par}_{G,\Lambda}$
associated with the canonical L-embedding $^{L}G_{b}\to\,^{L}G$.
\end{thm}

Note that for non-basic $b$, our definition of $\Psi_{G}^{b}$ differs
from that of Fargues-Scholze, and is in fact simpler to use, as reflected
in the fact that we have the canonical \emph{L-}embedding appearing
in the previous theorem rather than the twisted embedding used in
Fargues-Scholze.
\begin{xca}
Check that the map $\Psi_{G}^{b}$ is unaltered upon replacing $i_{b!}^{\mathrm{ren}}$
by $i_{b\sharp}^{\mathrm{ren}}$ in its definition.
\end{xca}

When $\Lambda=\overline{\mathbf{Q}_{\ell}}$, it can be helpful to
think more geometrically: the datum of the map $\Psi_{G}$ is equivalent
to the datum of a map of (ind-)varieties $\Psi_{G}^{\mathrm{geom}}:X_{G}\to X_{G}^{\mathrm{spec}}$.
From this perspective, the Fargues-Scholze construction of \emph{L}-parameters
is extremely transparent: any irreducible smooth representation $\pi$
determines a closed point in $X_{G}$, and thus a closed point in
$X_{G}^{\mathrm{spec}}$ via the map $\Psi_{G}^{\mathrm{geom}}$.
But closed points in $X_{G}^{\mathrm{spec}}$ correspond exactly to
semisimple \emph{L-}parameters.

However, Fargues-Scholze go much further than this, and prove the
following substantial upgrade of the above construction.
\begin{thm}
There is a canonical $\Lambda$-linear $\otimes$-action of $\mathrm{Perf}(\mathrm{Par}_{G,\Lambda})$
on $D(\mathrm{Bun}_{G},\Lambda)$ compatible with the action of Hecke
operators and preserving the subcategory $D(\mathrm{Bun}_{G},\Lambda)^{\omega}$
of compact objects.
\end{thm}

We write $C\ast(-)\circlearrowright D(\mathrm{Bun}_{G},\Lambda)$
for the spectral action of $C\in\mathrm{Perf}(\mathrm{Par}_{G,\Lambda})$.
Compatibility with Hecke operators means that the diagram
\[
\xymatrix{ & \mathrm{End}\left(D(\mathrm{Bun}_{G},\Lambda)\right)\\
\mathrm{Rep}_{\Lambda}(^{L}G)\ar[r]\ar[ur]^{V\mapsto T_{V}} & \mathrm{Rep}_{\Lambda}(\hat{G})\ar[u]_{V\mapsto\tau_{G}^{\ast}V\ast(-)}
}
\]
commutes.

Fargues-Scholze then formulate a deep conjectural refinement of these
constructions. To state this, we need to assume that $G$ is quasisplit.
We also choose a Whittaker datum, i.e. a Borel subgroup $B\subset G$
and a generic character $\psi:U(E)\to\Lambda^{\times}$, where $U\subset B$
is the unipotent radical. We will typically use $\psi$ as shorthand
for the choice of Whittaker datum. We may then define the space
\[
W_{\psi}=\mathrm{ind}_{U}^{G}(\psi)
\]
of compactly supported Whittaker functions with coefficients in $\Lambda$.
In other words, $W_{\psi}\subset\mathcal{C}(G(E),\Lambda)$ is the
space of functions such that $f(ug)=\psi(u)f(g)$ for all $u\in U(E)$
and $g\in G(E)$, $f$ is right-invariant by some open compact subgroup
of $G(E)$, and the support of $f$ has compact image in $U(E)\backslash G(E)$.
Note that $W_{\psi}$ is a ``large'' $G(E)$-representation, but
nonetheless has excellent properties (which we will recall in Appendix
A).
\begin{conjecture}
\label{conj:categoricalLLCmain}Let $\Lambda$ be as above, and containing
all $p$-power roots of unity. Then there is a natural $\Lambda$-linear
equivalence of categories
\[
\mathbf{L}_{\mathfrak{\psi}}:D(\mathrm{Bun}_{G},\Lambda)^{\omega}\overset{\sim}{\to}\mathrm{Coh}_{\mathrm{Nilp}}(\mathrm{Par}_{G,\Lambda})
\]
which is compatible with the spectral action, and which (after ind-completion)
sends the Whittaker sheaf $i_{1!}W_{\psi}$ to the structure sheaf
$\mathcal{O}_{\mathrm{Par}_{G,\Lambda}}$.
\end{conjecture}

Compatibility with the spectral action means that we should have $\mathbf{L}_{\psi}(\mathcal{F}\ast A)\simeq\mathcal{F}\otimes\mathbf{L}_{\psi}(A)$
for all $\mathcal{F}\in\mathrm{Perf}(\mathrm{Par}_{G,\Lambda})$ and
$A\in D(\mathrm{Bun}_{G},\Lambda)^{\omega}$. Applying this compatibility
with $\mathcal{F}\in\mathrm{Perf}(\mathrm{Par}_{G,\Lambda})$ and
$A=i_{1!}W_{\psi}$, in combination with the expectation that $\mathbf{L}_{\psi}(i_{1!}W_{\psi})=\mathcal{O}_{\mathrm{Par}_{G},\Lambda}$,
we deduce that necessarily we should have $\mathbf{L}_{\psi}(\mathcal{F}\ast i_{1!}W_{\psi})=\mathcal{F}$
for $\mathcal{F}\in\mathrm{Perf}(\mathrm{Par}_{G,\Lambda})$.\footnote{Fargues likes to advocate the perspective that $\mathcal{F}\mapsto\mathcal{F}\ast i_{1!}W_{\psi}$
is a kind of non-abelian Fourier transform, and $\mathbf{L}_{\psi}^{G}$
should be some kind of ``continuous'' extension of it from $\mathrm{Perf}$
to $\mathrm{Coh}$.} In particular, since $\mathrm{Perf}\cap\mathrm{Coh_{\mathrm{Nilp}}}=\mathrm{Perf^{qc}}$,
we see that $\mathcal{F}\mapsto\mathcal{F}\ast i_{1!}W_{\psi}$ should
map $\mathrm{Perf^{qc}}(\mathrm{Par}_{G,\Lambda})$ towards \emph{compact
}objects in $D(\mathrm{Bun}_{G},\Lambda)$. This is not obvious! When
$\Lambda$ is a field, this is closely related to showing that the
map $\pi\mapsto\varphi_{\pi}$ has finite fibers. We also see that
$\mathcal{F}\mapsto\mathcal{F}\ast i_{1!}W_{\psi}$ should be \emph{fully
faithful }as a functor from $\mathrm{Perf}$ towards $D(\mathrm{Bun}_{G},\Lambda)$,
which again is far from obvious.

What further conditions do we expect $\mathbf{L}_{\psi}$ to satisfy?

\subsubsection*{Compatibility with the central grading}

Using the inclusion $Z(\hat{G})^{\Gamma}\subset\hat{G}$, we get a
decomposition
\[
\mathrm{Coh_{Nilp}}(\mathrm{Par}_{G,\Lambda})\cong\oplus_{\chi\in X^{\ast}(Z(\hat{G})^{\Gamma})}\mathrm{Coh_{Nilp}}(\mathrm{Par}_{G,\Lambda})^{Z(\hat{G})^{\Gamma}=\chi}.
\]
On the other hand, by a classic theorem of Kottwitz we have a canonical
bijection
\begin{align*}
\pi_{0}(\mathrm{Bun}_{G})=\pi_{1}(G)_{\Gamma} & \cong X^{\ast}(Z(\hat{G})^{\Gamma})\\
\alpha & \mapsto\chi_{\alpha}
\end{align*}
which induces a corresponding decomposition
\[
D(\mathrm{Bun}_{G},\Lambda)^{\omega}\cong\oplus_{\alpha\in\pi_{1}(G)_{\Gamma}}D(\mathrm{Bun}_{G}^{\alpha},\Lambda)^{\omega}.
\]
Writing $\mathrm{Bun}_{G}=\coprod_{\alpha\in\pi_{1}(G)_{\Gamma}}\mathrm{Bun}_{G}^{\alpha}$,
we expect that $\mathbf{L}_{\mathfrak{\psi}}$ should restrict to
compatible equivalences
\[
D(\mathrm{Bun}_{G}^{\alpha},\Lambda)^{\omega}\overset{\sim}{\to}\mathrm{Coh_{Nilp}}(\mathrm{Par}_{G,\Lambda})^{Z(\hat{G})^{\Gamma}=\chi_{\alpha}}
\]
for all $\alpha\in\pi_{1}(G)_{\Gamma}$.

\subsubsection*{Compatibility with duality}

We let $\mathbf{D}_{\mathrm{GS}}=R\mathscr{H}\mathrm{om}(-,\omega)$
denote Grothendieck-Serre duality functor on $\mathrm{Coh}(\mathrm{Par}_{G,\Lambda})$.
(We note for later use that $\omega=\mathcal{O}_{\mathrm{Par}_{G,\Lambda}}$
canonically.) Let $c:\mathrm{Par}_{G,\Lambda}\overset{\sim}{\to}\mathrm{Par}_{G,\Lambda}$
be the involution defined by composition with the Chevalley involution
at the level of \emph{L-}parameters. Note that $\mathbf{D}_{\mathrm{GS}}$
and $c^{\ast}$ commute, so \emph{twisted Grothendieck-Serre duality
}
\[
\mathbf{D}_{\mathrm{tw.GS}}=c^{\ast}\mathbf{D}_{\mathrm{GS}}
\]
still defines an involutive anti-equivalence on $\mathrm{Coh}(\mathrm{Par}_{G,\Lambda})$.
The compatibility of categorical local Langlands with duality can
now be formulated as follows.
\begin{conjecture}
There is a natural equivalence of functors
\[
\mathbf{L}_{\psi^{-1}}\circ\mathbf{D}_{\mathrm{BZ}}\simeq\mathbf{D}_{\mathrm{tw.GS}}\circ\mathbf{L}_{\psi}
\]
from $D(\mathrm{Bun}_{G},\Lambda)^{\omega}$ towards $\mathrm{Coh}_{\mathrm{Nilp}}(\mathrm{Par}_{G,\Lambda})$.
\end{conjecture}

Note that this is again a meta-conjecture, since we have not actually
given a candidate for the functor $\mathbf{L}_{\psi}$.

\subsubsection*{Compatibility with Eisenstein series}

To explain this, we need to first formulate our expectations for Eisenstein
series. A much more detailed discussion of these constructions appears
in \cite{HHS}.

\textbf{Expectation.} For any parabolic $P=MU\subset G$, there is
a canonically associated functor
\[
\mathrm{Eis}_{P}=\mathrm{Eis}_{P}^{G}:D(\mathrm{Bun}_{M},\Lambda)\to D_{\mathrm{}}(\mathrm{Bun}_{G},\Lambda)
\]
with the following properties.

1. There is a natural equivalence of functors $\mathrm{Eis}_{P}\circ i_{1!}^{M}\simeq i_{1!}\circ i_{P}^{G}$,
where $i_{1!}$ and $i_{1!}^{M}:D(M(E),\Lambda)\to D(\mathrm{Bun}_{M},\Lambda)$
are the appropriate extension by zero functors, and $i_{P}^{G}$ denotes
the functor of normalized parabolic induction. More generally, if
$b\in B(M)$ is any element whose image in $B(G)$ is basic, we expect
that $\mathrm{Eis}_{P}\circ i_{b!}^{M}\simeq i_{b!}\circ i_{P_{b}}^{G_{b}}$.

2. $\mathrm{Eis}_{P}$ is compatible with composition: for any inclusion
of parabolics $P_{1}=M_{1}U_{1}\subset P_{2}=M_{2}U_{2}$, $P_{1}\cap M_{2}$
is a parabolic in $M_{2}$ with Levi $M_{1}$, and there should be
a natural equivalence $\mathrm{Eis}_{P_{1}}^{G}\simeq\mathrm{Eis}_{P_{2}}^{G}\circ\mathrm{Eis}_{P_{1}\cap M_{2}}^{M_{2}}$.
This equivalence should be compatible with triple composition in the
evident sense.

3. $\mathrm{Eis}_{P}$ is compatible with any extension of scalars
$\Lambda\to\Lambda'$.

4. $\mathrm{Eis}_{P}$ commutes with all direct sums, and preserves
compact objects.

5. $\mathrm{Eis}_{P}$ preserves ULA objects with quasicompact support.

6. When $\Lambda$ is killed by a power of $\ell$ (so $D_{\mathrm{lis}}=D_{\mathrm{\acute{e}t}}$),
$\mathrm{Eis}_{P}$ is the functor $p_{!}(\mathrm{IC}_{\mathrm{Bun}_{P},\mathbf{Z}_{\ell}[\sqrt{q}]}\otimes_{\mathbf{Z}_{\ell}[\sqrt{q}]}q^{\ast}(-))$,
where
\[
\xymatrix{\mathrm{Bun}_{P}\ar[d]^{q}\ar[r]^{p} & \mathrm{Bun}_{G}\\
\mathrm{Bun}_{M}
}
\]
is the usual diagram, and $\mathrm{IC}_{\mathrm{Bun}_{P},\mathbf{Z}_{\ell}[\sqrt{q}]}\in D_{\mathrm{\acute{e}t}}(\mathrm{Bun}_{P},\mathbf{Z}_{\ell}[\sqrt{q}])$
is a certain explicit invertible object (a square root of the dualizing
complex on $\mathrm{Bun}_{P}$, which can be described explicitly
\cite{HI}).

Of course, when $\Lambda$ is a torsion ring, the formula in 6. can
and should be taken as the definition of $\mathrm{Eis}_{P}$. With
this definition, properties 1.-3. are relatively easy, but 4. and
5. seem to lie significantly deeper. The formula in 6. should in fact
be applicable for any coefficient ring, once the sheaf-theoretic machinery
is sufficiently developed.

We now expect the following compatibility with categorical local Langlands.
\begin{conjecture}
\label{conj:eiscategoricalmatch}For all standard parabolics $P=MU\subset G$,
there is an equivalence $\mathrm{Eis}_{P}^{\mathrm{spec}}\circ\mathbf{L}_{\psi_{M}}\simeq\mathbf{L}_{\psi}\circ\mathrm{Eis}_{P}$,
where $\mathrm{Eis}_{P}^{\mathrm{spec}}=p_{\ast}^{\mathrm{spec}}\circ q^{\mathrm{spec}\ast}:\mathrm{Coh_{Nilp}}(\mathrm{Par}_{M,\Lambda})\to\mathrm{Coh_{Nilp}}(\mathrm{Par}_{G,\Lambda})$
is the spectral Eisenstein functor associated with the diagram
\[
\xymatrix{\mathrm{Par}_{P,\Lambda}\ar[d]^{q^{\mathrm{spec}}}\ar[r]^{p^{\mathrm{spec}}} & \mathrm{Par}_{G,\Lambda}\\
\mathrm{Par}_{M,\Lambda}
}
\]
of (derived) Artin stacks. This equivalence is compatible with composition
in the evident sense.
\end{conjecture}

Note: Zhu defines the spectral Eisenstein functor by the formula $p_{\ast}^{\mathrm{spec}}\circ q^{\mathrm{spec}!}$.
However, there is an isomorphism $q^{\mathrm{spec}!}\simeq q^{\mathrm{spec}\ast}$
because $q$ is quasismooth, hence Gorenstein by \cite[Corollary 2.2.7]{AG},
so $q^{\mathrm{spec}!}\mathcal{F}\simeq q^{\mathrm{spec\ast}}\mathcal{F}\otimes q^{\mathrm{spec}!}\mathcal{O}$
\cite[Remark 7.2.7]{Gai}, and one can show that $q^{\mathrm{spec}!}\mathcal{O}\simeq\mathcal{O}$
(see \cite{Zhu}, Remark 2.3.8 and the comments after the proof of
Proposition 2.3.9).
\begin{rem}
\label{rem:Eisensteinpush}The functors $i_{b!}^{\mathrm{ren}}$ and
$i_{b\sharp}^{\mathrm{ren}}$ are very closely related to Eisenstein
series. More precisely, let $b\in B(G)$ be any element, and let $M\subset G$
be the centralizer of its Newton point $\nu_{b}$; we may assume $b\in M(\breve{E})$,
so $M_{b}=G_{b}$. Let $P$ and $\overline{P}$ be the attracting
and repelling dynamic parabolics associated with $\nu_{b}$. It is
then true that (for any coefficient ring) that there are natural isomorphisms
\[
i_{b\sharp}^{\mathrm{ren}}\simeq\mathrm{Eis}_{P}^{G}\circ i_{b!}^{M}
\]
and
\[
i_{b!}^{\mathrm{ren}}\simeq\mathrm{Eis}_{\overline{P}}^{G}\circ i_{b!}^{M}
\]
as functors $D(M_{b}(E),\Lambda)\to D_{\mathrm{}}(\mathrm{Bun}_{G},\Lambda)$.
When $\Lambda$ is a torsion ring, this is an easy exercise (modulo
the identification of $\mathrm{IC}_{\mathrm{Bun}_{P},\mathbf{Z}_{\ell}[\sqrt{q}]}$).
See \cite[Example V.3.4]{FS} for a hint and \cite{HHS} for a full
discussion.
\end{rem}

\subsection{More conjectures}

In this section we illustrate how the categorical conjecture leads
naturally to various additional conjectures.

Let us begin with the following question: How do Eisenstein series
interact with Hecke operators? More precisely, we could ask: for fixed
$P=MU\subset G$ and $V\in\mathrm{Rep}(^{L}G)$, is there an intellegent
way of rewriting the composite functor $T_{V}\mathrm{Eis}_{P}$? This
seems rather difficult at first glance, but we can ask the same question
on the spectral side. Here things become simpler, since the spectral
analogue of $T_{V}$ is tensoring with the tautological vector bundle
$\tau_{G}^{\ast}V$. We are now asking whether the functor $\tau_{G}^{\ast}V\otimes\mathrm{Eis}_{P}^{\mathrm{spec}}$
can be rewritten in an intelligent way. This turns out to have a very
satisfying answer.
\begin{prop}
Choose a finite filtration $0=V_{0}\subset V_{1}\subset\cdots\subset V_{m}=V|\hat{P}_{\Lambda}$
such that the $\hat{U}_{\Lambda}$-action on each graded piece $W_{i}=V_{i}/V_{i-1}$
is trivial, i.e. such that each $W_{i}$ is naturally inflated from
$\mathrm{Rep}(\hat{M}_{\Lambda})$. Then $\tau_{G}^{\ast}V\otimes\mathrm{Eis}_{P}^{\mathrm{spec}}(-)$
admits a corresponding finite filtration with graded pieces $\mathrm{Eis}_{P}^{\mathrm{spec}}(\tau_{M}^{\ast}W_{i}\otimes-)$.
\end{prop}

\begin{proof}
To begin, observe that we have a commutative diagram
\[
\xymatrix{\mathrm{Par}_{M,\Lambda}\ar[d]^{\tau_{M}} & \mathrm{Par}_{P,\Lambda}\ar[d]^{\tau_{P}}\ar[l]_{q^{\mathrm{spec}}}\ar[r]^{p^{\mathrm{spec}}} & \mathrm{Par}_{G,\Lambda}\ar[d]^{\tau_{G}}\\
B\hat{M}_{\Lambda} & B\hat{P}_{\Lambda}\ar[l]\ar[r] & B\hat{G}_{\Lambda}
}
\]
of derived Artin stacks. We can now use the projection formula to
write 
\begin{align*}
\tau_{G}^{\ast}V\otimes\mathrm{Eis}_{P}^{\mathrm{spec}}(-) & =\tau_{G}^{\ast}V\otimes p_{\ast}^{\mathrm{spec}}q^{\mathrm{spec}\ast}(-)\\
 & \simeq p_{\ast}^{\mathrm{spec}}(p^{\mathrm{spec}\ast}\tau_{G}^{\ast}V\otimes q^{\mathrm{spec}\ast}(-))\\
 & \simeq p_{\ast}^{\mathrm{spec}}(\tau_{P}^{\ast}(V|\hat{P}_{\Lambda})\otimes q^{\mathrm{spec}\ast}(-)).
\end{align*}
Then by assumption $\tau_{P}^{\ast}(V|\hat{P}_{\Lambda})$ has a finite
filtration with graded pieces $q^{\mathrm{spec}\ast}\tau_{M}^{\ast}W_{i}$,
so the functor $p_{\ast}^{\mathrm{spec}}(\tau_{P}^{\ast}(V|\hat{P}_{\Lambda})\otimes q^{\mathrm{spec}\ast}(-))$
acquires a finite filtration with graded pieces
\begin{align*}
p_{\ast}^{\mathrm{spec}}(q^{\mathrm{spec}\ast}\tau_{M}^{\ast}W_{i}\otimes q^{\mathrm{spec}\ast}(-)) & \simeq p_{\ast}^{\mathrm{spec}}q^{\mathrm{spec}\ast}(\tau_{M}^{\ast}W_{i}\otimes-)\\
 & =\mathrm{Eis}_{P}^{\mathrm{spec}}(\tau_{M}^{\ast}W_{i}\otimes-),
\end{align*}
using that $q^{\mathrm{spec}\ast}$ is symmetric monoidal.
\end{proof}
But now we can turn this into a conjecture on the automorphic side.
\begin{conjecture}
\label{conj:Eisheckefiltered}Choose a finite filtration $0=V_{0}\subset V_{1}\subset\cdots\subset V_{m}=V|\hat{P}_{\Lambda}$
such that the $\hat{U}_{\Lambda}$-action on each graded piece $W_{i}=V_{i}/V_{i-1}$
is trivial, i.e. such that each $W_{i}$ is naturally inflated from
$\mathrm{Rep}(\hat{M}_{\Lambda})$. Then $T_{V}\mathrm{Eis}_{P}^{\mathrm{}}(-)$
admits a corresponding finite filtration with graded pieces $\mathrm{Eis}_{P}(T_{W_{i}}-)$.
\end{conjecture}

We emphasize that while this conjecture takes place purely on the
automorphic side, it is forced on us by the categorical conjecture
and the previous proposition. This conjecture has been proved in many
cases by Hamann \cite{Ham}, and a similar argument should work in
general when $\Lambda$ is a torsion ring, and for arbitrary coefficients
once the sheaf-theoretic machinery improves.

We can also start combining our predictions in more artful ways. For
instance, recall that we conjectured a natural equivalence $\mathrm{Eis}_{P}^{\mathrm{spec}}\circ\mathbf{L}_{\psi_{M}}\simeq\mathbf{L}_{\psi}\circ\mathrm{Eis}_{P}$.
How does this interact with duality? Applying $\mathbf{D}_{\mathrm{tw.GS}}$
to both sides, we compute
\begin{align*}
\mathbf{L}_{\psi^{-1}}\circ\mathbf{D}_{\mathrm{BZ}}\circ\mathrm{Eis}_{P} & \simeq\mathbf{D}_{\mathrm{tw.GS}}\circ\mathbf{L}_{\psi}\circ\mathrm{Eis}_{P}\\
 & \simeq\mathbf{D}_{\mathrm{tw.GS}}\circ\mathrm{Eis}_{P}^{\mathrm{spec}}\circ\mathbf{L}_{\psi_{M}}\\
 & \overset{!}{\simeq}\mathrm{Eis}_{\overline{P}}^{\mathrm{spec}}\circ\mathbf{D}_{\mathrm{tw.GS}}^{M}\circ\mathbf{L}_{\psi_{M}}\\
 & \overset{}{\simeq}\mathrm{Eis}_{\overline{P}}^{\mathrm{spec}}\circ\mathbf{L}_{\psi_{M}^{-1}}\circ\mathbf{D}_{\mathrm{BZ}}^{M}\\
 & \simeq\mathbf{L}_{\psi^{-1}}\circ\mathrm{Eis}_{\overline{P}}\circ\mathbf{D}_{\mathrm{BZ}}^{M}
\end{align*}
Here we used several times the expected compatibility of the categorical
equivalence with duality, along with the equivalence $\mathrm{Eis}_{P}^{\mathrm{spec}}\circ\mathbf{L}_{\psi_{M}}\simeq\mathbf{L}_{\psi}\circ\mathrm{Eis}_{P}$
and its variant $\mathrm{Eis}_{\overline{P}}^{\mathrm{spec}}\circ\mathbf{L}_{\psi_{M}^{-1}}\simeq\mathbf{L}_{\psi^{-1}}\circ\mathrm{Eis}_{\overline{P}}$.
The only point which requires further analysis is the isomorphism
labelled ``!''. We leave this as an enlightening exercise to the
reader, with the key hint being that the Chevalley involution exchanges
$\hat{P}$ and $\hat{\overline{P}^{\mathrm{}}}$. Anyway, recall now
that $\mathbf{L}_{\psi^{-1}}$ is supposed to be an equivalence of
categories, in which case we may cancel it out from the first and
last expressions. We have thus arrived at the following conjecture,
which again lives purely on the automorphic side!\footnote{Conjecture \ref{conj:Eisdual} has now been resolved in all generality;
see \cite{HHS} for full details, and the independent work \cite{Tak}
for the case of torsion coefficients.}
\begin{conjecture}
\label{conj:Eisdual}For any given parabolic $P$ with opposite $\overline{P}$,
there is a natural equivalence of functors $\mathbf{D}_{\mathrm{BZ}}\circ\mathrm{Eis}_{P}\simeq\mathrm{Eis}_{\overline{P}}\circ\mathbf{D}_{\mathrm{BZ}}^{M}$
from $D(\mathrm{Bun}_{M},\Lambda)^{\omega}$ to $D(\mathrm{Bun}_{G},\Lambda)^{\omega}$.
\end{conjecture}

This is a geometric analogue of the well-known fact that $\mathbf{D}_{\mathrm{coh}}\circ i_{P}^{G}\simeq i_{\overline{P}}^{G}\circ\mathbf{D}_{\mathrm{coh}}^{M}$
with $\mathbf{C}$- or $\overline{\mathbf{Q}_{\ell}}$-coefficients,
which was first observed by Bernstein and is in fact \emph{equivalent}
to his famous second adjointness theorem. With general coefficients,
this isomorphism follows from recent work of Dat-Helm-Kurinczuk-Moss
\cite{DHKM2}.
\begin{xca}
1) Prove that Conjecture \ref{conj:Eisdual} implies second adjointness.

2) (Difficult.) Fix a torsion coefficient ring $\Lambda$ and a parabolic
$P$. Prove that the following two statements are \emph{equivalent}:

i. The functor $\mathrm{Eis}_{P}$ preserves compact objects, and
Conjecture \ref{conj:Eisdual} is true.

ii. There is a natural equivalence of ``constant term'' functors
$q_{\ast}p^{!}\simeq\overline{q}_{!}\overline{p}^{\ast}$. (Here $\overline{p}$
and $\overline{q}$ refer to the obvious maps in the defining diagram
for $\mathrm{Eis}_{\overline{P}}$.)
\end{xca}

Recall the map $\mathcal{O}(X_{G}^{\mathrm{spec}})\cong\mathcal{O}(\mathrm{Par}_{G})\to\mathfrak{Z}(D(\mathrm{Bun}_{G},\overline{\mathbf{Q}_{\ell}}))$
discussed at the beginning of section 1.4. By the very nature of the
categorical center, this induces a canonical and functorial map $\mathcal{O}(X_{G}^{\mathrm{spec}})\to\mathrm{End}(\mathcal{F})$
for any $\mathcal{F}\in D(\mathrm{Bun}_{G},\overline{\mathbf{Q}_{\ell}})$.
\begin{xca}
Show that the categorical conjecture and Conjecture \ref{conj:eiscategoricalmatch}
together imply the following: for any parabolic $P=MU\subset G$ and
any $\mathcal{F}\in D(\mathrm{Bun}_{M},\overline{\mathbf{Q}_{\ell}})$,
the diagram
\[
\xymatrix{\mathrm{End}(\mathcal{F})\ar[r] & \mathrm{End}(\mathrm{Eis}_{P}(\mathcal{F}))\\
\mathcal{O}(X_{M}^{\mathrm{spec}})\ar[u] & \mathcal{O}(X_{G}^{\mathrm{spec}})\ar[u]\ar[l]
}
\]
commutes.
\end{xca}

\subsection{Finiteness conditions and spectral decomposition of sheaves}

Now we set $\Lambda=\overline{\mathbf{Q}_{\ell}}$ and drop it from
the notation. Before continuing our discussion, it is very useful
to analyze how the various finiteness conditions on objects in $D(\mathrm{Bun}_{G})$
interact, and how they interact with ``spectral decomposition''
of sheaves.

We begin with the following observations. For any semisimple parameter
$\phi$, we can formally define the full subcategory $D(\mathrm{Bun}_{G})_{\phi}^{\mathrm{ULA}}\subset D(\mathrm{Bun}_{G})^{\mathrm{ULA}}$
of $\phi$\emph{-local }ULA sheaves spanned by objects $A$ such that
for all $b$ and $n$, every irreducible subquotient of $H^{n}(i_{b}^{\ast\mathrm{ren}}A)$
has Fargues-Scholze parameter $\phi$. It is easy to see that $D(\mathrm{Bun}_{G})_{\phi}^{\mathrm{ULA}}$
is a thick triangulated subcategory stable under Hecke operators,
and one can also prove that it is stable under the perverse truncation
functors. In fact, one can prove that $D(\mathrm{Bun}_{G})_{\phi}^{\mathrm{ULA}}$
is canonically a direct factor of $D(\mathrm{Bun}_{G})^{\mathrm{ULA}}$,
in the sense that any ULA sheaf $A$ has a canonical and functorial
decomposition $A\cong A_{\phi}\oplus A^{\phi}$ where $A_{\phi}\in D(\mathrm{Bun}_{G})_{\phi}^{\mathrm{ULA}}$
and 
\[
\mathrm{Hom}(B,A^{\phi})=\mathrm{Hom}(A^{\phi},B)=0
\]
for all $B\in D(\mathrm{Bun}_{G})_{\phi}^{\mathrm{ULA}}$ (see \cite{H}
for a detailed statement and proof). One can also prove that Verdier
duality induces an involutive anti-equivalence
\[
\mathbf{D}_{\mathrm{Verd}}:D(\mathrm{Bun}_{G})_{\phi}^{\mathrm{ULA}}\overset{\sim}{\to}D(\mathrm{Bun}_{G})_{\phi^{\vee}}^{\mathrm{ULA}}.
\]
At the level of perverse sheaves, we get an evident category $\mathrm{Perv}(\mathrm{Bun}_{G})_{\phi}^{\mathrm{ULA}}$
of $\phi$-local perverse ULA sheaves, which is a direct factor of
$\mathrm{Perv}(\mathrm{Bun}_{G})^{\mathrm{ULA}}$, and Verdier duality
induces an exact anti-equivalence of abelian categories
\[
\mathbf{D}_{\mathrm{Verd}}:\mathrm{Perv}(\mathrm{Bun}_{G})_{\phi}^{\mathrm{ULA}}\overset{\sim}{\to}\mathrm{Perv}(\mathrm{Bun}_{G})_{\phi^{\vee}}^{\mathrm{ULA}}.
\]

\begin{xca}
Check that $\mathrm{Perv}(\mathrm{Bun}_{G})_{\phi}^{\mathrm{ULA}}$
is a Serre subcategory of $\mathrm{Perv}(\mathrm{Bun}_{G})^{\mathrm{ULA}}$.
\end{xca}

It will also be very useful to consider the following more restrictive
finiteness condition.
\begin{defn}
\label{def:finite-sheaves}A sheaf $A\in D(\mathrm{Bun}_{G})$ is
\emph{finite }if it is both compact and ULA. Equivalently, $A$ is
finite if it has quasicompact support and $\oplus_{n}H^{n}(i_{b}^{\ast}A)$
is a finite length $G_{b}(E)$-representation for every $b$.
\end{defn}

Finite sheaves clearly form a thick triangulated subcategory of $D(\mathrm{Bun}_{G})$,
which we denote $D(\mathrm{Bun}_{G})_{\mathrm{fin}}$. The name is
meant to suggest that such objects have finite length in some sense.
Indeed, one easily checks that $D(\mathrm{Bun}_{G})_{\mathrm{fin}}$
is the thick triangulated subcategory of $D(\mathrm{Bun}_{G})$ consisting
of sheaves which can be obtained from objects of the form $i_{b!}^{\mathrm{ren}}\pi$
(with $\pi$ any irreducible $G_{b}(E)$-representation) via finitely
many shifts, cones and retracts. This category is stable under Hecke
operators, since Hecke operators preserve compactness and ULAness
separately. We also note that for any two finite sheaves $A,B$, $R\mathrm{Hom}(A,B)$
is a perfect complex of $\overline{\mathbf{Q}_{\ell}}$-vector spaces,
and in particular $\mathrm{End}(A)$ is an Artinian $\overline{\mathbf{Q}_{\ell}}$-algebra.
This is a special case of the more general fact that $R\mathrm{Hom}(A,B)$
is perfect whenever $A$ is compact and $B$ is ULA, which follows
from \cite[ Prop. VII.7.4 and Prop. VII.7.9]{FS}.

\textbf{Warning.} Finite sheaves are \emph{not }stable under Verdier
duality, except when $G$ is a torus. Indeed, for non-toral groups
it is easy to see that $i_{1!}\overline{\mathbf{Q}_{\ell}}$ is a
finite sheaf whose Verdier dual $i_{1\ast}\overline{\mathbf{Q}_{\ell}}$
does not have quasicompact support. However, we have the following
important result.
\begin{thm}
\label{thm:finiteBZduality}If $\pi\in\Pi(G_{b})$ is irreducible,
then $i_{b\sharp}^{\mathrm{ren}}\pi$ is finite. In particular, finite
sheaves are stable under Bernstein-Zelevinsky duality. 
\end{thm}

\begin{proof}
By Proposition \ref{prop:dualitypush} and an easy induction on support,
it is enough to prove the first statement. This immediately reduces
to proving that $i_{b'}^{\ast\mathrm{ren}}i_{b\sharp}^{\mathrm{ren}}\pi$
is compact and admissible for any $b'\preceq b$. Compactness is clear,
since both $i_{b\sharp}^{\mathrm{ren}}$ and $i_{b'}^{\ast\mathrm{ren}}$
preserve compactness. For admissibility, we apply the following criterion:
a compact object $B\in D(G_{b'}(E),\overline{\mathbf{Q}_{\ell}})$
is admissible iff its support in the Bernstein variety $X_{G_{b'}}$
is zero-dimensional. The ``only if'' direction is clear; for the
``if'' direction, note that any compact $B$ is admissible ``over
the Bernstein center'' \cite{Ber}, so true admissibility follows
if the action of the Bernstein center factors over an Artinian quotient.

To apply this criterion, we note that the composite map 
\[
c:X_{G_{b'}}\overset{\Psi_{G_{b'}}^{\mathrm{geom}}}{\longrightarrow}X_{G_{b'}}^{\mathrm{spec}}\to X_{G}^{\mathrm{spec}}
\]
has discrete fibers, and in fact it is a finite morphism after restricting
to any connected component of the source. This is clear for the second
map (see \cite{DHKM2} for a more general statement which also works
integrally), and for the first map it is Lemma \ref{lem:psihasdiscretefibers}
below. Writing $x_{\pi}\in X_{G}^{\mathrm{spec}}$ for the closed
point corresponding to the semisimple \emph{L-}parameter of $\pi$,
we conclude by observing that the support of $i_{b'}^{\ast\mathrm{ren}}i_{b\sharp}^{\mathrm{ren}}\pi$
is contained (set-theoretically) in $c^{-1}(x_{\pi})$, which follows
from Theorem \ref{thm:ibcompatibleparameters}. The admissibility
criterion then applies, giving the desired result.
\end{proof}
\begin{lem}
\label{lem:psihasdiscretefibers}For any $G$, the Fargues-Scholze
map $\Psi_{G}^{\mathrm{geom}}:X_{G}\to X_{G}^{\mathrm{spec}}$ has
discrete fibers, and in fact is a finite morphism after restricting
to any connected component of $X_{G}$.
\end{lem}

\begin{proof}
This can be checked one component at a time in $X_{G}$. Then the
usual explicit description of individual Bernstein components (in
terms of a fixed pair $[M,\sigma]$ with $\sigma$ supercuspidal)
plus compatibility of everything with parabolic induction and twisting
reduces us to the special case of a cuspidal component $D$. In this
case, one simply uses the fact that the map from $D$ to the spectral
Bernstein variety is compatible with unramified twists, and unramified
twisting has finite stabilizers in both settings and is transitive
on $D$. (Alternatively - but it's basically the same argument in
different clothing - in the case of a cuspidal component one can reduce
the claim for $G$ to the analogous claims for $G^{\mathrm{der}}$
and $Z_{G}^{\circ}$ separately, using compatibility with products
and central isogenies. But for $G^{\mathrm{der}}$ it is trivial,
since cuspidal components of semisimple groups are zero-dimensional,
and for $Z_{G}^{\circ}$ it is clear since one has $X_{T}=X_{T}^{\mathrm{spec}}$
for $T$ any torus.)
\end{proof}
\begin{rem}
I do not know how to prove this theorem with $\overline{\mathbf{F}_{\ell}}$-coefficients.
The problem is that I have no idea how to prove Lemma \ref{lem:psihasdiscretefibers}
with $\overline{\mathbf{F}_{\ell}}$-coefficients in general, although
the same argument as written above works for $\ell$ outside an explicit
finite set of bad primes.
\end{rem}

\begin{xca}
Prove that for any compact objects $A,B\in D(\mathrm{Bun}_{G})^{\omega}$,
the Fargues-Scholze map $\mathcal{O}(\mathrm{Par}_{G})\to\mathfrak{Z}(D(\mathrm{Bun}_{G}))$
naturally makes $\mathrm{Hom}(A,B)$ into a \emph{finitely generated}
$\mathcal{O}(\mathrm{Par}_{G})$-module. Hint: Reduce to the case
where $A,B$ are $!$- or $\sharp$-pushforwards from individual strata,
then use Lemma \ref{lem:psihasdiscretefibers} and Theorem \ref{thm:ibcompatibleparameters}
together with Bernstein's basic finiteness theorems.
\end{xca}

Now suppose $A$ is a finite sheaf. Then $A$ is ULA, so for any semisimple
parameter we have a decomposition $A\cong A_{\phi}\oplus A^{\phi}$
as before, where $A_{\phi}$ and $A^{\phi}$ are both finite. Since
$A$ is finite, it is easy to see that $A_{\phi}=0$ for all but finitely
many $\phi$, so we get a canonical direct sum decomposition
\[
A\cong\bigoplus_{\phi}A_{\phi}
\]
in $D(\mathrm{Bun}_{G})_{\mathrm{fin}}$ where only finitely many
summands are nonzero. By the functoriality of this decomposition,
we even get a canonical direct sum decomposition of categories
\[
D(\mathrm{Bun}_{G})_{\mathrm{fin}}\cong\bigoplus_{\phi}D(\mathrm{Bun}_{G})_{\mathrm{fin},\phi}
\]
where of course we set $D(\mathrm{Bun}_{G})_{\mathrm{fin},\phi}\overset{\mathrm{def}}{=}D(\mathrm{Bun}_{G})_{\mathrm{fin}}\cap D(\mathrm{Bun}_{G})_{\phi}^{\mathrm{ULA}}$.
\begin{xca}
Check that Bernstein-Zelevinsky duality induces an involutive anti-equivalence
\[
\mathbf{D}_{\mathrm{BZ}}:D(\mathrm{Bun}_{G})_{\mathrm{fin},\phi}\overset{\sim}{\to}D(\mathrm{Bun}_{G})_{\mathrm{fin},\phi^{\vee}}.
\]
\end{xca}

Finally, recall that have constructed the hadal t-structure (Theorem
\ref{thm:hadal-t-structure}), which is a certain t-structure on $D(\mathrm{Bun}_{G})^{\omega}$
with heart denoted $\mathrm{Had}(\mathrm{Bun}_{G})$. We then write
\[
\mathrm{Had}(\mathrm{Bun}_{G})_{\mathrm{fin}}\overset{\mathrm{def}}{=}\mathrm{Had}(\mathrm{Bun}_{G})\cap D(\mathrm{Bun}_{G})_{\mathrm{fin}}
\]
for the subcategory of finite hadal sheaves. Using Theorem \ref{thm:finiteBZduality},
one can check that the hadal truncation functors preserve $D(\mathrm{Bun}_{G})_{\mathrm{fin}}$,
so the pair
\[
\left(^{h}D^{\leq0}(\mathrm{Bun}_{G},\Lambda)^{\omega}\cap D(\mathrm{Bun}_{G})_{\mathrm{fin}},{}^{h}D^{\geq0}(\mathrm{Bun}_{G},\Lambda)^{\omega}\cap D(\mathrm{Bun}_{G})_{\mathrm{fin}}\right)
\]
defines a t-structure on $D(\mathrm{Bun}_{G})_{\mathrm{fin}}$ with
heart $\mathrm{Had}(\mathrm{Bun}_{G})_{\mathrm{fin}}$. Finally, by
the functoriality of the direct sum decompositions above, we get a
canonical decomposition
\[
\mathrm{Had}(\mathrm{Bun}_{G})_{\mathrm{fin}}\cong\bigoplus_{\phi}\mathrm{Had}(\mathrm{Bun}_{G})_{\mathrm{fin},\phi}.
\]

\begin{xca}
i. Check that $\mathrm{Had}(\mathrm{Bun}_{G})_{\mathrm{fin}}$ is
a Serre subcategory of $\mathrm{Had}(\mathrm{Bun}_{G})$.

ii. Check that every object in $\mathrm{Had}(\mathrm{Bun}_{G})_{\mathrm{fin}}$
is of finite length.
\end{xca}

\textbf{Warning. }When $G$ is not a torus, the interplay between
these finiteness conditions can be very subtle. In particular, we
highlight the following phenomena:

1) The perverse truncation functors do not always preserve the property
of being finite. An explicit example is given in the discussion after
Conjecture \ref{conj:categorical-fin-ssgen-perverse-version}.

2) There are examples of irreducible perverse sheaves which are not
finite sheaves. For instance, the constant sheaf $\overline{\mathbf{Q}_{\ell}}$
on one component of $\mathrm{Bun}_{G}$ has this property. 

3) There are examples of finite sheaves which are perverse, but whose
Jordan-Holder series in the perverse category is infinite. For instance,
the sheaf $i_{b_{n}!}^{\mathrm{ren}}\pi_{n}$ appearing at the end
of section 2.2 has this property.

\subsection{The categorical conjecture over $\overline{\mathbf{Q}_{\ell}}$,
unconditionally}

We continue to fix $\Lambda=\overline{\mathbf{Q}_{\ell}}$ as our
coefficient ring, and omit it from the notation. This leads to several
simplifications in the formulation of the categorical conjecture:

1) the nilpotent singular support condition on $\mathrm{Coh}(\mathrm{Par}_{G})$
is automatic \cite[Prop. VIII.2.11]{FS}, and

2) each connected component of $\mathrm{Par}_{G}$ has the property
that $\mathrm{IndPerf}=\mathrm{QCoh}$ \cite[Corollary 3.22]{BFN}.

Using 2), we can formally upgrade the spectral action to a monoidal
action of $\mathrm{QCoh}(\mathrm{Par}_{G})=\mathrm{Ind}(\mathrm{Perf^{qc}}(\mathrm{Par}_{G}))$
on $D(\mathrm{Bun}_{G})$. Acting on the Whittaker sheaf in particular
yields a functor
\begin{align*}
a_{\psi}:\mathrm{QCoh}(\mathrm{Par}_{G}) & \to D(\mathrm{Bun}_{G})\\
\mathcal{F} & \mapsto\mathcal{F}\ast i_{1!}W_{\psi}
\end{align*}
where we choose to write ``$a$'' for \emph{a}ction. As noted in
the discussion after Conjecture \ref{conj:categoricalLLCmain}, we
expect that $a_{\psi}$ is fully faithful: under the hoped-for equivalence
$D(\mathrm{Bun}_{G})\simeq\mathrm{IndCoh}(\mathrm{Par}_{G})$, $a_{\psi}$
should correspond to the natural fully faithful inclusion $\mathrm{QCoh}\overset{\Xi}{\to}\mathrm{Ind}\mathrm{Coh}$.
As noted earler, we also expect that $a_{\psi}$ carries $\mathrm{Perf^{qc}}(\mathrm{Par}_{G})$
into $D(\mathrm{Bun}_{G})^{\omega}$. This is known unconditionally
for many groups: more precisely, for those $G$ which are \emph{reasonable}
in the sense of Definition \ref{def:reasonablegroups}.

\textbf{Warning. }The category $\mathrm{IndCoh}$ contains two copies
of $\mathrm{Coh}$: the ``native'' copy coming from the tautological
inclusion $\mathcal{C}\subset\mathrm{Ind}(\mathcal{C})$, and a second
``phantom'' copy $\Xi(\mathrm{Coh})$. These are not the same, and
their overlap is actually just $\mathrm{Perf}$. In particular, non-perfect
objects in the ``phantom'' copy are not compact when viewed as objects
of $\mathrm{IndCoh}$. Translating into the above picture, we see
that for $\mathcal{F}$ a coherent complex on $\mathrm{Par}_{G}$
which is not perfect, $a_{\psi}(\mathcal{F})$ should not be compact.
In particular, $a_{\psi}$ cannot be the functor realizing the categorical
local Langlands equivalence.

However, there is a closely related functor which \emph{should }realize
this equivalence.
\begin{prop}
There is a (unique) functor
\[
c_{\psi}:D(\mathrm{Bun}_{G})\to\mathrm{QCoh}(\mathrm{Par}_{G})
\]
such that for all $A\in D(\mathrm{Bun}_{G})$ and all $\mathcal{F}\in\mathrm{QCoh}(\mathrm{Par}_{G})$,
there is a natural isomorphism
\[
R\mathrm{Hom}(i_{1!}W_{\psi},\mathcal{F}\ast A)\simeq R\Gamma(\mathrm{Par}_{G},\mathcal{F}\otimes c_{\psi}(A)).
\]
\end{prop}

The functor $c_{\psi}$ is motivated by classical geometric Langlands,
where the analogous beast is usually called the functor of \emph{enhanced
Whittaker coefficient}, and its construction in our setting is exactly
the same as in classical geometric Langlands (see \cite[Section 10.2]{FR}
and \cite{Ra}). As justification for the name, note that
\begin{align*}
R\Gamma(\mathrm{Par}_{G},c_{\psi}(A)) & \simeq R\mathrm{Hom}(i_{1!}W_{\psi},A)\\
 & =R\mathrm{Hom}(W_{\psi},i_{1}^{\ast}A)
\end{align*}
is exactly the space of Whittaker models of $i_{1}^{\ast}A$. Note
also that $c_{\psi}$ is $\mathrm{QCoh}$-linear, in the sense that
$c_{\psi}(\mathcal{G}\ast A)\simeq\mathcal{G}\otimes c_{\psi}(A)$
for all $A\in D(\mathrm{Bun}_{G})$ and $\mathcal{G}\in\mathrm{QCoh}(\mathrm{Par}_{G})$.
In particular, for all $V\in\mathrm{Rep}(^{L}G)$ we have $c_{\psi}(T_{V}A)\simeq V\otimes c_{\psi}(A)$
and consequently 
\[
R\Gamma(\mathrm{Par}_{G},V\otimes c_{\psi}(A))\simeq R\mathrm{Hom}(i_{1!}W_{\psi},T_{V}A).
\]

\begin{prop}
The functor $c_{\psi}$ is right adjoint to $a_{\psi}$.
\end{prop}

\begin{proof}
For $\mathcal{F}\in\mathrm{Perf^{qc}}(\mathrm{Par}_{G})$ and $A\in D(\mathrm{Bun}_{G})$
we compute
\begin{align*}
\mathrm{Hom}(a_{\psi}(\mathcal{F}),A) & =\mathrm{Hom}(\mathcal{F}\ast i_{1!}W_{\psi},A)\\
 & \cong\mathrm{Hom}(i_{1!}W_{\psi},\mathbf{D}_{\mathrm{GS}}\mathcal{F}\ast A)\\
 & \cong\mathrm{Hom}(\mathcal{O},\mathbf{D}_{\mathrm{GS}}\mathcal{F}\otimes c_{\psi}(A))\\
 & \cong\mathrm{Hom}(\mathcal{F},c_{\psi}(A)).
\end{align*}
The only nontrivial point here is the second line, which follows from
the fact that for any given $\mathcal{F}\in\mathrm{Perf^{qc}}(\mathrm{Par}_{G})$,
the endofunctor $A\mapsto\mathcal{F}\ast A$ on $D(\mathrm{Bun}_{G})$
is both left and right adjoint to the endofunctor $A\mapsto\mathbf{D}_{\mathrm{GS}}\mathcal{F}\ast A$.
Since both sides of this isomorphism convert colimits in $\mathcal{F}$
into limits, it formally extends to an isomorphism valid for any $\mathcal{F}\in\mathrm{Ind(Perf^{qc}})=\mathrm{QCoh}$.
\end{proof}
Since $c_{\psi}$ is right adjoint to $a_{\psi}$, under the hoped-for
equivalence $D(\mathrm{Bun}_{G})\simeq\mathrm{IndCoh}(\mathrm{Par}_{G})$
it should correspond to the natural functor $\mathrm{IndCoh}\overset{\Psi}{\to}\mathrm{QCoh}$
right adjoint to $\Xi$. Now, it is easy to see that $\Psi$ defines
an equivalence from the ``native'' copy of $\mathrm{Coh}$ in IndCoh
onto the usual copy of $\mathrm{Coh}$ inside $\mathrm{QCoh}$.\footnote{This follows, for instance, from the observation that the sequence
of functors $\mathrm{QCoh}\overset{\Xi}{\to}\mathrm{IndCoh}\overset{\Psi}{\to}\mathrm{QCoh}$
is obtained by ind-completing the tautological sequence of functors
$\mathrm{Perf^{qc}}\overset{}{\to}\mathrm{Coh}\overset{}{\to}\mathrm{QCoh}$.} We are thus led to the following conjecture.
\begin{conjecture}
\label{conj:categorical-cpsi-unconditional}The functor $c_{\psi}$
restricts to an equivalence of categories
\[
c_{\psi}:D(\mathrm{Bun}_{G})^{\omega}\overset{\sim}{\to}\mathrm{Coh}(\mathrm{Par}_{G}).
\]
\end{conjecture}

In other words, $c_{\psi}$ should realize the categorical local Langlands
equivalence. Note that this is a precise and unconditional conjecture!
However, very little is obvious here. For instance, implicit in this
conjecture is the expectation that $c_{\psi}$ carries compact objects
in $D(\mathrm{Bun}_{G})$ to coherent complexes, which is already
far from obvious.
\begin{xca}
Prove that $c_{\psi}$ is compatible with the central grading.
\end{xca}

\textbf{Warning. }If Conjecture \ref{conj:categorical-cpsi-unconditional}
is true, the equivalence postulated there can be formally ind-completed
to the desired equivalence
\[
\mathbf{L}_{\psi}:D(\mathrm{Bun}_{G})\overset{\sim}{\to}\mathrm{IndCoh}(\mathrm{Par}_{G}).
\]
However, in general one should be careful to distinguish this functor
from the functor $c_{\psi}$ as initially defined above: $c_{\psi}$
can be recovered from $\mathbf{L}_{\psi}$ by composing with the quotient
functor $\Psi:\mathrm{IndCoh}(\mathrm{Par}_{G})\twoheadrightarrow\mathrm{QCoh}(\mathrm{Par}_{G})$.
Since $\Psi$ is an equivalence on the two copies of $\mathrm{Coh}$
as discussed above, there is no real distinction between $\mathbf{L}_{\psi}$
and $c_{\psi}$ as functors on \emph{compact }sheaves on $\mathrm{Bun}_{G}$,
but on general sheaves they do differ. As a sobering exercise, one
can unconditionally check that if we take $A$ to be the constant
sheaf $\overline{\mathbf{Q}_{\ell}}$ on $\mathrm{Bun}_{G}$, then
$c_{\psi}(A)=0$ for any non-toral $G$. Of course, $A$ is not compact,
so this is no contradiction. When $G=\mathrm{PGL}_{2}$, an explicit
candidate object $\mathcal{F}$ in $\mathrm{IndCoh}(\mathrm{Par}_{G})$
which should match this particular sheaf under the putative functor
$\mathbf{L}_{\psi}$ has been constructed by Bertoloni Meli, and it
is visibly clear from his construction that $\Psi(\mathcal{F})=0$.

In this setting, we can also formulate the duality compatibility unconditionally.\footnote{In \cite{HM}, we will prove that if Conjecture \ref{conj:categorical-cpsi-unconditional}
holds for $\psi$ and $\psi^{-1}$, then Conjecture \ref{conj:dualitycpsi}
\emph{automatically }follows!}
\begin{conjecture}
\label{conj:dualitycpsi}There is a natural equivalence of functors
\[
\mathbf{D}_{\mathrm{tw.GS}}\circ c_{\psi}\simeq c_{\psi^{-1}}\circ\mathbf{D}_{\mathrm{BZ}}
\]
from $D(\mathrm{Bun}_{G})^{\omega}$ towards $\mathrm{Coh}(\mathrm{Par}_{G})$.
\end{conjecture}

Here as before, $c_{\psi^{-1}}$ denotes the enhanced coefficient
functor associated with the dual Whittaker datum $(B,\psi^{-1})$.
Composing with $i_{1!}$ and using $i_{1!}\mathbf{D}_{\mathrm{coh}}\simeq\mathbf{D}_{\mathrm{BZ}}i_{1!}$,
we see that Hellmann's functor $R_{\psi}=c_{\psi}\circ i_{1!}$ should
satisfy $\mathbf{D}_{\mathrm{tw.GS}}\circ R_{\psi}\simeq R_{\psi^{-1}}\circ\mathbf{D}_{\mathrm{coh}}$.

As some evidence for this, we have the following result; a detailed
proof will appear in \cite{HM}. For the notion of a reasonable group,
see Definition \ref{def:reasonablegroups}.
\begin{prop}
\label{prop:actiondualitycompatible}If $G$ is reasonable, there
is a natural equivalence of functors
\[
\mathbf{D}_{\mathrm{BZ}}\circ a_{\psi}\simeq a_{\psi^{-1}}\circ\mathbf{D}_{\mathrm{tw.GS}}
\]
from $\mathrm{Perf^{qc}}(\mathrm{Par}_{G})$ towards $D(\mathrm{Bun}_{G})^{\omega}$.
\end{prop}

\begin{proof}[Sketch]
This follows by combining Theorem \ref{thm:Wpsinice} and Proposition
\ref{prop:reasonable-implications}.i with the following result: For
any $\mathcal{F}\in\mathrm{Perf}(\mathrm{Par}_{G})$ and $A\in D(\mathrm{Bun}_{G})^{\omega}$,
one has $\mathbf{D}_{\mathrm{BZ}}(\mathcal{F}\ast A)\simeq\mathbf{D}_{\mathrm{tw.GS}}(\mathcal{F})\ast\mathbf{D}_{\mathrm{BZ}}(A)$.
This can be deduced from \cite[Theorem VIII.5.1 and Theorem IX.2.2]{FS}.
\end{proof}
\begin{xca}
(Difficult.) Assuming that $G$ is reasonable and $\ell$ is a very
good prime in the sense of \cite{FS}, formulate and prove the appropriate
variant of Proposition \ref{prop:actiondualitycompatible} with $\Lambda=\mathbf{\overline{\mathbf{Z}_{\ell}}}$
or with $\Lambda=\overline{\mathbf{F}_{\ell}}$. (Note that this can
be done \emph{without }imposing some additional condition of being
``$\ell$-integrally reasonable''.)
\end{xca}

We also emphasize that Conjecture \ref{conj:categorical-cpsi-unconditional}
is really the \emph{only }way the categorical equivalence can be realized:
\begin{xca}
Check that there is \emph{at most one }equivalence of categories $\mathbf{L}_{\psi}:D(\mathrm{Bun}_{G})\overset{\sim}{\to}\mathrm{IndCoh}(\mathrm{Par}_{G})$
such that the diagram
\[
\xymatrix{\mathrm{QCoh}(\mathrm{Par}_{G})\ar[r]^{a_{\psi}}\ar[dr]^{\Xi} & D(\mathrm{Bun}_{G})\ar[d]_{\mathbf{L}_{\psi}}^{\wr}\\
 & \mathrm{IndCoh}(\mathrm{Par}_{G})
}
\]
commutes. Show that this equivalence exists\emph{ if and only if}
Conjecture \ref{conj:categorical-cpsi-unconditional} is true, in
which case $\mathbf{L}_{\psi}$ is the ind-completion of the resulting
equivalence $c_{\psi}:D(\mathrm{Bun}_{G})^{\omega}\overset{\sim}{\to}\mathrm{Coh}(\mathrm{Par}_{G})$.
Show that if such an equivalence exists, $a_{\psi}$ is fully faithful. 

(Hint: Pass to right adjoints in the triangle above, and use that
$\Psi$ is an equivalence on the obvious copies of $\mathrm{Coh}$.)
\end{xca}

\subsubsection{A variant with restricted variation}

Recall that we defined the category of finite sheaves $D(\mathrm{Bun}_{G})_{\mathrm{fin}}$.
\begin{conjecture}
\label{conj:categorical-restr-var}The functor $c_{\psi}$ restricts
to an equivalence from $D(\mathrm{Bun}_{G})_{\mathrm{fin}}$ towards
the full subcategory $\mathrm{Coh}(\mathrm{Par}_{G})_{\mathrm{fin}}\subset\mathrm{Coh}(\mathrm{Par}_{G})$
spanned by objects which are supported set-theoretically on finitely
many fibers of the map $q:\mathrm{Par}_{G}\to X_{G}^{\mathrm{spec}}$.
\end{conjecture}

This is reminiscent of (and motivated by) the recent ``restricted
variation'' variant of classical geometric Langlands. The compatibility
of $c_{\psi}$ with the action of the spectral Bernstein center easily
implies that for any $A\in D(\mathrm{Bun}_{G})_{\mathrm{fin}}$, the
quasicoherent complex $c_{\psi}(A)$ does in fact satisfy the desired
support condition. One can be more precise: for any finite sheaf $A$,
the ring $\mathrm{End}(A)$ is Artinian, so the natural map $\mathcal{O}(X_{G}^{\mathrm{spec}})\to\mathrm{End}(A)$
factors over an Artinian quotient. Writing $Z_{A}\subset X_{G}^{\mathrm{spec}}$
for the associated finite-length subscheme, the sheaf $c_{\psi}(A)$
then has support contained in $q^{-1}(Z_{A})$. With a little more
thought, it's not difficult to see that Conjecture \ref{conj:categorical-cpsi-unconditional}
implies Conjecture \ref{conj:categorical-restr-var}. In fact, the
converse implication also holds, in the following form.
\begin{prop}
Suppose that the functor $c_{\psi}$ preserves compact objects and
that Conjecture \ref{conj:categorical-restr-var} is true, with coefficients
in all algebraically closed fields $L/\overline{\mathbf{Q}_{\ell}}$.
Then Conjecture \ref{conj:categorical-cpsi-unconditional} is true.
\end{prop}

\begin{proof}
This follows by a straightforward adaptation of the arguments in \cite[Section 21.4]{AGKRRV}.
\end{proof}
Without the freedom to vary the coefficient field, the question of
whether restricted implies full is harder, and will be addressed in
some cases in \cite{HM}.

We emphasize that the coherence property of $c_{\psi}$ applied to
any finite sheaf does not seem easy. Here is some weak evidence in
its favor.
\begin{prop}
Suppose $A\in D(\mathrm{Bun}_{G})_{\mathrm{fin}}$. Then for any $\mathcal{F}\in\mathrm{Perf}(\mathrm{Par}_{G})$,
$R\Gamma(\mathrm{Par}_{G},\mathcal{F}\otimes c_{\psi}(A))\in\mathrm{Perf}(\overline{\mathbf{Q}_{\ell}})$.
\end{prop}

\begin{proof}[Sketch.]
Reduce to the case where $\mathcal{F}$ is a vector bundle pulled
back from $V\in\mathrm{Rep}(B\hat{G})$, so then
\begin{align*}
R\Gamma(\mathrm{Par}_{G},V\otimes c_{\psi}(A)) & \simeq R\mathrm{Hom}(i_{1!}W_{\psi},T_{V}A)\\
 & \simeq R\mathrm{Hom}(W_{\psi},i_{1}^{\ast}T_{V}A)
\end{align*}
by direct examination of the definitions. Then $A$ is finite and
Hecke operators preserve finite sheaves, so $i_{1}^{\ast}T_{V}A$
has only finitely many nonzero cohomologies, each of finite length,
and hence is supported on a finite union of Bernstein components.
If $e$ is the idempotent projector onto this finite union of Bernstein
components, then $eW_{\psi}$ is compact by a result of Bushnell-Henniart
cited in the discussion before Theorem \ref{thm:Wpsinice}. Then
\[
R\mathrm{Hom}(W_{\psi},i_{1}^{\ast}T_{V}A)\simeq R\mathrm{Hom}(eW_{\psi},i_{1}^{\ast}T_{V}A)
\]
is perfect, as desired.
\end{proof}
Finally, we note that it can be useful to localize even further and
study Conjecture \ref{conj:categorical-restr-var} ``one semisimple
parameter at a time''. More precisely, for any fixed semisimple parameter
$\phi$, define $\mathrm{Coh}(\mathrm{Par}_{G})_{\phi}$ as the full
subcategory of $\mathrm{Coh}(\mathrm{Par}_{G})$ spanned by objects
which are supported set-theoretically on the fiber $q^{-1}(x_{\phi})$.
Clearly we have a direct sum decomposition of categories
\[
\mathrm{Coh}(\mathrm{Par}_{G})_{\mathrm{fin}}=\bigoplus_{\phi}\mathrm{Coh}(\mathrm{Par}_{G})_{\phi}.
\]
Then Conjecture \ref{conj:categorical-restr-var} is obviously equivalent
to the expectation that $c_{\psi}$ induces an equivalence $D(\mathrm{Bun}_{G})_{\mathrm{fin},\phi}\overset{\sim}{\to}\mathrm{Coh}(\mathrm{Par}_{G})_{\phi}$
for every $\phi$.
\begin{xca}
Show that if $G$ is reasonable (Definition \ref{def:reasonablegroups}),
the functor $a_{\psi}$ induces functors
\[
\mathrm{Perf}(\mathrm{Par}_{G})_{\mathrm{fin}}\to D(\mathrm{Bun}_{G})_{\mathrm{fin}}
\]
and

\[
\mathrm{Perf}(\mathrm{Par}_{G})_{\phi}\to D(\mathrm{Bun}_{G})_{\mathrm{fin},\phi}
\]
for every semisimple parameter $\phi$.
\end{xca}

\section{Compatibility with classical local Langlands}

We continue to fix $\Lambda=\overline{\mathbf{Q}_{\ell}}$ as our
coefficient ring, and drop it from the notation. In what sense does
the categorical local Langlands conjecture encode classical local
Langlands? The answer to this question turns out to be surprisingly
complicated. Let us begin with a naive guess for what might be true,
closely following \cite[Remark I.10.3]{FS}.

\textbf{Best Hope. }The categorical equivalence is t-exact with respect
to the perverse t-structure on $D(\mathrm{Bun}_{G})$ and some exotic
perverse coherent t-structure on $\mathrm{Coh}(\mathrm{Par}_{G})$,
and thus induces a bijection between irreducible objects in the hearts.
On the $\mathrm{Bun}_{G}$ side, the irreducibles are indexed by pairs
$(b,\pi)$ where $b\in B(G)$ and $\pi\in\mathrm{Irr}(G_{b}(E))$.
On the $\mathrm{Par}_{G}$ side, the irreducibles are indexed by pairs
$(\phi,\rho)$ where $\phi$ is a Frobenius-semisimple $L$-parameter
and $\rho$ is an irreducible algebraic representation of the centralizer
$S_{\phi}$.

Unfortunately, this ``best hope'' seems slightly too simple, for
several related reasons:
\begin{enumerate}
\item $\mathrm{Par}_{G}$ includes many points corresponding to $L$-parameters
which are \emph{not }Frobenius-semisimple, which play no role in the
classical local Langlands correspondence.
\item The perverse t-structure on $D(\mathrm{Bun}_{G})$ is \emph{not} self-dual
with respect to $\mathbf{D}_{\mathrm{BZ}}$, which however is the
natural duality appearing in the categorical conjecture.
\item The geometry of $\mathrm{Par}_{G}$ seems to preclude any simple direct
definition of the hoped-for t-structure, especially around \emph{L}-parameters
with nontrivial monodromy, since perverse coherence is not (naively)
a reasonable notion around these parameters.
\end{enumerate}
One of our main points is that neverthless, the ``best hope'' should
be true for most $L$-parameters, in a sense we will make precise.
The starting point for this is a remarkable recent result of Bertoloni
Meli-Oi \cite{BMO}, which says that the bijection on irreducibles
predicted by the ``best hope'' is actually true! We will then turn
things around and use their results as a guide to formulate some precise
guesses for how the categorical LLC interacts with the classical LLC.

Here is a brief and impressionistic outline of our current understanding.
We will spend most of this section developing this outline into a
precise collection of conjectures, and giving evidence for them.
\begin{itemize}
\item After localizing around a \emph{generous} \emph{L}-parameter (Definition
\ref{def:generous}), the perverse and hadal t-structures on (appropriately
decorated versions of) $D(\mathrm{Bun}_{G})$ should coincide, and
they should match with the standard t-structure on $\mathrm{Coh}(\mathrm{Par}_{G})$.
Moreover, sheaves on different strata of $\mathrm{Bun}_{G}$ will
not interact with each other, and Hecke operators should be t-exact.
\item After localizing around a \emph{semisimple generic} \emph{L-}parameter
(Definition \ref{prop:ss-gen-equivalent-conditions}), the perverse
t-structure on $D(\mathrm{Bun}_{G})$ should match the standard t-structure
on $\mathrm{Coh}(\mathrm{Par}_{G})$, and the hadal t-structure on
$D(\mathrm{Bun}_{G})$ should match some perverse coherent t-structure
on $\mathrm{Coh}(\mathrm{Par}_{G})$. Sheaves on different strata
of $\mathrm{Bun}_{G}$ will interact with each other in highly nontrivial
ways, but Hecke operators should still be t-exact for both t-structures.
\item Around a generic \emph{L}-parameter, the spectral action should encode
the classical local Langlands correspondence in a precise way, although
we do not yet know how to upgrade this to a t-exact matching of sheaves
as in the best hope. See section 3.1 for a detailed discussion.
\end{itemize}
Beyond the semisimple generic case, we have the following very natural
general question: The hadal t-structure is defined on the whole category
$D(\mathrm{Bun}_{G})^{\omega}$, so Conjecture \ref{conj:categorical-cpsi-unconditional}
implies that it \emph{must }correspond to \emph{some} t-structure
on $\mathrm{Coh}(\mathrm{Par}_{G})$. Can we describe this t-structure
intrinsically?
\begin{rem}
It is of significant interest to extend these ideas to integral and
torsion coefficients. Some precise conjectures in these settings have
been proposed by Hamann and Hamann-Lee (\cite[Section 3]{Ham3} and
\cite[Section 6.2]{HL}). We note in particular that our conditions
of being ``generous'' and ``semisimple generic'' match with the
conditions of being ``Langlands-Shahidi type'' and ``weakly Langlands-Shahidi
type'', respectively, as formulated in \cite[Definition 6.2]{HL}.
However, we emphasize that \cite[Definition 6.2]{HL} also makes sense
with $\overline{\mathbf{F}_{\ell}}$-coefficients, which falls outside
the scope of our discussion. I would also like to point out that I
landed on the notion of generous parameters while trying to understand
\cite{Ham}. Finally, we also refer to \cite{Ham3,HL} for some very
striking global applications of this philosophy.
\end{rem}

\subsection{Generous \emph{L}-parameters}

Let us begin by recalling the $B(G)_{\mathrm{bas}}$ form of the local
Langlands correspondence. Fix $G$ quasisplit and pinned as usual,
and fix a Whittaker datum $\psi$. For any $\phi$, we set $S_{\phi}^{\natural}=S_{\phi}/(S_{\phi}\cap\hat{G}^{\mathrm{der}})^{\circ}$.
This is a disconnected reductive group whose identity component is
a torus, and there is a natural map of algebraic groups $Z(\hat{G})^{\Gamma}\to S_{\phi}^{\natural}$
which turns out to be an isogeny.

The following form of the local Langlands conjecture was formulated
by Kottwitz in unpublished work. To the best of my knowledge it first
appeared in print in \cite{RapICM} (see the discussion preceding
Conjecture 5.1). We also refer to \cite[Conjecture F]{KalN} for a
modern and slightly more precise formulation.
\begin{conjecture}[$B(G)_{\mathrm{bas}}$ local Langlands correspondence]
For each $b\in B(G)_{\mathrm{bas}}$, there is a natural finite-to-one
map $\Pi(G_{b})\to\Phi(G)$. Writing $\Pi_{\phi}(G_{b})$ for the
fiber over $\phi$, there is a natural bijection
\[
\iota_{\psi}:\coprod_{b\in B(G)_{\mathrm{bas}}}\Pi_{\phi}(G_{b})\overset{\sim}{\to}\mathrm{Irr}(S_{\phi}^{\natural})
\]
depending only on the choice of Whittaker datum, such that the diagram
\[
\xymatrix{\underset{b\in B(G)_{\mathrm{bas}}}{\coprod}\Pi_{\phi}(G_{b})\ar[r]^{\,\;\;\;\;\;\;\;\iota_{\psi}}\ar[d] & \mathrm{Irr}(S_{\phi}^{\natural})\ar[d]\\
B(G)_{\mathrm{bas}}\ar[r]^{\kappa_{G}} & X^{\ast}(Z(\hat{G})^{\Gamma})
}
\]
commutes. Here the left vertical arrow is the obvious projection,
the right vertical map is induced by restriction along $Z(\hat{G})^{\Gamma}\to S_{\phi}^{\natural}$,
and $\kappa_{G}$ is the Kottwitz isomorphism.
\end{conjecture}

Of course, the word ``natural'' is doing a lot of heavy lifting
here. However, this conjecture is now known unconditionally for many
groups.
\begin{thm}[Bertoloni Meli-Oi]
\label{thm:BMOllc}Suppose that the $B(G)_{\mathrm{bas}}$ LLC holds
for $G$ and all its standard Levi subgroups. Then for each $b\in B(G)$,
there is a natural finite-to-one map $\Pi(G_{b})\to\Phi(G)$. Writing
$\Pi_{\phi}(G_{b})$ for the fiber over $\phi$, there is a natural
bijection
\[
\iota_{\psi}:\coprod_{b\in B(G)_{\mathrm{}}}\Pi_{\phi}(G_{b})\overset{\sim}{\to}\mathrm{Irr}(S_{\phi})
\]
depending only on the choice of Whittaker datum, such that the diagram
\[
\xymatrix{\underset{b\in B(G)}{\coprod}\Pi_{\phi}(G_{b})\ar[r]^{\;\;\;\;\;\;\iota_{\psi}}\ar[d] & \mathrm{Irr}(S_{\phi})\ar[d]\\
B(G)\ar[r]^{\kappa_{G}} & X^{\ast}(Z(\hat{G})^{\Gamma})
}
\]
commutes. The left vertical map is the obvious projection and the
right vertical map is induced by restriction along $Z(\hat{G})^{\Gamma}\to S_{\phi}$.
\end{thm}

We like to think of this as follows: there is a natural bijection
between pairs $(\phi,\rho)$ where $\phi$ is a Frobenius-semisimple
$L$-parameter and $\rho\in\mathrm{Irr}(S_{\phi})$ is an irreducible
algebraic representation, and pairs $(b,\pi)$ where $b\in B(G)$
and $\pi$ is an irreducible smooth representation of $G_{b}(E)$.
We will sometimes call this the BM-O bijection. We refer the reader
to \cite{BMO} for the details behind this beautiful construction.

\textbf{Warning.} For non-basic $b$, the set $\Pi_{\phi}(G_{b})$
appearing in Theorem \ref{thm:BMOllc} is usually not an \emph{L}-packet.
Rather, it is the (finite) union of the \emph{L}-packets attached
to the parameters $\phi':W_{E}\times\mathrm{SL}_{2}\to\,^{L}G_{b}$
such that $i\circ\phi'$ is conjugate to $\phi$, for $i:\,^{L}G_{b}\to\,^{L}G$
the evident map of $L$-groups. Said differently, the set $\Pi_{\phi}(G_{b})$
is the fiber of the composite map
\[
\Pi(G_{b})\to\Phi(G_{b})\to\Phi(G),
\]
and the second arrow is typically not injective.
\begin{rem}
If $\phi$ is discrete, then $S_{\phi}\cong S_{\phi}^{\natural}$,
and $\Pi_{\phi}(G_{b})$ is empty for all non-basic $b$. In this
case the $B(G)$ LLC contains the same information as the $B(G)_{\mathrm{bas}}$
LLC. In general, the BM-O bijection extends the $B(G)_{\mathrm{bas}}$
correspondence in the evident way, by restricting to $\rho\in\mathrm{Irr}(S_{\phi}^{\natural})\subset\mathrm{Irr}(S_{\phi})$
on the right-hand side and $b\in B(G)_{\mathrm{bas}}$ on the left-hand
side.
\end{rem}

\begin{example}
Take $G=\mathrm{GL}_{2}$, and let $\phi=\chi_{1}\oplus\chi_{2}$
be a sum of two random characters. By local class field theory we
can also think of the $\chi_{i}$'s as smooth characters of $E^{\times}$.
Then $S_{\phi}=\mathbf{G}_{m}^{2}$, so $\mathrm{Irr}(S_{\phi})=\mathbf{Z}^{2}$.
How to attach to an ordered pair $(m,n)\in\mathbf{Z}^{2}$ a pair
$(b,\pi)$? It's easy to guess what to do for $b$: just take the
isocrystal with slopes $m$ and $n$. Now when $m=n$, $G_{b}=G$
and we can take $\pi=i_{B}^{G}(\chi_{1}\boxtimes\chi_{2})$. When
$m\neq n$, $G_{b}=T$ is the usual maximal torus in $G$, and we
take $\pi=\chi_{1}\boxtimes\chi_{2}$ or $\pi=\chi_{2}\boxtimes\chi_{1}$
according to whether $m>n$ or $m<n$. Note that $\Pi_{\phi}(G_{b})$
contains two elements when $b$ is not basic.
\end{example}

\begin{defn}
\label{def:generous}A semisimple $L$-parameter $\phi:W_{E}\to\,^{L}G(\overline{\mathbf{Q}_{\ell}})$
is \emph{generous }if $\mathrm{Par}_{G}\times_{X_{G}^{\mathrm{spec}}}\{x_{\phi}\}\cong BS_{\phi}$.
\end{defn}

The name is meant to suggest on one hand that such parameters are
``generic'', in two different ways - they form an open dense subset
of $X_{G}^{\mathrm{spec}}$, and they are generic in the technical
sense of Proposition \ref{prop:ss-gen-equivalent-conditions}.(2)
- and also that they have nice properties which make the classical
and categorical local Langlands correspondence simpler at these parameters.
\begin{example}
Suppose $G=\mathrm{GL}_{n}$. Then an \emph{L-}parameter is generous
iff it is a direct sum $\phi\simeq\phi_{1}\oplus\cdots\oplus\phi_{d}$
where the $\phi_{i}$'s are pairwise-distinct supercuspidal parameters
with $\sum_{i}\mathrm{dim}\phi_{i}=n$, and with $\phi_{i}\not\simeq\phi_{j}(1)$
for any $i\neq j$. Note that the associated $\mathrm{GL}_{n}(E)$-representation
$\pi$ is generic, and is irreducibly induced from a supercuspidal
representation of a Levi.
\end{example}

\begin{prop}
\emph{1) }A discrete $L$-parameter is generous iff it is supercuspidal. 

\emph{2) }If $\phi$ is generous, then $S_{\phi}^{\circ}$ is a torus.

\emph{3) }If $\phi$ is generous, then $q$ is flat in a neighborhood
of $q^{-1}(x_{\phi})$, and there is a natural \emph{regular} closed
immersion 
\[
i_{\phi}:BS_{\phi}=\mathrm{Par}_{G}\times_{X_{G}^{\mathrm{spec}}}\{x_{\phi}\}\hookrightarrow\mathrm{Par}_{G},
\]
and $\mathrm{Par}_{G}$ is smooth in a neighborhood of $\mathrm{im}i_{\phi}$.
\end{prop}

In particular, if $\phi$ is generous, the pushforward $i_{\phi\ast}$
sends irreducible representations of $S_{\phi}$ towards perfect complexes
on $\mathrm{Par}_{G}$. We can now state the main conjecture relating
classical and categorical LLC at generous parameters.
\begin{conjecture}
\label{conj:generousdeep}Suppose that $(\phi,\rho)$ and $(b,\pi)$
match under the BM-O bijection associated with the Whittaker datum
$\psi$, where $\phi$ is a generous L-parameter. 

\emph{i. }There is an isomorphism 
\[
a_{\psi}(i_{\phi\ast}\rho)\overset{\mathrm{def}}{=}i_{\phi\ast}\rho\ast i_{1!}W_{\psi}\simeq i_{b!}^{\mathrm{ren}}\pi
\]
in $D(\mathrm{Bun}_{G})$. 

\emph{ii. }The natural maps
\[
i_{b\sharp}^{\mathrm{ren}}\pi\overset{\sim}{\to}i_{b!}^{\mathrm{ren}}\pi\overset{\sim}{\to}i_{b\ast}^{\mathrm{ren}}\pi
\]
are isomorphisms.
\end{conjecture}

This immediately suggests another conjecture, which in some cases
can be verified more easily.
\begin{conjecture}
\label{conj:generouseigen}Suppose that $\phi$ is a generous L-parameter.
Then
\[
\mathscr{F}_{\phi}=\underset{b\in B(G),\pi\in\Pi_{\phi}(G_{b})}{\bigoplus}i_{b!}^{\mathrm{ren}}\pi^{\oplus\mathrm{dim}\iota_{\psi}(b,\pi)}
\]
is a Hecke eigensheaf with eigenvalue $\phi$.
\end{conjecture}

This follows immediately from Conjecture \ref{conj:generousdeep},
since that conjecture formally implies an isomorphism 
\[
\mathscr{F}_{\phi}\simeq i_{\phi\ast}\mathcal{O}(S_{\phi})\ast i_{1!}W_{\psi}
\]
where $\mathcal{O}(S_{\phi})$ is the regular representation, and
the right-hand side is a Hecke eigensheaf of the stated type for completely
formal reasons. Note that the second part of Conjecture \ref{conj:generousdeep}
implies additionally that $\mathscr{F}_{\phi}$ is perverse.
\begin{rem}
\label{rem:generous-Hecke}We emphasize that Conjecture \ref{conj:generousdeep}
completely describes how to compute Hecke operators on the atomic
sheaves attached to generous parameters. More precisely, suppose $\phi$
is generous and $(b,\pi)$ matches $(\phi,\rho)$ under the BM-O bijection.
Let $V\in\mathrm{Rep}(^{L}G)$ be any representation, with associated
Hecke operator $T_{V}$. To compute $T_{V}i_{b!}^{\mathrm{ren}}\pi$,
simply decompose $V|_{S_{\phi}}\otimes\rho\simeq\oplus_{j}\rho_{j}^{\oplus m_{j}}$
as a sum of irreducible $S_{\phi}$-representations (with multiplicity).
Let $(b_{j},\pi_{j})$ be the pair matching $(\phi,\rho_{j})$ under
the BM-O bijection. Then Conjecture \ref{conj:generousdeep} predicts
that
\[
T_{V}i_{b!}^{\mathrm{ren}}\pi\simeq\oplus_{j}i_{b_{j}!}^{\mathrm{ren}}\pi_{j}^{\oplus m_{j}}.
\]
\end{rem}

With slightly more effort, one can also describe the Galois action
on $T_{V}i_{b!}^{\mathrm{ren}}$ explicitly. The final outcome is
the formula
\[
T_{V}i_{b!}^{\mathrm{ren}}\pi\simeq\bigoplus_{(b',\pi')}i_{b'!}^{\mathrm{ren}}\pi'\boxtimes\mathrm{Hom}_{S_{\phi}}(\iota_{\psi}(b,\pi)^{\vee}\otimes\iota_{\psi}(b',\pi'),V\circ\phi)
\]
in $D(\mathrm{Bun}_{G})^{BW_{E}}$, where the $W_{E}$-action comes
via the second factor in the evident sense. When $\phi$ is supercuspidal,
this exactly recovers the Kottwitz conjecture.
\begin{rem}
\label{rem:generous-duality-trick}If we know the first part of Conjecture
\ref{conj:generousdeep} for $\phi$ and $\phi^{\vee}$, then in fact
the second part follows. More precisely, suppose that $(\phi,\rho)$
and $(b,\pi)$ match under the BM-O bijection for $\psi$. Then we
expect that $(\phi^{\vee},c\circ\rho^{\vee})$ and $(b,\pi^{\vee})$
match under the BM-O bijection for $\psi^{-1}$ \cite{KalGen}, so
the first part of Conjecture \ref{conj:generousdeep} implies that
\[
\left(i_{\phi^{\vee}\ast}c\circ\rho^{\vee}\right)\ast i_{1!}W_{\psi^{-1}}\simeq i_{b!}^{\mathrm{ren}}\pi^{\vee}.
\]
On the other hand, it is easy to compute that 
\begin{align*}
\mathbf{D}_{\mathrm{tw.GS}}(i_{\phi\ast}\rho)[\dim S_{\phi}] & \simeq c^{\ast}i_{\phi\ast}\rho^{\vee}\\
 & \simeq i_{\phi^{\vee}\ast}c\circ\rho^{\vee},
\end{align*}
so using the compatibility of the spectral action with duality proved
in Proposition \ref{prop:actiondualitycompatible}, we compute that
\begin{align*}
\left(i_{\phi^{\vee}\ast}c\circ\rho^{\vee}\right)\ast i_{1!}W_{\psi^{-1}} & \simeq\mathbf{D}_{\mathrm{BZ}}\left(i_{\phi\ast}\rho\ast i_{1!}W_{\psi}\right)[\dim S_{\phi}]\\
 & \simeq\mathbf{D}_{\mathrm{BZ}}\left(i_{b!}^{\mathrm{ren}}\pi\right)[\dim S_{\phi}]\\
 & \simeq i_{b\sharp}^{\mathrm{ren}}\mathbf{D}_{\mathrm{coh}}(\pi)[\dim S_{\phi}]
\end{align*}
where we used Proposition \ref{prop:dualitypush} in the last line.
Equating these two calculations, we get an isomorphism $i_{b!}^{\mathrm{ren}}\pi^{\vee}\simeq i_{b\sharp}^{\mathrm{ren}}\mathbf{D}_{\mathrm{coh}}(\pi)[\dim S_{\phi}]$.
From this we immediately see (by taking the stalk at $b$) that $\pi^{\vee}\simeq\mathbf{D}_{\mathrm{coh}}(\pi)[\dim S_{\phi}]\simeq\mathrm{Zel}(\pi)$,
and then that $i_{b!}^{\mathrm{ren}}\pi^{\vee}\simeq i_{b\sharp}^{\mathrm{ren}}\pi^{\vee}$.
Applying $\mathbf{D}_{\mathrm{BZ}}$ to this last isomorphism and
using Proposition \ref{prop:dualitypush} again, we get the same isomorphism
for $\pi$. This implies the first isomorphism in the second part
of the conjecture.
\end{rem}

To get the second isomorphism, we argue more generally as follows.
Fix a semisimple $L$-parameter $\phi$, and suppose that $i_{b\sharp}^{\mathrm{ren}}\pi\overset{\sim}{\to}i_{b!}^{\mathrm{ren}}\pi$
is an isomorphism for all $b$ and all $\pi\in\Pi(G_{b})$ with Fargues-Scholze
parameter $\phi$. It then follows that $i_{b!}^{\mathrm{ren}}\pi\overset{\sim}{\to}i_{b\ast}^{\mathrm{ren}}\pi$
is an isomorphism for all $b$ and all $\pi\in\Pi(G_{b})$ with Fargues-Scholze
parameter $\phi$. Indeed, it's clearly enough to prove that $i_{b'}^{\ast\mathrm{ren}}i_{b\ast}^{\mathrm{ren}}\pi=0$
for all $b'\neq b$. Using that $i_{b'}^{\ast\mathrm{ren}}i_{b\ast}^{\mathrm{ren}}\pi$
is admissible and left-bounded, with all of its irreducible constituents
having Fargues-Scholze parameter $\phi$, we easily reduce further
to proving that $R\mathrm{Hom}(\tau,i_{b'}^{\ast\mathrm{ren}}i_{b\ast}^{\mathrm{ren}}\pi)=0$
for all $b'\neq b$ and all $\tau\in\Pi(G_{b'})$ with Fargues-Scholze
parameter $\phi$. But now we compute that
\begin{align*}
R\mathrm{Hom}_{G_{b'}}(\tau,i_{b'}^{\ast\mathrm{ren}}i_{b\ast}^{\mathrm{ren}}\pi) & \cong R\mathrm{Hom}(i_{b'\sharp}^{\mathrm{ren}}\tau,i_{b\ast}^{\mathrm{ren}}\pi)\\
 & \cong R\mathrm{Hom}(i_{b'!}^{\mathrm{ren}}\tau,i_{b\ast}^{\mathrm{ren}}\pi)\\
 & \cong R\mathrm{Hom}_{G_{b}}(i_{b}^{\ast\mathrm{ren}}i_{b'!}^{\mathrm{ren}}\tau,\pi)\\
 & =0
\end{align*}
where the first and third isomorphisms come from the obvious adjunctions,
the second isomorphism follows from our assumptions, and the last
line follows from the trivial vanishing $i_{b}^{\ast\mathrm{ren}}i_{b'!}^{\mathrm{ren}}\tau=0$
for $b'\neq b$.
\begin{rem}
The logic in the last paragraph of the previous remark can also be
reversed, in the evident sense.
\end{rem}

\subsubsection{Example: Supercuspidal parameters}

Here we exactly recover Fargues's original eigensheaf conjecture.
More precisely, for supercuspidal $\phi$, Conjecture \ref{conj:generouseigen}
was already formulated by Fargues in 2014 \cite{Far}, and Conjecture
\ref{conj:generousdeep} is essentially stated in \cite{FS}. Here
some partial results are known. More precisely, Conjecture \ref{conj:generousdeep}
is known for $\mathrm{GL}_{n}$ \cite{ALB,H2} and unramified $\mathrm{U}_{2n+1}/\mathbf{Q}_{p}$
\cite{BMHN}. Additionally, Conjecture \ref{conj:generouseigen} is
known for $\mathrm{GSp}_{4}$ \cite{Ham2} and $\mathrm{SO}_{2n+1}$
(H., unpublished) when $E/\mathbf{Q}_{p}$ is unramified with $p>2$.

\subsubsection{Example: Generic toral parameters for $\mathrm{GL}_{n}$}

We now take $G=\mathrm{GL}_{n}$. Let $\phi=\chi_{1}\oplus\cdots\oplus\chi_{n}$
be a direct sum of characters, which we also identify with characters
$\chi_{i}:E^{\times}\to\overline{\mathbf{Q}_{\ell}}^{\times}$ through
the usual reciprocity map. We assume $\phi$ is generous, so in particular
$S_{\phi}\cong\mathbf{G}_{m}^{n}$ and $\mathrm{Irr}(S_{\phi})\cong\mathbf{Z}^{n}$.
Our goal here is to sketch the following result.
\begin{thm}
Notation as above, suppose also that $\phi$ is $\ell$-integral and
generous mod-$\ell$. Then Conjecture \ref{conj:generousdeep} is
true for $\phi$.
\end{thm}

The argument relies critically on deep work of Hamann \cite{Ham}.
The extra assumptions related to $\ell$-integrality are needed in
Hamann's work, and will be irrelevant once the sheaf-theoretic machinery
improves.
\begin{proof}
Under the obvious bijection $\mathbf{j}=(j_{1},\dots,j_{n})\in\mathbf{Z}^{n}\leftrightarrow\rho_{\mathbf{j}}\in\mathrm{Irr}(S_{\phi})$,
set $|\mathbf{j}|=\sum_{i}|j_{i}|$, and let $(b_{\mathbf{j}},\pi_{\mathbf{j}}$)
be the pair associated with $(\phi,\rho_{\mathbf{j}})$ under the
BM-O bijection. We will prove the isomorphism $i_{\phi\ast}\rho_{\mathbf{j}}\ast i_{1!}W_{\psi}\simeq i_{b_{\mathbf{j}}!}^{\mathrm{ren}}\pi_{\mathbf{j}}$
by induction on $|\mathbf{j}|$.

\textbf{Base case $|\mathbf{j}|=0$.} This amounts to the assertion
that 
\[
\sigma:=W_{\psi}\otimes_{\mathfrak{Z}(G)}\mathfrak{Z}(G)/\mathfrak{m}_{\phi}\simeq i_{B}^{G}(\chi_{1}\boxtimes\cdots\boxtimes\chi_{n}),
\]
which can be proved directly. Sketch: Using Theorem \ref{thm:Wpsinice}
and the fact that $X_{G}\cong X_{G}^{\mathrm{spec}}$ for $\mathrm{GL}_{n}$
\cite{FS}, together with some standard structure theory a la Bernstein,
it is easy to see that $\sigma$ is a finite-length admissible representation
with $\sigma^{\mathrm{ss}}\simeq i_{B}^{G}(\chi_{1}\boxtimes\cdots\boxtimes\chi_{n})^{\oplus m}$
for some $m\geq1$. To show that $m\leq1$, it is enough to see that
the Jacquet module $j_{G}^{B}\sigma$ has length $\leq n!$. This
can be done by noting that $j_{G}^{B}W_{\psi}\simeq\mathcal{C}_{c}^{\infty}(T,\overline{\mathbf{Q}_{\ell}}$),
and then using the general fact that 
\[
j_{G}^{P}(-\otimes_{\mathfrak{Z}(G)}\mathfrak{Z}(G)/I)\simeq j_{G}^{P}(-)\otimes_{\mathfrak{Z}(M)}\mathfrak{Z}(M)/f(I)\mathfrak{Z}(M)
\]
as functors, where $f:\mathfrak{Z}(G)\to\mathfrak{Z}(M)$ is the usual
map induced by parabolic induction. In the case at hand, one concludes
by noting that $\mathfrak{Z}(T)/f(\mathfrak{m}_{\phi})\mathfrak{Z}(T)\simeq\oplus_{\sigma\in S_{n}}\mathfrak{Z}(T)/\mathfrak{m}_{\chi_{\sigma(1)}\boxtimes\cdots\boxtimes\chi_{\sigma(n)}}$.

\textbf{Induction step.} Suppose given $\mathbf{j}$. We can choose
some $\mathbf{j}'$ with $|\mathbf{j}'|=|\mathbf{j}|-1$ and $V\in\{\mathrm{std},\mathrm{std}^{\vee}\}$
such that $\rho_{\mathbf{j}}\in V|_{S_{\phi}}\otimes\rho_{\mathbf{j'}}$.
Note that $V|_{S_{\phi}}\otimes\rho_{\mathbf{j'}}$ is multiplicity
free, so we can decompose it as a sum $V|_{S_{\phi}}\otimes\rho_{\mathbf{j'}}\simeq\rho_{\mathbf{j}_{1}}\oplus\cdots\oplus\rho_{\mathbf{j}_{n}}$
where $\mathbf{j}'$ and $\mathbf{j}_{i}$ have the same component
except in the $i$th spot. We will now compute $T_{V}i_{b_{\mathbf{j}'}!}^{\mathrm{ren}}\pi_{\mathbf{j}'}\in D(\mathrm{Bun}_{G})^{BW_{E}}$
in two different ways, and then use the Weil group action to break
apart the results of the calculation.

\emph{First method.} By the induction hypothesis we have $i_{b_{\mathbf{j'}}!}^{\mathrm{ren}}\pi_{\mathbf{j}'}\simeq i_{\phi\ast}\rho_{\mathbf{j}'}\ast i_{1!}W_{\psi}$,
so we compute that
\begin{align*}
T_{V}i_{b_{\mathbf{j'}}!}^{\mathrm{ren}}\pi_{\mathbf{j}'} & \simeq T_{V}(i_{\phi\ast}\rho_{\mathbf{j}'}\ast i_{1!}W_{\psi})\\
 & \simeq(V\otimes i_{\phi\ast}\rho_{\mathbf{j}'})\ast i_{1!}W_{\psi}\\
 & \simeq i_{\phi\ast}(V|_{S_{\phi}}\otimes\rho_{\mathbf{j}'})\ast i_{1!}W_{\psi}\\
 & \simeq\oplus_{1\leq k\le n}(i_{\phi\ast}\rho_{\mathbf{j}_{k}}\ast i_{1!}W_{\psi})\boxtimes\chi_{k}^{\pm1}
\end{align*}
where the $\pm$ sign is chosen according to whether $V=\mathrm{std}$
or $V=\mathrm{std}^{\vee}$. Note that the Weil group action in the
fourth line comes from the tautological Weil group action on $V|_{S_{\phi}}$.

\emph{Second method.} Hamann's results show that $i_{b_{\mathbf{j'}}!}^{\mathrm{ren}}\pi_{\mathbf{j}'}\simeq\mathrm{Eis}_{B}(i_{b_{\mathbf{j}'}!}^{T}\chi)$
for all $\mathbf{j}'$ \cite[Theorem 9.1]{Ham}. Now using the filtered
commutation of $\mathrm{Eis}$ with Hecke operators, together with
the genericity, we compute that
\begin{align*}
T_{V}\mathrm{Eis}_{B}(i_{b_{\mathbf{j}'}!}^{T}\chi) & \simeq\mathrm{Eis}_{B}(T_{V|\hat{T}}i_{b_{\mathbf{j}'}!}^{T}\chi)\\
 & \simeq\mathrm{Eis}_{B}(\oplus_{1\leq k\leq n}i_{b_{\mathbf{j}_{k}}!}^{T}\chi\boxtimes\chi_{k}^{\pm1})\\
 & \simeq\oplus_{1\leq k\leq n}\mathrm{Eis}_{B}(i_{b_{\mathbf{j}_{k}}!}^{T}\chi)\boxtimes\chi_{k}^{\pm1}\\
 & \simeq\oplus_{1\leq k\leq n}i_{b_{\mathbf{j}_{k}}!}^{\mathrm{ren}}\pi_{\mathbf{j}_{k}}\boxtimes\chi_{k}^{\pm1}
\end{align*}
with the same sign convention, where we have used Hamann's results
again. 

Equating the outcomes of these two calculations, and using the fact
that $\chi_{1},\dots,\chi_{k}$ are distinct as characters of $W_{E}$,
we get the desired result.
\end{proof}
Next, we explain how the finiteness conditions and t-structures should
behave around generous parameters.
\begin{prop}
\label{prop:generous-decomposition-t-exact}Let $\phi$ be a generous
L-parameter, and assume Conjecture \ref{conj:generousdeep} for $\phi$.
Assume also that the BM-O bijection for $\phi$ exhausts all pairs
$(b,\pi)$ such that $i_{b!}^{\mathrm{ren}}\pi$ has Fargues-Scholze
parameter $\phi$.

Then the functors $i_{b!}^{\mathrm{ren}}$ and $i_{b}^{\ast\mathrm{ren}}$
induce a canonical direct product decomposition
\[
D(\mathrm{Bun}_{G})_{\phi}^{\mathrm{ULA}}\cong\prod_{b}D(G_{b}(E),\overline{\mathbf{Q}_{\ell}})_{\phi}^{\mathrm{ULA}}
\]
identifying the perverse t-structure on the left-hand side with the
product of the standard t-structures on the right-hand side, and the
Hecke action of $\mathrm{Rep}(\hat{G})$ on $D(\mathrm{Bun}_{G})_{\phi}^{\mathrm{ULA}}$
is perverse t-exact. 

Moreover, the perverse and hadal t-structures on $D(\mathrm{Bun}_{G})_{\mathrm{fin}}$
coincide, and the same functors as above induce a canonical direct
sum decomposition
\[
D(\mathrm{Bun}_{G})_{\mathrm{fin},\phi}\cong\bigoplus_{b}D(G_{b}(E),\overline{\mathbf{Q}_{\ell}})_{\mathrm{fin},\phi}
\]
compatible with the decomposition above and identifying the hadal
(=perverse) t-structure on the left-hand side with the sum of the
standard t-structures on the right-hand side. Finally, the Hecke action
of $\mathrm{Rep}(\hat{G})$ on $D(\mathrm{Bun}_{G})_{\mathrm{fin},\phi}$
is hadal t-exact.
\end{prop}

\begin{proof}[Sketch]
Using the second half of Conjecture \ref{conj:generousdeep} together
with the exhaustion assumption, it is clear that we have isomorphisms
of functors $i_{b!}^{\mathrm{ren}}\simeq i_{b\ast}^{\mathrm{ren}}\simeq i_{b\sharp}^{\mathrm{ren}}$
and $i_{b}^{\ast\mathrm{ren}}\simeq i_{b}^{!\mathrm{ren}}$ on the
$\phi$-localized ULA/finite categories, which forces these functors
to be t-exact with respect to the perverse/hadal t-structure on $\mathrm{Bun}_{G}$
and the standard t-structures on the representation categories. Moreover,
all the gluing functors $i_{b'}^{\ast\mathrm{ren}}i_{b\ast}^{\mathrm{ren}}$
between $\phi$-localized ULA sheaves on different strata vanish identically.
This implies the direct product decomposition and the identifications
of t-structures as stated. For the statements regarding the Hecke
action, it is enough to prove that any sheaf of the form $T_{V}i_{b!}^{\mathrm{ren}}\pi$
is perverse/hadal. This follows from the discussion in Remark \ref{rem:generous-Hecke}.
\end{proof}
\begin{xca}
Fix a generous parameter $\phi$. Assume that the hypotheses of Proposition
\ref{prop:generous-decomposition-t-exact} hold for $\phi$ and $\phi^{\vee}$,
and that the BM-O bijections for $\phi$ and $\phi^{\vee}$ satisfy
the expected compatibility with smooth duality as in Remark \ref{rem:generous-duality-trick}.
How do the decompositions of categories in Proposition \ref{prop:generous-decomposition-t-exact}
interact with the relevant dualities?
\end{xca}

At this point, it is hard not to state the following (unconditional!)
conjecture.
\begin{conjecture}
If $\phi$ is a generous parameter, the functor
\[
c_{\psi}:D(\mathrm{Bun}_{G})_{\mathrm{fin},\phi}\to\mathrm{QCoh}(\mathrm{Par}_{G})
\]
is t-exact with respect to the hadal t-structure on the left-hand
side and the standard t-structure on the right-hand side.
\end{conjecture}

Of course, as noted above, we expect that the perverse and hadal t-structures
should coincide on the left-hand side. However, if we said ``perverse''
instead of ``hadal'' in the formulation of this conjecture, it would
no longer be entirely unconditional, because it is not clear a priori
that perverse truncations preserve $D(\mathrm{Bun}_{G})_{\mathrm{fin},\phi}$.

\subsection{The trivial \emph{L}-parameter}

In this section we assume for simplicity that $G$ is split. Let $\phi$
be the \emph{trivial} \emph{L}-parameter. Clearly $S_{\phi}=\hat{G}$,
so $\mathrm{Irr}(S_{\phi})\cong X^{\ast}(\hat{T})^{+}=X_{\ast}(T)^{+}$
by usual highest weight theory. Therefore, for any $\lambda\in X_{\ast}(T)^{+}$,
the BM-O bijection defines an associated pair $(b_{\lambda},\pi_{\lambda})$
which can be described totally explicitly. In fact, $b_{\lambda}$
is just the element $\lambda(\varpi)\in G(E)$, and $G_{b_{\lambda}}=C_{G}(\lambda)$
is the standard Levi subgroup of $G$ centralizing $\lambda$. The
representation $\pi_{\lambda}$ turns out to be the (irreducible)
normalized parabolic induction $i_{B}^{G_{b_{\lambda}}}(\mathbf{1})$
of the trivial representation, where $B\subset G_{b_{\lambda}}$ is
any choice of Borel. Our goal here is to formulate a conjecture describing
the coherent sheaves on $\mathrm{Par}_{G}$ associated with the sheaves
$i_{b_{\lambda}!}^{\mathrm{ren}}\pi_{\lambda}$ and $i_{b_{\lambda}\sharp}^{\mathrm{ren}}\pi_{\lambda}$
on $\mathrm{Bun}_{G}$. It turns out that a precise guess for these
sheaves is forced on us by the expected compatibility of the categorical
equivalence with Eisenstein functors and with duality. However, the
situation is much more complicated than for generous \emph{L-}parameters. 

To begin, observe that the fiber of the map $\mathrm{Par}_{G}\to X_{G}^{\mathrm{spec}}$
over the trivial \emph{L}-parameter is exactly the quotient stack
$\mathcal{N}/\hat{G}$, where $\mathcal{N}\subset\hat{\mathfrak{g}}$
is the nilpotent cone with its usual $\hat{G}$-action. In particular,
we have a canonical closed immersion $\nu:\mathcal{N}/\hat{G}\hookrightarrow\mathrm{Par}_{G}$.
Now, recall the $\hat{G}$-equivariant diagram of schemes
\[
\xymatrix{ & \tilde{\mathcal{N}}=T^{\ast}(\hat{G}/\hat{B})\cong\hat{\mathfrak{u}}\times^{\hat{B}}\hat{G}\ar[dl]_{\eta}\ar[dr]^{\pi}\\
\hat{G}/\hat{B} &  & \mathcal{N}
}
\]
where $\pi:\tilde{\mathcal{N}}\to\mathcal{N}$ is the Springer resolution.
For any $\lambda\in X^{\ast}(\hat{T})$, we have the usual equivariant
line bundle $\mathcal{L}_{\lambda}$ on $\hat{G}/\hat{B}$, and we
may form the associated $\hat{G}$-equivariant coherent complex $A_{\lambda}=\pi_{\ast}\eta^{\ast}\mathcal{L}_{\lambda}\in\mathrm{Coh}(\mathcal{N}/\hat{G})$.
The $A_{\lambda}$'s are sometimes called \emph{Andersen-Jantzen sheaves},
and they are very interesting objects in geometric representation
theory. We refer the reader to \cite{Achar} for a beautiful overview
of this topic. Among the highlights of this theory, we note that Bezrukavnikov
\cite{Bez} proved that the $A_{\lambda}$'s are \emph{perverse} with
respected to a suitable perverse coherent t-structure on $\mathrm{Coh}(\mathcal{N}/\hat{G})$,
whose heart we denote $\mathrm{PCoh}(\mathcal{N}/\hat{G})$. We also
note that $A_{0}=\mathcal{O}_{\mathcal{N}}$, and that for any $\lambda\in X^{\ast}(\hat{T})^{+}$,
$A_{\lambda}$ is an honest coherent sheaf, i.e. is concentrated in
cohomological degree zero. When $\lambda$ is dominant, it is conventional
to write $\overline{\nabla}_{\lambda}=A_{\lambda}$ and $\overline{\Delta}_{\lambda}=A_{w_{0}(\lambda)}$.
We also note that for $\lambda$ dominant, there is a unique (up to
scalar) nonzero map $\overline{\Delta}_{\lambda}\to\overline{\nabla}_{\lambda}$
whose image in $\mathrm{PCoh}$ is an irreducible object $IC_{\lambda}$,
and that this recipe gives all the irreducible objects in the heart
of the perverse coherent t-structure.
\begin{conjecture}
\label{conj:trivialparametermain}For any $\lambda\in X_{\ast}(T)^{+}$,
there are isomorphisms
\[
c_{\psi}(i_{b_{\lambda}!}^{\mathrm{ren}}\pi_{\lambda})\simeq\nu_{\ast}\overline{\nabla}_{\lambda}\;\;\;\mathrm{and}\;\;\;c_{\psi}(i_{b_{\lambda}\sharp}^{\mathrm{ren}}\pi_{\lambda})\simeq\nu_{\ast}\overline{\Delta}_{\lambda}.
\]
Moreover, the functor $c_{\psi}$ sends the canonical map
\[
i_{b_{\lambda}\sharp}^{\mathrm{ren}}\pi_{\lambda}\to i_{b_{\lambda}!}^{\mathrm{ren}}\pi_{\lambda}
\]
to the map obtained by applying $\nu_{\ast}$ to the canonical map
$\overline{\Delta}_{\lambda}\to\overline{\nabla}_{\lambda}$.
\end{conjecture}

\begin{rem}
\label{rem:trivialparameterapsivariant}It is not hard to see that
$\nu_{\ast}\overline{\nabla}_{\lambda}$ and $\nu_{\ast}\overline{\Delta}_{\lambda}$
are perfect complexes on $\mathrm{Par}_{G}$, so we could also formulate
an obvious variant of this conjecture using the more humane functor
$a_{\psi}$ instead of $c_{\psi}$.
\end{rem}

This conjecture follows from the expected compatibilities of the categorical
conjecture with Eisenstein series and duality. To see this, recall
that the unramified component of $\mathrm{Par}_{T}$ is canonically
$\mathrm{Par}_{T}^{\mathrm{nr}}\cong\hat{T}\times B\hat{T}$. We have
a line bundle $\mathcal{O}_{\lambda}$ on $B\hat{T}$ corresponding
to $\lambda$. Pushing forward along the evident closed immersion
$e:B\hat{T}\to\hat{T}\times B\hat{T}\cong\mathrm{Par}_{T}^{\mathrm{nr}}\subset\mathrm{Par}_{T}$
gives a coherent sheaf $e_{\ast}\mathcal{O}_{\lambda}$ on $\mathrm{Par}_{T}$.
Under the known categorical equivalence for $T$ \cite{Zou}, it corresponds
to the sheaf $i_{\lambda!}\mathbf{1}$ on $\mathrm{Bun}_{T}=\coprod_{X_{\ast}(T)}[\ast/T(E)]$
given by the $!$-extension of the trivial representation from the
component labelled by $\lambda$.
\begin{prop}
Notation and assumptions as above, we have an isomorphism $\mathrm{Eis}_{B}^{\mathrm{spec}}(e_{\ast}\mathcal{O}_{\lambda})\cong\nu_{\ast}A_{\lambda}$
in $\mathrm{Coh}(\mathrm{Par}_{G})$ for any $\lambda$.
\end{prop}

\begin{proof}[Proof sketch]
The key point is that there is commutative diagram of (derived) Artin
stacks
\[
\xymatrix{B\hat{T}\ar[d]^{e} & \hat{\mathfrak{u}}/\hat{B}\cong\tilde{\mathcal{N}}/\hat{G}\ar[d]^{f}\ar[l]_{q^{\mathrm{nil}}}\ar[r]^{\;\;\;\pi} & \mathcal{N}/\hat{G}\ar[d]^{\nu}\\
\mathrm{Par}_{T} & \mathrm{Par}_{B}\ar[l]^{q^{\mathrm{spec}}}\ar[r]_{p^{\mathrm{spec}}} & \mathrm{Par}_{G}
}
\]
where the vertical maps are closed immersions, and moreover the lefthand
square is Cartesian and $q^{\mathrm{spec}}$ is flat in a neighborbood
of $\mathrm{im}e$.\footnote{The hard part here is showing that the left-hand square is actually
Cartesian, and not just Cartesian on classical truncations, since
$\mathrm{Par}_{B}$ is genuinely a derived Artin stack. To verify
this, one needs to check that the derived structure of $\mathrm{Par}_{B}$
is trivial in a neighborhood of $(q^{\mathrm{spec}})^{-1}(\mathrm{im}e)$.
See Proposition \ref{prop:ssgennotderived} for a much more general
statement.} It is then straightforward to compute
\begin{align*}
\mathrm{Eis}_{B}^{\mathrm{spec}}(e_{\ast}\mathcal{O}_{\lambda}) & =p_{\ast}^{\mathrm{spec}}q^{\mathrm{spec}\ast}e_{\ast}\mathcal{O}_{\lambda}\\
 & \simeq p_{\ast}^{\mathrm{spec}}f_{\ast}q^{\mathrm{nil}\ast}\mathcal{O}_{\lambda}\\
 & \simeq\nu_{\ast}\pi_{\ast}q^{\mathrm{nil}\ast}\mathcal{O}_{\lambda}\\
 & \simeq\nu_{\ast}A_{\lambda}
\end{align*}
where the second line follows from flat base change, the third line
is trivial, and the fourth line follows from the definition of $A_{\lambda}$
plus the simple observation that $q^{\mathrm{nil}\ast}\mathcal{O}_{\lambda}=\eta^{\ast}\mathcal{L}_{\lambda}$
as line bundles on $\tilde{\mathcal{N}}/\hat{G}$.
\end{proof}
\begin{prop}
\label{prop:eisdomtrivialparameter}Notation as above, we have an
isomorphism $\mathrm{Eis}_{B}(i_{\lambda!}\mathbf{1})\simeq i_{b_{\lambda}!}^{\mathrm{ren}}\pi_{\lambda}$
for $\lambda$ dominant.
\end{prop}

\begin{proof}
This follows from Remark \ref{rem:Eisensteinpush} plus a little thought.
See also \cite[Proposition 9.4]{Ham} and the discussion immediately
afterwards.
\end{proof}
Now we can put things together: since $e_{\ast}\mathcal{O}_{\lambda}$
matches $i_{\lambda!}\mathbf{1}$ under the (known) categorical equivalence
for $T$, compatibility of the categorical equivalence with Eisenstein
functors on both sides \emph{forces} us to expect that for any dominant
$\lambda$, $\nu_{\ast}\overline{\nabla}_{\lambda}\simeq\mathrm{Eis}_{B}^{\mathrm{spec}}(e_{\ast}\mathcal{O}_{\lambda})$
should match $i_{b_{\lambda}!}^{\mathrm{ren}}\pi_{\lambda}\simeq\mathrm{Eis}_{B}(i_{\lambda!}\mathbf{1})$
under the categorical conjecture for $G$. This gives the first isomorphism
in Conjecture \ref{conj:trivialparametermain}. The second isomorphism
follows from the expected compatibility with duality. Indeed, one
checks directly that $\mathbf{D}_{\mathrm{tw.GS}}(\nu_{\ast}A_{\lambda})\cong\nu_{\ast}A_{w_{0}(\lambda)}[-\dim T]$
for any $\lambda$ (this follows, for instance, from the arguments
in section 4 of \cite{AH}), and also that $\mathbf{D}_{\mathrm{BZ}}(i_{b_{\lambda}!}^{\mathrm{ren}}\pi_{\lambda})\cong i_{b_{\lambda}\sharp}^{\mathrm{ren}}\pi_{\lambda}[-\dim T]$,
which is an easy consequence of Proposition \ref{prop:dualitypush}.
Conjecture \ref{conj:dualitycpsi} then implies the desired isomorphism.

We can use Conjecture \ref{conj:trivialparametermain} to do some
actual calculations. Let us illustrate this in the simplest case $G=\mathrm{PGL}_{2}$.
Here we identify $X^{\ast}(\hat{T})^{+}=\mathbf{Z}_{\geq0}$ in the
usual way, and write $\overline{\nabla}_{n}$ and $\overline{\Delta}_{n}$
for the associated perverse coherent sheaves as defined above. Let
$b_{n}\in|\mathrm{Bun}_{\mathrm{PGL}_{2}}|$ be the point labelled
by the highest weight $n$, so $b_{n}$ corresponds to the image of
$\mathrm{diag}(\varpi^{n},1)$ in $B(\mathrm{PGL}_{2})$. Note that
one component of $\mathrm{Bun}_{\mathrm{PGL}_{2}}$ consists of the
chain of specializations $b_{0}\rightsquigarrow b_{2}\rightsquigarrow b_{4}\cdots$,
while the other component consists of the chain $b_{1/2}\rightsquigarrow b_{1}\rightsquigarrow b_{3}\cdots$.
Here $b_{1/2}$ is the ``missing'' point, which won't play any role
in our discussion (since no sheaf supported at this point can have
trivial \emph{L}-parameter). When $n=0$, $G_{b_{0}}=G$ and $\pi_{0}=i_{B}^{G}(\mathbf{1})$
is an irreducible principal series representation. For $n\geq1$,
$G_{b_{n}}(E)=E^{\times}$ and the representation $\pi_{n}$ is just
the trivial representation. In particular, the sheaf $i_{b_{n}!}^{\mathrm{ren}}\mathbf{1}$
corresponding to $\nu_{\ast}\overline{\nabla}_{n}$ is very simple
and explicit, and is supported at one point. But how to calculate
the stalks of the sheaf $i_{b_{n}\sharp}^{\mathrm{ren}}\mathbf{1}$?
It turns out that the categorical conjecture lets us do this!

For this, recall that (as we already mentioned) there is a unique
nonzero map $\overline{\Delta}_{\lambda}\to\overline{\nabla}_{\lambda}$
whose image is an irreducible perverse coherent sheaf $IC_{\lambda}$.
Translating the diagram $\nu_{\ast}\overline{\Delta}_{\lambda}\to\nu_{\ast}IC_{\lambda}\to\nu_{\ast}\overline{\nabla}_{\lambda}$
to the other side of the categorical conjecture, we now predict the
existence of a canonical indecomposable sheaf $\mathscr{F}_{\lambda}$
on $\mathrm{Bun}_{G}$ admitting maps $i_{b_{\lambda}\sharp}^{\mathrm{ren}}\pi_{\lambda}\to\mathscr{F}_{\lambda}\to i_{b_{\lambda}!}^{\mathrm{ren}}\pi_{\lambda}$
whose composite is the canonical map, and such that $\mathbf{D}_{\mathrm{BZ}}\mathscr{F}_{\lambda}\simeq\mathscr{F}_{\lambda}[-\dim T]$.
In general, these sheaves are hard to calculate. However, for $\mathrm{PGL}_{2}$,
everything can be made very explicit. In particular, the following
results are known \cite[Section 5]{Achar}:

a) there are isomorphisms $\overline{\Delta_{0}}\cong\overline{\nabla}_{0}\cong IC_{0}\cong\mathcal{O}_{\mathcal{N}/\hat{G}}$
and $\overline{\Delta_{1}}\cong\overline{\nabla}_{1}\cong IC_{1}$

b) for $n\geq2$ there are short exact sequences of perverse coherent
sheaves
\[
0\to IC_{n}\to\overline{\nabla}_{n}\to\overline{\nabla}_{n-2}\to0
\]
and
\[
0\to\overline{\Delta}_{n-2}\to\overline{\Delta}_{n}\to IC_{n}\to0,
\]
and isomorphisms $IC_{n}\simeq i_{0\ast}V_{n-2}[-1]$, where $i_{0}:B\hat{G}\to\mathcal{N}/\hat{G}$
is the inclusion of the closed orbit.

Note that the information in a) translates into isomorphisms $i_{b_{0}\sharp}^{\mathrm{ren}}\pi_{0}\simeq i_{b_{0}!}^{\mathrm{ren}}\pi_{0}$
and $i_{b_{1}\sharp}^{\mathrm{ren}}\mathbf{1}\simeq i_{b_{1}!}^{\mathrm{ren}}\mathbf{1}\simeq\mathscr{F}_{1}$.
The first of these is completely tautological, while the second is
not tautological but follows a priori from the fact that no sheaf
supported at $b_{1/2}$ (which is the unique generization of $b_{1}$)
can have trivial \emph{L-}parameter. 

Using the first sequence in b) we can now calculate the stalks of
$\mathscr{F}_{n}$ inductively via the distinguished triangles $\mathscr{F}_{n}\to i_{b_{n}!}^{\mathrm{ren}}\pi_{n}\to i_{b_{n-2}!}^{\mathrm{ren}}\pi_{n-2}\to$
which map to this sequence under the categorical equivalence. The
result is easy to discern:
\begin{prop}
For $n\geq2$, the stalks of $\mathscr{F}_{n}$ vanish outside the
points $b_{n}$ and $b_{n-2}$. 

For $n=2$ we have $i_{b_{2}}^{\ast}\mathscr{F}_{2}=\delta^{1/2}[-2]$
and $i_{b_{0}}^{\ast}\mathscr{F}_{2}=\pi_{0}[-1]$. 

For $n>2$ we have $i_{b_{n}}^{\ast}\mathscr{F}_{n}=\delta^{1/2}[-n]$
and $i_{b_{n-2}}^{\ast}\mathscr{F}_{n}=\delta^{1/2}[1-n]$.
\end{prop}

We can now use this information together with the triangles $i_{b_{n-2}\sharp}^{\mathrm{ren}}\pi_{n-2}\to i_{b_{n}\sharp}^{\mathrm{ren}}\pi_{n}\to\mathscr{F}_{n}\to$
corresponding to the second sequence in b) to inductively calculate
the stalks of $i_{b_{n}\sharp}^{\mathrm{ren}}\mathbf{1}$. For even
$n$, the outcome is the following:
\begin{prop}
Assume $n\geq2$ is even. The stalks of $i_{b_{n}\sharp}^{\mathrm{ren}}\mathbf{1}$
vanish outside the points $b_{n},b_{n-2},\dots,b_{2},b_{0}$. At these
points, all stalk cohomologies are zero except the following:

\emph{i.} $H^{n}(i_{b_{n}}^{\ast}i_{b_{n}\sharp}^{\mathrm{ren}}\mathbf{1})=\delta^{1/2}$

\emph{ii. }For $0<2j<n$, $H^{2j}(i_{b_{2j}}^{\ast}i_{b_{n}\sharp}^{\mathrm{ren}}\mathbf{1})=H^{2j+1}(i_{b_{2j}}^{\ast}i_{b_{n}\sharp}^{\mathrm{ren}}\mathbf{1})=\delta^{1/2}$

\emph{iii. $H^{0}(i_{b_{0}}^{\ast}i_{b_{n}\sharp}^{\mathrm{ren}}\mathbf{1})=H^{1}(i_{b_{0}}^{\ast}i_{b_{n}\sharp}^{\mathrm{ren}}\mathbf{1})=i_{B}^{G}(\mathbf{1})$.}
\end{prop}

We encourage the reader to check this for themselves, and to formulate
and prove a similar result for $n\geq3$ odd. As a notable consequence
of these calculations, we get the following suggestive result.
\begin{prop}
For all $n\geq0$ the sheaves $i_{b_{n}\sharp}^{\mathrm{ren}}\pi_{n}$,
$i_{b_{n}!}^{\mathrm{ren}}\pi_{n}$, and $\mathscr{F}_{n}$ are hadal
sheaves, and the distinguished triangles
\[
\mathscr{F}_{n}\to i_{b_{n}!}^{\mathrm{ren}}\pi_{n}\to i_{b_{n-2}!}^{\mathrm{ren}}\pi_{n-2}\to
\]
and
\[
i_{b_{n-2}\sharp}^{\mathrm{ren}}\pi_{n-2}\to i_{b_{n}\sharp}^{\mathrm{ren}}\pi_{n}\to\mathscr{F}_{n}\to
\]
are short exact sequences of hadal sheaves. The hadal sheaf $\mathscr{F}_{n}$
is irreducible.
\end{prop}

\begin{proof}
For $n=0,1$ it is clear that these sheaves are hadal. For $n\geq2$,
the calculation in the previous proposition shows that $i_{b_{n}\sharp}^{\mathrm{ren}}\pi_{n}$
is coconnective for the hadal t-structure, but we also know it is
connective by definition. Thus $i_{b_{n}\sharp}^{\mathrm{ren}}\pi_{n}$
is a hadal sheaf. Now from the first triangle we get that all $\mathscr{F}_{n}$'s
are coconnective for the hadal t-structure, so then the map $i_{b_{n-2}\sharp}^{\mathrm{ren}}\pi_{n-2}\to i_{b_{n}\sharp}^{\mathrm{ren}}\pi_{n}$
in the second triangle is an injection of hadal sheaves, which then
implies that $\mathscr{F}_{n}$ is hadal. Now the first triangle implies
that $i_{b_{n}!}^{\mathrm{ren}}\pi_{n}$ is hadal by an easy induction
on $n$. 
\end{proof}
In \cite[Remark I.10.3]{FS} one finds the suggestion that the categorical
conjecture might match the perverse t-structure on $\mathrm{Bun}_{G}$
with some perverse coherent t-structure on $\mathrm{Coh}(\mathrm{Par}_{G})$.
Since the $IC_{\lambda}$'s are perverse coherent, this might naively
lead one to guess that the $\mathscr{F}_{\lambda}$'s should be perverse,
but already for $\mathrm{PGL}_{2}$ this is false. Indeed, we calculated
above that $i_{b_{0}}^{\ast}\mathscr{F}_{2}$ has nonzero cohomology
in degree $1$, which could not happen if $\mathscr{F}_{2}$ were
perverse. However, this numerology matches the fact that $IC_{2}\simeq i_{0\ast}V_{0}[-1]$
has a nonzero cohomology sheaf in degree one only. We will argue later
in these notes that after localizing over a large (and explicit) open
substack of $\mathrm{Par}_{G}$, the categorical equivalence should
be t-exact for the perverse t-structure on $\mathrm{Bun}_{G}$ and
the \emph{naive }t-structure on $\mathrm{Coh}(\mathrm{Par}_{G})$.
On the other hand, we also saw above that $\mathscr{F}_{n}$ is a
hadal sheaf. This suggests that perverse coherent t-structures on
$\mathrm{Coh}(\mathrm{Par}_{G}$) should be relevant, but that they
will be related to the \emph{hadal} t-structure on $\mathrm{Bun}_{G}$
instead.

We now make these ideas precise for the trivial parameter, again in
the context of a general split $G$.
\begin{thm}
\label{thm:trivial-parameter-categorical-magic}Suppose that for all
$\lambda\in X^{\ast}(\hat{T})^{+}$ we have isomorphisms
\[
a_{\psi}(\nu_{\ast}\overline{\Delta}_{\lambda})\simeq i_{b_{\lambda}\sharp}^{\mathrm{ren}}\pi_{\lambda}\;\;\;\mathrm{and}\;\;\;a_{\psi}(\nu_{\ast}\overline{\nabla}_{\lambda})\simeq i_{b_{\lambda}!}^{\mathrm{ren}}\pi_{\lambda}
\]
as in Remark \ref{rem:trivialparameterapsivariant}. Then the following
results are true.

\emph{i. }The triangulated functor 
\[
a_{\psi}\circ\nu_{\ast}:\mathrm{Coh}(\mathcal{N}/\hat{G})\to D(\mathrm{Bun}_{G})_{\mathrm{fin}}
\]
is t-exact with respect to the perverse coherent t-structure on the
left and the hadal t-structure on the right. In particular, it induces
an exact functor
\[
a_{\psi}\circ\nu_{\ast}:\mathrm{PCoh}(\mathcal{N}/\hat{G})\to\mathrm{Had}(\mathrm{Bun}_{G})_{\mathrm{fin}}.
\]

\emph{ii. }The sheaves $i_{b_{\lambda}\sharp}^{\mathrm{ren}}\pi_{\lambda}$
and $i_{b_{\lambda}!}^{\mathrm{ren}}\pi_{\lambda}$ are hadal.

\emph{iii. }The image of the (unique up to scalar) nonzero map $\overline{\Delta}_{\lambda}\to\overline{\nabla}_{\lambda}$
under the functor $a_{\psi}\circ\nu_{\ast}$ is nonzero, and we have
an isomorphism
\[
a_{\psi}(\nu_{\ast}IC_{\lambda})\simeq i_{b_{\lambda}\sharp!}^{\mathrm{ren}}\pi_{\lambda}.
\]
In particular, the exact functor
\[
a_{\psi}\circ\nu_{\ast}:\mathrm{PCoh}(\mathcal{N}/\hat{G})\to\mathrm{Had}(\mathrm{Bun}_{G})_{\mathrm{fin}}
\]
is faithful, and sends irreducible objects to irreducible objects.

\emph{iv. }Assume that the pairs $(b_{\lambda},\pi_{\lambda})$ exhaust
all pairs $(b,\pi)$ for which $i_{b!}^{\mathrm{ren}}\pi$ has trivial
Fargues-Scholze parameter. Then the sheaves $i_{b_{\lambda}!}^{\mathrm{ren}}\pi_{\lambda}$
are perverse.
\end{thm}

Note that the sheaves $\mathscr{F}_{\lambda}$ which we predicted
earlier, and calculated by hand for $\mathrm{PGL}_{2}$, are now revealed
more clearly: they are exactly the irreducible hadal sheaves $i_{b_{\lambda}\sharp!}^{\mathrm{ren}}\pi_{\lambda}$.

We emphasize again that the \emph{only }assumption in this theorem
is a set-theoretic matching of objects, and this matching is \emph{forced}
on us by the most basic expectations regarding the categorical local
Langlands conjecture. The fact that this matching implies so much
more is really some magic, related to the remarkable properties of
the hadal and perverse coherent t-structures.
\begin{proof}
By \cite{Bez}, the left half, resp. right half of the perverse coherent
t-structure is generated under extensions by objects of the form $\overline{\Delta}_{\lambda}[n],\,n\geq0$,
resp. objects of the form $\overline{\nabla}_{\lambda}[n],\,n\leq0$.
Our assumption now guarantees that $a_{\psi}\circ\nu_{\ast}$ sends
the left resp. right half into the left resp. right half of the hadal
t-structure. This gives i, and then ii. is an immediate consequence
of the fact that $\overline{\Delta}_{\lambda}$ and $\overline{\nabla}_{\lambda}$
are perverse coherent.

For iii., let $\gamma_{\lambda}:\overline{\Delta}_{\lambda}\to\overline{\nabla}_{\lambda}$
be a nonzero map. Pick a total ordering $<$ on $X^{\ast}(\hat{T})^{+}$
refining the usual partial ordering. By \cite{Bez}, the cone $\mathcal{K}_{\lambda}$
of the map $\gamma_{\lambda}$ lies in the triangulated subcategory
of $\mathrm{Coh}(\mathcal{N}/\hat{G})$ generated by $\overline{\nabla}_{\mu}[m]$
for $m\in\mathbf{Z}$ and $\mu<\lambda$. This implies that the $\ast$-stalks
of the sheaf $a_{\psi}(\nu_{\ast}\mathcal{K}_{\lambda})$ can only
be nonzero at points $b_{\mu}$ with $\mu<\lambda$, and in particular
the $\ast$-stalk at $b_{\lambda}$ vanishes. This implies that the
first map in the distinguished triangle
\[
a_{\psi}(\nu_{\ast}\overline{\Delta}_{\lambda})\simeq i_{b_{\lambda}\sharp}^{\mathrm{ren}}\pi_{\lambda}\overset{(a_{\psi}\circ\nu_{\ast})(\gamma_{\lambda})}{\longrightarrow}a_{\psi}(\nu_{\ast}\overline{\nabla}_{\lambda})\simeq i_{b_{\lambda}!}^{\mathrm{ren}}\pi_{\lambda}\to a_{\psi}(\nu_{\ast}\mathcal{K}_{\lambda})\to
\]
cannot be zero, because otherwise we would get an isomorphism
\[
a_{\psi}(\nu_{\ast}\mathcal{K}_{\lambda})\simeq i_{b_{\lambda}!}^{\mathrm{ren}}\pi_{\lambda}\oplus i_{b_{\lambda}\sharp}^{\mathrm{ren}}\pi_{\lambda}[1],
\]
and the right-hand side here clearly has nonzero $\ast$-stalk at
$b_{\lambda}$. Therefore, the image of $\gamma_{\lambda}$ under
$a_{\psi}\circ\nu_{\ast}$ is the (unique up to scalar) nonzero map
\[
i_{b_{\lambda}\sharp}^{\mathrm{ren}}\pi_{\lambda}\overset{\delta_{\lambda}}{\to}i_{b_{\lambda}!}^{\mathrm{ren}}\pi_{\lambda}.
\]
Since we have already seen that the source and target of $\delta_{\lambda}$
are hadal sheaves, it is clear that $\mathrm{im}\delta_{\lambda}=i_{b_{\lambda}\sharp!}^{\mathrm{ren}}\pi_{\lambda}$
by the definition of the latter as in Theorem \ref{thm:irreducible-hadal-sheaves}.
On the other hand, $a_{\psi}\circ\nu_{\ast}$ is an exact functor
of abelian categories, so it preserves images, and therefore
\[
i_{b_{\lambda}\sharp!}^{\mathrm{ren}}\pi_{\lambda}=\mathrm{im}\delta_{\lambda}\simeq a_{\psi}\circ\nu_{\ast}(\mathrm{im}\gamma_{\lambda})\simeq a_{\psi}\circ\nu_{\ast}(IC_{\lambda}).
\]
This also implies that $a_{\psi}\circ\nu_{\ast}$ is conservative
(since it doesn't send any irreducible object to the zero object),
and thus faithful.

For iv., we already know that $i_{b_{\lambda}!}^{\mathrm{ren}}\pi_{\lambda}$
lies in the left half of the of the perverse t-structure. To see that
it lies in the right half, it is enough (by the exhaustion assumption
in iv.) to check that
\[
\mathrm{Hom}(i_{b_{\mu}!}^{\mathrm{ren}}\pi_{\mu}[n],i_{b_{\lambda}!}^{\mathrm{ren}}\pi_{\lambda})=0
\]
for all $\mu$ and all $n>0$. We already know that $i_{b_{\lambda}!}^{\mathrm{ren}}\pi_{\lambda}\in\phantom{}^{h}D^{\geq0}$
in general, but from ii. we also know that $i_{b_{\mu}!}^{\mathrm{ren}}\pi_{\mu}$
is hadal for all $\mu$, and therefore $i_{b_{\mu}!}^{\mathrm{ren}}\pi_{\mu}[n]\in\phantom{}^{h}D^{\leq-1}$
for all $n>0$. The desired $\mathrm{Hom}$-vanishing then follows
from basic properties of t-structures.
\end{proof}
\begin{xca}
1. Show that for $\lambda,\mu$ dominant with $\lambda\neq\mu$, we
have $\mathrm{Hom}(i_{b_{\lambda}\sharp}^{\mathrm{ren}}\pi_{\lambda},i_{b_{\mu}!}^{\mathrm{ren}}\pi_{\mu})=0$.
(Hint: Use Proposition \ref{prop:easyvanishing}.iii.) Match this
under the categorical conjecture with Proposition 4.b) from \cite{Bez}.

2. Show that for $\lambda,\mu$ dominant with $\lambda\ngeq\mu$,
we have $\mathrm{Hom}(i_{b_{\lambda}!}^{\mathrm{ren}}\pi_{\lambda},i_{b_{\mu}!}^{\mathrm{ren}}\pi_{\mu})=0$.
(Hint: Use Proposition \ref{prop:easyvanishing}.ii.) Match this under
the categorical conjecture with the vanishing portion of \cite[Theorem 4.12]{Achar}.
How does the ``other part'' of \cite[Theorem 4.12]{Achar} translate
under the categorical conjecture?

3. Assume that the pairs $(b_{\lambda},\pi_{\lambda})$ exhaust all
pairs $(b,\pi)$ for which $i_{b!}^{\mathrm{ren}}\pi$ has trivial
Fargues-Scholze parameter. Pick a total ordering $<$ on $X^{\ast}(\hat{T})^{+}$
refining the usual partial ordering. Prove by hand that the collection
$\left\{ i_{b_{\lambda}!}^{\mathrm{ren}}\pi_{\lambda}\right\} _{\lambda\in X^{\ast}(\hat{T})^{+}}$
generates the triangulated category $D(\mathrm{Bun}_{G})_{\mathrm{fin},\phi_{\mathrm{triv}}}$
and defines a dualizable quasi-exceptional set within it, whose associated
t-structure is the hadal t-structure, and with dual quasi-exceptional
set $\left\{ i_{b_{\lambda}\sharp}^{\mathrm{ren}}\pi_{\lambda}\right\} _{\lambda\in X^{\ast}(\hat{T})^{+}}$.

4. Assume that $G=\mathrm{GL}_{n}$. If $\lambda$ is minuscule, show
that the natural map $i_{b_{\lambda}\sharp}^{\mathrm{ren}}\pi_{\lambda}\to i_{b_{\lambda}!}^{\mathrm{ren}}\pi_{\lambda}$
is an isomorphism. Observe that this is \emph{predicted} by the categorical
conjecture for all $G$: for $\lambda$ minuscule, the natural map
$\overline{\Delta}_{\lambda}\to\overline{\nabla}_{\lambda}$ is always
an isomorphism (see e.g. \cite[Proof of Prop. 3.8]{Achar}). (Hint:
If $b$ is a proper generization of $b_{\lambda}$, show that $G_{b}$
is not quasisplit. Deduce that no sheaf supported at $b$ can have
trivial \emph{L-}parameter, using the known compatibility of the FS
correspondence with the classical LLC for inner forms of $\mathrm{GL}_{n}$
\cite{HKW}.)
\end{xca}

\begin{rem}
Let $i_{0}:B\hat{G}\to\mathcal{N}/\hat{G}$ be the inclusion of the
closed orbit. It is a general fact \cite[Proposition 3.9]{Achar}
that for any dominant $\lambda$, $IC_{\lambda+2\rho}\simeq i_{0\ast}V_{\lambda}[-\dim U]$.\footnote{This is more or less the only part of the Lusztig-Vogan bijection
which can be explicitly understood.} In the particular case of $G=\mathrm{PGL}_{2}$, this reduces to
isomorphisms $IC_{n}\simeq i_{0\ast}V_{n-2}[-1]$ for all $n\geq2$,
as we've already noted. Writing $\mathcal{O}(\hat{G})\simeq\oplus_{n\geq2}V_{n-2}^{\oplus n-1}$,
we formally see that $\nu_{\ast}i_{0\ast}\mathcal{O}(\hat{G})\simeq\oplus_{n\geq2}\nu_{\ast}IC_{n}^{\oplus n-1}[1]$,
so translating to the other side we get that
\begin{align*}
a_{\psi}(\nu_{\ast}i_{0\ast}\mathcal{O}(\hat{G})) & \simeq\oplus_{n\geq2}a_{\psi}(IC_{n})^{\oplus n-1}[1]\\
 & \simeq\oplus_{n\geq2}\mathscr{F}_{n}^{\oplus n-1}[1]
\end{align*}
should be a Hecke eigensheaf with trivial eigenvalue. For general
$G$, similar reasoning leads formally to the expectation that $\mathscr{G}=\oplus_{\lambda\in X^{\ast}(\hat{T})^{+}}\mathscr{F}_{\lambda+2\rho}^{\oplus\mathrm{dim}V_{\lambda}}$
should be a Hecke eigensheaf with trivial eigenvalue.

It seems that much more can be said about the sheaves $\mathscr{F}_{\lambda}$.
We already noted above that there are isomorphisms $\mathbf{D}_{\mathrm{BZ}}\mathscr{F}_{\lambda}\simeq\mathscr{F}_{\lambda}[-\dim T]$.
However, a heuristic argument with compactified Eisenstein series
suggests that for $\lambda\in2\rho+X^{\ast}(\hat{T})^{+}$, we should
also have isomorphisms $\mathbf{D}_{\mathrm{Verd}}\mathscr{F}_{\lambda}\simeq\mathscr{F}_{\lambda}[2\mathrm{dim}U]$,
and $\mathscr{F}_{\lambda}[\dim U]$ should be an irreducible perverse
sheaf. We will come back to this later in the notes. For now, we note
that the special case of $\mathrm{PGL}_{2}$ can be understood directly
using our previous calculations. Indeed, recall from earlier the natural
map $i_{b_{n}!}^{\mathrm{ren}}\pi_{n}\to i_{b_{n-2}!}^{\mathrm{ren}}\pi_{n-2}$,
which is a surjection of hadal sheaves with kernel $\mathscr{F}_{n}$.
However, the source and target of this map are also perverse (by Theorem
\ref{thm:trivial-parameter-categorical-magic}.iv), and by some general
nonsense this same map must be an \emph{injection }of perverse sheaves,
with cokernel in the perverse category given by $\mathscr{F}_{n}[1]$!
This immediately implies that $\mathscr{F}_{n}[1]$ is perverse, and
a small additional argument identifies it with the intermediate extension
$i_{b_{n-2}!\ast}^{\mathrm{ren}}\pi_{n-2}$. In other words, we have
short exact sequences of perverse sheaves
\[
0\to i_{b_{n+2}!}^{\mathrm{ren}}\pi_{n+2}\to i_{b_{n}!}^{\mathrm{ren}}\pi_{n}\to i_{b_{n}!\ast}^{\mathrm{ren}}\pi_{n}\to0
\]
for all $n\geq0$. This implies that the perverse sheaf $i_{b_{n}!}^{\mathrm{ren}}\pi_{n}$
is uniserial of infinite length, with Jordan-Holder series
\[
\cdots i_{b_{n+6}!\ast}^{\mathrm{ren}}\pi_{n+6}-i_{b_{n+4}!\ast}^{\mathrm{ren}}\pi_{n+4}-i_{b_{n+2}!\ast}^{\mathrm{ren}}\pi_{n+2}-i_{b_{n}!\ast}^{\mathrm{ren}}\pi_{n}.
\]
Note also that $i_{b_{n}!}^{\mathrm{ren}}\pi_{n}$ has infinite length
as a perverse sheaf, despite clearly being an object of $D(\mathrm{Bun}_{G})_{\mathrm{fin}}$.
Finally, taking the Verdier dual of the previous short exact sequence
and using the self-duality of intermediate extensions, we get short
exact seqences of perverse sheaves
\[
0\to i_{b_{n}!\ast}^{\mathrm{ren}}\pi_{n}\to i_{b_{n}\ast}^{\mathrm{ren}}\pi_{n}\to i_{b_{n+2}\ast}^{\mathrm{ren}}\pi_{n+2}\to0
\]
for all $n\geq0$. This immediately implies that $i_{b_{n}\ast}^{\mathrm{ren}}\pi_{n}$
is uniserial of infinite length, with the same perverse Jordan-Holder
factors as $i_{b_{n}!}^{\mathrm{ren}}\pi_{n}$ but arranged in the
opposite order. By a simple induction, this triangle also lets us
easily compute the $\ast$-stalks of the sheaves $i_{b_{n}\ast}^{\mathrm{ren}}\pi_{n}$.
\end{rem}

\begin{prop}
For all $n\geq0$ and $j>0$, we have
\[
i_{b_{n+2j}}^{\ast\mathrm{ren}}i_{b_{n}\ast}^{\mathrm{ren}}\pi_{n}\simeq\pi_{n+2j}\oplus\pi_{n+2j}[1].
\]
\end{prop}

\subsection{Semisimple generic parameters}

In this section we sketch our expectations regarding the maximal common
generalizations of the discussion in section 2.1 and 2.2. We again
allow any quasisplit $G$. It will be convenient to adopt the following
notation: If $H$ is any linear algebraic group over $\overline{\mathbf{Q}_{\ell}}$,
we have the associated Artin stack $\mathcal{N}_{H}=\mathrm{Lie}(H)^{\mathrm{nil}}/H$
where $\mathrm{Lie}(H)^{\mathrm{nil}}$ is the nilpotent cone. If
the connected component $H^{\circ}$ is reductive, the obvious map
$\mathcal{N}_{H^{\circ}}\to\mathcal{N}_{H}$ is finite etale, and
it is not hard to see that that we can define a perverse coherent
t-structure on $\mathrm{Coh}(\mathcal{N}_{H})$ whose heart is exactly
the objects in $\mathrm{Coh}(\mathcal{N}_{H})$ whose pullback to
$\mathrm{Coh}(\mathcal{N}_{H^{\circ}})$ is perverse coherent.

Let $\phi$ be a semisimple parameter, $x_{\phi}\in X_{G}^{\mathrm{spec}}$
the associated closed point. To begin, observe that there is a canonical
closed immersion $\nu_{\phi}:\mathcal{N}_{S_{\phi}}\to\mathrm{Par}_{G}$
factoring over the closed substack $q^{-1}(x_{\phi})$, which parametrizes
exactly the \emph{L}-parameters with open kernel whose Frobenius-semisimplification
is $\phi$. We omit the proof, which is an easy exercise with the
Jordan-Chevalley decomposition. 
\begin{prop}
\label{prop:ss-gen-equivalent-conditions}The following conditions
on a semisimple parameter $\phi$ are equivalent.

1) The morphism $\nu_{\phi}$ factors through an isomorphism $\mathcal{N}_{S_{\phi}}\overset{\sim}{\to}q^{-1}(x_{\phi})^{\mathrm{red}}$.

2) The parameter $\phi$ is \emph{generic}, i.e. the adjoint L-function
$L(s,\mathrm{ad}\circ\phi)$ is regular at $s=1$.

3) The composite map $BS_{\phi}\to\mathcal{N}_{S_{\phi}}\overset{\nu_{\phi}}{\to}\mathrm{Par}_{G}$
defines a smooth point of $\mathrm{Par}_{G}$.

4) The stack $\mathrm{Par}_{G}$ is smooth in a Zariski neighborhood
of $q^{-1}(x_{\phi})$.
\end{prop}

\begin{proof}
Omitted. We note that condition 2) may seem to depend a priori on
a choice of isomorphism $\mathbf{C}\simeq\overline{\mathbf{Q}_{\ell}}$,
and we encourage the reader to convince themselves that it is actually
independent of any such choice.
\end{proof}
\begin{example}
Suppose $G=\mathrm{GL}_{n}$. Then an \emph{L-}parameter $\phi$ is
semisimple generic iff it is a direct sum $\phi\simeq\phi_{1}\oplus\cdots\oplus\phi_{d}$
where the $\phi_{i}$'s are supercuspidal parameters with $\sum_{i}\mathrm{dim}\phi_{i}=n$,
and with $\phi_{i}\not\simeq\phi_{j}(1)$ for any $i\neq j$. However,
unlike in the generous case, we allow the $\phi_{i}$'s to occur with
arbitrary multiplicities. Note that a semisimple \emph{L-}parameter
$\phi$ for $\mathrm{GL}_{n}$ is generic \emph{if and only if }the
associated $\mathrm{GL}_{n}(E)$-representation $\pi_{\phi}$ is an
irreducible parabolic induction of a supercuspidal representation
on a Levi subgroup.
\end{example}

It is not hard to see that the locus $X_{G}^{\mathrm{spec},\mathrm{gen}}\subset X_{G}^{\mathrm{spec}}$
parametrizing semisimple generic parameters is open and dense, and
contains all generous parameters. Here is another key simplifying
feature associated with these parameters.
\begin{prop}
\label{prop:ssgennotderived}Choose any parabolic $P=MU\subset G$,
so we have the natural composite map $\mathrm{Par}_{P}\overset{p^{\mathrm{spec}}}{\to}\mathrm{Par}_{G}\overset{q}{\to}X_{G}^{\mathrm{spec}}$.
Let 
\[
\mathrm{Par}_{P}^{G-\mathrm{gen}}=\mathrm{Par}_{P}\times_{X_{G}^{\mathrm{spec}}}X_{G}^{\mathrm{spec,gen}}\subset\mathrm{Par}_{P}
\]
be the evident open substack. Then the derived structure on $\mathrm{Par}_{P}^{G-\mathrm{gen}}$
is trivial.
\end{prop}

\begin{proof}
By the calculations in \cite[Section 2.3]{Zhu}, we know that $q^{\mathrm{spec}}:\mathrm{Par}_{P}\to\mathrm{Par}_{M}$
is a quasismooth morphism of quasismooth derived stacks with classical
target, and where the source and target both have virtual (equi)dimension
zero. By Proposition \ref{prop:soap}, it is then enough to prove
the following: for every closed point $x_{\phi}\in X_{M}^{\mathrm{spec}}$
whose image in $X_{G}^{\mathrm{spec}}$ lies in $X_{G}^{\mathrm{spec,gen}}$,
with associated closed immersion $\nu_{\phi}^{M}:\mathcal{N}_{S_{\phi}}\simeq(q^{M})^{-1}(x_{\phi})^{\mathrm{red}}\hookrightarrow\mathrm{Par}_{M}$,
the natural map on classical truncations $\mathrm{Par}_{P}^{\mathrm{cl}}\times_{\mathrm{Par}_{M},\nu_{\phi}^{M}}\mathcal{N}_{S_{\phi}}\to\mathcal{N}_{S_{\phi}}$
is smooth of relative dimension zero. (Here $S_{\phi}=\mathrm{Cent}_{\hat{M}}(\phi)$
is computed relative to the Levi.) This follows in turn by explicitly
identifying this fiber product with $\mathcal{N}_{S_{\phi}^{P}}$
where $S_{\phi}^{P}=\mathrm{Cent}_{\hat{P}}(\phi)$, and then using
the fact that for any surjection $H\to G$ of linear algebraic groups
with unipotent kernel, the associated map $\mathcal{N}_{H}\to\mathcal{N}_{G}$
is smooth of relative dimension zero.
\end{proof}
\begin{rem}
It is probably true that for most groups, $X_{G}^{\mathrm{spec,gen}}$
is the maximal open subvariety of $X_{G}^{\mathrm{spec}}$ whose preimage
in $\mathrm{Par}_{P}$ has trivial derived structure for all parabolics
$P\subset G$. 
\end{rem}

Let $\phi$ be a semisimple generic parameter, so we have maps $BS_{\phi}\overset{i_{0}}{\to}\mathcal{N}_{S_{\phi}}\overset{\nu_{\phi}}{\to}\mathrm{Par}_{G}$.
Here we write $i_{0}$ for the inclusion of the zero orbit, so $i_{\phi}=\nu_{\phi}\circ i_{0}:BS_{\phi}\to\mathrm{Par}_{G}$
is the closed orbit of the semisimple parameter $\phi$ just as in
the discussion of generous parameters. For any $\rho\in\mathrm{Irr}(S_{\phi})$,
we set $\delta_{\rho}=i_{0\ast}\rho\in\mathrm{Coh}(\mathcal{N}_{S_{\phi}})$,
so the pushforward $\nu_{\phi\ast}\delta_{\rho}=i_{\phi\ast}\rho$
is still a perfect complex on $\mathrm{Par}_{G}$. However, there
should be several other sheaves in the picture, reflecting the more
complicated geometry of $q^{-1}(x_{\phi})$ and our experience with
the trivial \emph{L}-parameter. More precisely, we expect that there
is a clean and explicit recipe which assigns to any $\rho\in\mathrm{Irr}(S_{\phi})$
three canonical perverse-coherent objects $\overline{\Delta}_{\rho},IC_{\rho},\overline{\nabla}_{\rho}\in\mathrm{PCoh}(\mathcal{N}_{S_{\phi}})$
together with maps
\[
\overline{\Delta}_{\rho}\twoheadrightarrow IC_{\rho}\hookrightarrow\overline{\nabla}_{\rho}
\]
realizing $IC_{\rho}$ as the socle of $\overline{\nabla}_{\rho}$
and the cosocle of $\overline{\Delta}_{\rho}$. Moreover, $IC_{\rho}$
should be irreducible, and $\overline{\nabla}_{\rho}$ should be a
genuine coherent sheaf. These sheaves should interpolate the following
properties.

1. When $\phi$ is generous, $\mathcal{N}_{S_{\phi}}\cong BS_{\phi}$
and $\overline{\Delta}_{\rho}=IC_{\rho}=\overline{\nabla}_{\rho}=\delta_{\rho}=\rho$.

2. When $S_{\phi}$ is connected, $\mathrm{Irr}(S_{\phi})$ is parametrized
by highest weights and $\overline{\Delta}_{\rho},IC_{\rho},\overline{\nabla}_{\rho}$
should be the standard, resp. irreducible, resp. costandard perverse-coherent
objects as in the discussion preceding Conjecture \ref{conj:trivialparametermain}. 

3. When $\rho$ factors over the quotient $S_{\phi}\to S_{\phi}^{\natural}$,
$\overline{\Delta}_{\rho}=IC_{\rho}=\overline{\nabla}_{\rho}$ should
be the pullback of $\rho$ along the tautological map $\mathcal{N}_{S_{\phi}}\to BS_{\phi}$.

4. For any $\rho$, there should be a uniquely determined $\rho'$
such that $\delta_{\rho}\simeq IC_{\rho'}[d]$, where $d=\mathrm{dim}U_{S_{\phi}^{\circ}}$
for $U_{S_{\phi}^{\circ}}$ the unipotent radical of a(ny) Borel subgroup
of $S_{\phi}^{\circ}$.

5. There should be a surjection of coherent sheaves $\overline{\nabla}_{\rho}\twoheadrightarrow\delta_{\rho}$.

Finally, these sheaves should have the property that if $(b,\pi)$
is the pair matching $(\phi,\rho)$ under the BM-O bijection, then
the categorical equivalence induces the following matchings of sheaves
($c_{\psi}$ in one direction, $a_{\psi}$ in the other)
\begin{align*}
\nu_{\phi\ast}\overline{\Delta}_{\rho} & \longleftrightarrow i_{b\sharp}^{\mathrm{ren}}\pi\\
\nu_{\phi\ast}IC_{\rho} & \longleftrightarrow i_{b\sharp!}^{\mathrm{ren}}\pi\\
\nu_{\phi\ast}\overline{\nabla}_{\rho} & \longleftrightarrow i_{b!}^{\mathrm{ren}}\pi\\
\nu_{\phi\ast}\delta_{\rho}=i_{\phi\ast}\rho & \longleftrightarrow i_{b!\ast}^{\mathrm{ren}}\pi
\end{align*}
compatible with the evident maps. The first three sheaves on the right
should be hadal, and the last two should be perverse. The surjection
$\nu_{\phi\ast}\overline{\nabla}_{\rho}\twoheadrightarrow\nu_{\phi\ast}\delta_{\rho}$
from 5. above should correspond to the natural surjection $i_{b!}^{\mathrm{ren}}\pi\twoheadrightarrow i_{b!\ast}^{\mathrm{ren}}\pi$
of perverse sheaves. The isomorphism $\delta_{\rho}\simeq IC_{\rho'}[d]$
from 4. should correspond to an isomorphism $i_{b!\ast}^{\mathrm{ren}}\pi\simeq i_{b'\sharp!}^{\mathrm{ren}}\pi'[d]$,
where $(b',\pi')$ is the pair matching $(\phi,\rho')$ under the
BM-O bijection. In particular, we see that irreducible perverse sheaves
over a semisimple generic parameter should be finite sheaves.

If we believe in this matching, it is also not hard to see that Hecke
operators should be t-exact for the hadal t-structure on $D(\mathrm{Bun}_{G})_{\mathrm{fin},\phi}$.
Indeed, $T_{V}i_{b\sharp}^{\mathrm{ren}}\pi$ matches with $\nu_{\phi\ast}(V\otimes\overline{\Delta}_{\rho})$,
and $V\otimes\overline{\Delta}_{\rho}$ is still perverse coherent,
so by general nonsense about t-structures associated with quasiexceptional
collections it will have a finite filtration with graded pieces of
the form $\overline{\Delta}_{\rho_{i}}[n_{i}]$ for some $n_{i}\geq0$.
Passing back to the other side, we see that $T_{V}i_{b\sharp}^{\mathrm{ren}}\pi$
should have a finite filtration with graded pieces of the form $i_{b_{i}\sharp}^{\mathrm{ren}}\pi_{i}[n_{i}]$,
which are connective for the hadal t-structure. This shows that $T_{V}$
is right t-exact, and a similar argument gives left t-exactness.
\begin{rem}
\label{rem:bzdual-hadal-interaction}It is natural to wonder how Bernstein-Zelevinsky
duality interacts with the hadal t-structure. Let us say that a semisimple
\emph{L}-parameter $\phi$ is \emph{cohomologically inert }if there
is a fixed (nonnegative) integer $d_{\phi}$ such that for all $b\in B(G)$
and all $\pi\in\Pi(G_{b})$ with Fargues-Scholze parameter $\phi$,
$\mathbf{D}_{\mathrm{coh}}(\pi)\simeq\mathrm{Zel}(\pi)[-d_{\phi}]$.
Here $\mathrm{Zel}(-)$ denotes the Aubert-Zelevinsky involution on
$\Pi(G_{b})$ as in Remark \ref{rem:dualitycomparisons}. It is easy
to see that if $\phi$ is cohomologically inert, then also $\phi^{\vee}$
is cohomologically inert with $d_{\phi}=d_{\phi^{\vee}}$, and one
can show in this case that $\mathbf{D}_{\mathrm{BZ}}(-)[d_{\phi}]$
restricts to an exact anti-equivalence of abelian categories
\begin{align*}
\mathrm{Had}(\mathrm{Bun}_{G})_{\mathrm{fin},\phi} & \overset{\sim}{\to}\mathrm{Had}(\mathrm{Bun}_{G})_{\mathrm{fin},\phi^{\vee}}\\
A & \mapsto\mathbf{D}_{\mathrm{BZ}}(A)[d_{\phi}]
\end{align*}
sending the irreducible hadal sheaf $\mathscr{G}_{b,\pi}$ to the
irreducible hadal sheaf $\mathscr{G}_{b,\mathrm{Zel}(\pi)}$. One
can also show that this result is best possible: if $\phi$ is not
cohomologically inert, then no fixed shift of $\mathbf{D}_{\mathrm{BZ}}(-)$
can induce such an equivalence. Finally, all evidence and examples
point to the speculation that a semisimple parameter is cohomologically
inert if and only if it is generic.\footnote{Note added October 2023: I recently realized that for some groups
of low rank, this is not true. In particular, for $G=\mathrm{SL}_{2}$,
every semisimple parameter is cohomologically inert.}
\end{rem}

\subsection{Perverse t-exactness of Hecke operators}

In the previous section we pointed out that Hecke operators should
be hadal t-exact on $\phi$-local finite sheaves when $\phi$ is semisimple
generic. It is also natural to wonder how the Hecke action interacts
with the perverse t-structure. Our main goal in this section is to
justify the following conjecture.
\begin{conjecture}
\label{conj:t-exact-Hecke-ssgen}If $\phi$ is a semisimple generic
parameter, the Hecke action of $\mathrm{Rep}(\hat{G})$ on $D(\mathrm{Bun}_{G})_{\phi}^{\mathrm{ULA}}$
is perverse t-exact.
\end{conjecture}

When $\phi$ is supercuspidal, this conjecture was essentially formulated
in Fargues-Scholze. For generous parameters, this conjecture is a
straightforward consequence of Conjecture \ref{conj:generousdeep},
as discussed in the proof of Proposition \ref{prop:generous-decomposition-t-exact}.
For general semisimple generic parameters, however, there is no easy
evidence in its favor. Nevertheless, we will rigorously show that
when $\phi$ is the trivial parameter, Conjecture \ref{conj:trivialparametermain}
actually implies Conjecture \ref{conj:t-exact-Hecke-ssgen}! A very
similar argument will apply to all semisimple generic parameters,
once the ideas sketched in section 2.3 are more thoroughly developed.

The essential point on the spectral side is the following result.
Here we return to the notation and setup of section 2.2; in particular
we assume $G$ is split.
\begin{prop}
\label{prop:costandard-filtration-key}For any $\lambda\in X^{\ast}(\hat{T})^{+}$
and any $V\in\mathrm{Rep}(\hat{G})$, the sheaf $V\otimes A_{\lambda}\in\mathrm{Coh}(\mathcal{N}/\hat{G})$
admits a finite filtration with graded pieces of the form $A_{\nu}[n]$
for some $\nu\in X^{\ast}(\hat{T})^{+}$ and $n\geq0$. 
\end{prop}

Despite its innocent nature, I was only able to prove this by using
the full power of the Arkhipov-Bezrukavnikov equivalence.\footnote{It is quite striking that in order to prove something about the Hecke
action on $D(\mathrm{Bun}_{G})$, we will first go to the spectral
side via the categorical equivalence, and then pass through the Langlands
mirror \emph{again }via the AB equivalence.} To prepare for the argument, recall that we have a diagram of functors
between triangulated categories
\[
\xymatrix{\mathrm{Coh}(\tilde{\mathcal{N}}/\hat{G})\ar[d]^{\pi_{\ast}} & D^{b}\mathrm{ExCoh}(\tilde{\mathcal{N}}/\hat{G})\ar[d]^{\pi_{\ast}}\ar[l]_{\sim}\ar[r]_{\sim}^{F_{\mathcal{IW}}} & D^{b}\mathrm{Perv}_{\mathcal{IW}}(\mathrm{Fl}_{\mathbf{G}},\overline{\mathbf{Q}_{\ell}})\ar[r]^{\sim} & D_{\mathcal{IW}}^{b}(\mathrm{Fl}_{\mathbf{G}},\overline{\mathbf{Q}_{\ell}})\\
\mathrm{Coh}(\mathcal{N}/\hat{G}) & D^{b}\mathrm{PCoh}(\mathcal{N}/\hat{G})\ar[l]_{\sim} & D^{b}\mathrm{Perv}_{I}(\mathrm{Fl}_{\mathbf{G}},\overline{\mathbf{Q}_{\ell}})\ar[r]\ar[u]_{\mathrm{Av}_{\mathcal{IW}}} & D_{I}^{b}(\mathrm{Fl}_{\mathbf{G}},\overline{\mathbf{Q}_{\ell}})\ar[u]_{\mathrm{Av}_{\mathcal{IW}}}
}
\]
where all but one horizontal arrow is an equivalence of categories.
Here $\mathbf{G}/\overline{\mathbf{F}_{p}}$ is a split reductive
group equipped with a fixed identification of dual groups $\hat{\mathbf{G}}\cong\hat{G}$
over $\overline{\mathbf{Q}_{\ell}}$, and all our remaining notation
essentially follows the book \cite{AR}, which we will refer to heavily.
Aside from notation we have already seen in section 2.2, we recall
that $\mathrm{ExCoh}(\tilde{\mathcal{N}}/\hat{G})$ denotes the heart
of the ``exotic'' t-structure on equivariant coherent sheaves on
the Springer resolution, and $F_{\mathcal{IW}}$ denotes the Arkhipov-Bezrukavnikov
equivalence. All unlabelled horizontal arrows are induced by the usual
realization functors, and we will elide them in our notation; this
should cause no confusion. We also recall that $F_{\mathcal{IW}}$
is t-exact for the exotic and perverse t-structures, and defines an
exact equivalence of abelian categories between $\mathrm{Ex}\mathrm{Coh}$
and $\mathrm{Perv}_{\mathcal{IW}}$. Moreover, $\pi_{\ast}$ is t-exact
for the exotic and perverse coherent t-structures. Finally, $\mathrm{Av}_{\mathcal{IW}}$
is t-exact for the evident perverse t-structures.

Each of the abelian categories $\mathrm{PCoh}$, $\mathrm{ExCoh}$,
$\mathrm{Perv}_{\mathcal{IW}}$, $\mathrm{Perv}_{I}$ is equipped
with a canonical collection of costandard objects. For $\mathrm{PCoh}$
these are exactly the sheaves $A_{\lambda}=\overline{\nabla}_{\lambda}$,
$\lambda\in X^{\ast}(\hat{T})^{+}$, which we have seen in section
2.2. For $\mathrm{ExCoh}$ and $\mathrm{Perv}_{\mathcal{IW}}$ the
indexing set for the costandard objects is the set of all $\mu\in X^{\ast}(\hat{T})$,
and we write $\nabla_{\mu}^{\mathrm{ex}}\in\mathrm{ExCoh}$ resp.
$\nabla_{\mu}^{\mathcal{IW}}\in\mathrm{Perv}_{\mathcal{IW}}$ for
the associated costandard objects. For $\mathrm{Perv}_{I}$ the indexing
set for costandard objects is the extended affine Weyl group $\widetilde{W}=W\ltimes X^{\ast}(\hat{T})$,
and we write $\nabla_{w}^{I}$ for the costandard object corresponding
to an element $w\in\widetilde{W}$.

For $\mathcal{C}$ any one of these four abelian categories, we write
$D_{+}^{b}\mathcal{C}$ for the full (but not triangulated!) subcategory
of $D^{b}\mathcal{C}$ spanned by objects which admit a finite filtration
with graded pieces of the form $A[n]$ where $A$ is costandard and
$n\geq0$. We also write $D_{\mathcal{IW}}^{b}(\mathrm{Fl}_{\mathbf{G}},\overline{\mathbf{Q}_{\ell}})_{+}$
and $D_{I}^{b}(\mathrm{Fl}_{\mathbf{G}},\overline{\mathbf{Q}_{\ell}})_{+}$
for the full subcategories spanned by objects admitting finite filtrations
whose graded pieces are nonnegative shifts of costandard objects.
\begin{prop}
\label{prop:plus-preservation}The functors $\pi_{\ast}$, $F_{\mathcal{IW}}$,
and $\mathrm{Av}_{\mathcal{IW}}$, together with the realization functors,
induce a commutative diagram of functors
\[
\xymatrix{D_{+}^{b}\mathrm{ExCoh}(\tilde{\mathcal{N}}/\hat{G})\ar[d]^{\pi_{\ast}}\ar[r]_{\sim}^{F_{\mathcal{IW}}} & D_{+}^{b}\mathrm{Perv}_{\mathcal{IW}}(\mathrm{Fl}_{\mathbf{G}},\overline{\mathbf{Q}_{\ell}})\ar[r]^{\sim} & D_{\mathcal{IW}}^{b}(\mathrm{Fl}_{\mathbf{G}},\overline{\mathbf{Q}_{\ell}})_{+}\\
D_{+}^{b}\mathrm{PCoh}(\mathcal{N}/\hat{G}) & D_{+}^{b}\mathrm{Perv}_{I}(\mathrm{Fl}_{\mathbf{G}},\overline{\mathbf{Q}_{\ell}})\ar[u]_{\mathrm{Av}_{\mathcal{IW}}}\ar[r] & D_{I}^{b}(\mathrm{Fl}_{\mathbf{G}},\overline{\mathbf{Q}_{\ell}})_{+}\ar[u]_{\mathrm{Av}_{\mathcal{IW}}}
}
\]
where both upper horizontal arrows are equivalences.
\end{prop}

\begin{proof}
Recall that all three functors in the left half of the diagram come
from exact functors on the evident abelian categories. For $\pi_{\ast}$
the claim follows from the fact that $\pi_{\ast}\nabla_{\mu}^{\mathrm{ex}}\simeq\overline{\nabla}_{\mathrm{dom}(\mu)}$
for all $\mu$, where $\mathrm{dom}(\mu)$ is the unique dominant
weight in the $W$-orbit of $\mu$ \cite[Lemma 7.3.10]{AR}. For $F_{\mathcal{IW}}$
the claim follows from the fact that $F_{\mathcal{IW}}(\nabla_{\mu}^{\mathrm{ex}})\simeq\nabla_{\mu}^{\mathcal{IW}}$
for all $\mu$ \cite[Proposition 7.1.5]{AR}. Finally, for $\mathrm{Av}_{\mathcal{IW}}$
the claim follows from the fact that $\mathrm{Av}_{\mathcal{IW}}(\nabla_{w}^{I})\simeq\nabla_{\mu}^{\mathcal{IW}}$,
where $\mu\in X^{\ast}(\hat{T})$ is the unique element with $W\cdot(1,\mu)=W\cdot w$
\cite[Lemma 6.4.5]{AR}.
\end{proof}
We now return to the task of proving Proposition \ref{prop:costandard-filtration-key}.
Note that in the ``$+$'' notation introduced above, we are simply
trying to prove that for all $\lambda\in X^{\ast}(\hat{T})^{+}$ and
$V\in\mathrm{Rep}(\hat{G})$, $\overline{\nabla}_{\lambda}\otimes V$
lies in $D_{+}^{b}\mathrm{PCoh}(\mathcal{N}/\hat{G})$. A trivial
projection formula gives an isomorphism $\overline{\nabla}_{\lambda}\otimes V\simeq\pi_{\ast}(\nabla_{\lambda}^{\mathrm{ex}}\otimes V)$,
so by Proposition \ref{prop:plus-preservation} it's enough to prove
that $\nabla_{\lambda}^{\mathrm{ex}}\otimes V$ lies in $D_{+}^{b}\mathrm{ExCoh}(\tilde{\mathcal{N}}/\hat{G})$.
Going to the other side of the Arkhipov-Bezrukavnikov equivalence
and using Proposition \ref{prop:plus-preservation} again, it's enough
in turn to prove that $F_{\mathcal{IW}}(\nabla_{\lambda}^{\mathrm{ex}}\otimes V)$
lies in $D_{+}^{b}\mathrm{Perv}_{\mathcal{IW}}(\mathrm{Fl}_{\mathbf{G}},\overline{\mathbf{Q}_{\ell}})$.

Now the magic happens. Recall that $D_{\mathcal{IW}}^{b}$ is a right
module over $D_{I}^{b}$ via the convolution action of $D_{I}^{b}$
on itself, compatibly with the functor $\mathrm{Av}_{\mathcal{IW}}$.
It is then true that 
\[
F_{\mathcal{IW}}(\mathcal{G}\otimes V)\simeq F_{\mathcal{IW}}(\mathcal{G})\star^{I}\mathscr{Z}(V)
\]
for any $\mathcal{G}\in D^{b}\mathrm{ExCoh}$ and any $V\in\mathrm{Rep}(\hat{G})$,
where $\mathscr{Z}(V)$ denotes the central sheaf in $\mathrm{Perv}_{I}$
associated with $V$. Applying this property with $\mathcal{G}=\nabla_{\lambda}^{\mathrm{ex}}$,
using the identification of costandard objects under $F_{\mathcal{IW}}$
and $\mathrm{Av}_{\mathcal{IW}}$ stated previously, and rearranging
using the centrality of $\mathscr{Z}(V)$, we get that 
\begin{align*}
F_{\mathcal{IW}}(\nabla_{\lambda}^{\mathrm{ex}}\otimes V) & \simeq F_{\mathcal{IW}}(\nabla_{\lambda}^{\mathrm{ex}})\star^{I}\mathscr{Z}(V)\\
 & \simeq\nabla_{\lambda}^{\mathcal{IW}}\star^{I}\mathscr{Z}(V)\\
 & \simeq\mathrm{Av}_{\mathcal{IW}}\left(\nabla_{w_{\lambda}}^{I}\star^{I}\mathscr{Z}(V)\right)\\
 & \simeq\mathrm{Av}_{\mathcal{IW}}\left(\mathscr{Z}(V)\star^{I}\nabla_{w_{\lambda}}^{I}\right)\\
 & \simeq\mathrm{Av}_{\mathrm{\mathcal{IW}}}(\mathscr{Z}(V))\star^{I}\nabla_{w_{\lambda}}^{I}
\end{align*}
where $w_{\lambda}$ is the evident lift. By \cite[Theorem 6.5.2]{AR},
$\mathrm{Av}_{\mathcal{IW}}(\mathscr{Z}(V))$ admits a finite filtration
in $\mathrm{Perv}_{\mathcal{IW}}$ with costandard graded pieces.
This immediately reduces us to proving that any complex of the form
$\nabla_{\nu}^{\mathcal{IW}}\star^{I}\nabla_{w}^{I}$ lies in $D_{+}^{b}\mathrm{Perv}_{\mathcal{IW}}$.
But 
\[
\nabla_{\nu}^{\mathcal{IW}}\star^{I}\nabla_{w}^{I}\simeq\mathrm{Av}_{\mathcal{IW}}(\nabla_{w_{\nu}}^{I}\star^{I}\nabla_{w}^{I}),
\]
so using Proposition \ref{prop:plus-preservation} one more time,
we're now reduced to showing that any complex of the form $\nabla_{w'}^{I}\star^{I}\nabla_{w}^{I}$
lies in $D_{I}^{b}(\mathrm{Fl}_{\mathbf{G}},\overline{\mathbf{Q}_{\ell}})_{+}$.
But this is exactly the second half of \cite[Lemma 6.5.8]{AR}.
\begin{thm}
Assume Conjecture \ref{conj:trivialparametermain}, and also that
the pairs $(b_{\lambda},\pi_{\lambda})$ exhaust all pairs $(b,\pi)$
for which $i_{b!}^{\mathrm{ren}}\pi$ has trivial Fargues-Scholze
parameter. Then the Hecke action of $\mathrm{Rep}(\hat{G})$ on $D(\mathrm{Bun}_{G})_{\phi_{\mathrm{triv}}}^{\mathrm{ULA}}$
is perverse t-exact.
\end{thm}

The exhaustion hypothesis is known unconditionally for many groups,
including $\mathrm{GL}_{n}$, $\mathrm{SL}_{n}$, and $\mathrm{GSp}_{4}$.
\begin{proof}
It suffices to prove that any Hecke operator $T_{V}$ acting on $D(\mathrm{Bun}_{G})_{\phi_{\mathrm{triv}}}^{\mathrm{ULA}}$
is perverse right t-exact. Indeed, since the right adjoint of $T_{V}$
is $T_{V^{\vee}}$, this automatically implies that $T_{V^{\vee}}$
is perverse left t-exact. Varying over all $V$, we get the claimed
reduction.

Next, using the exhaustion hypothesis, one checks that $D(\mathrm{Bun}_{G})_{\phi_{\mathrm{triv}}}^{\mathrm{ULA}}\cap\,^{p}\!D^{\leq0}$
is generated under extensions and colimits by objects of the form
$i_{b_{\lambda}!}^{\mathrm{ren}}\pi_{\lambda}[n]$ for $\lambda\in X^{\ast}(\hat{T})^{+}$
and $n\geq0$. This reduces us to checking that any sheaf of the form
$T_{V}i_{b_{\lambda}!}^{\mathrm{ren}}\pi_{\lambda}$ is perverse connective.
Under the categorical equivalence, Conjecture \ref{conj:trivialparametermain},
$T_{V}i_{b_{\lambda}!}^{\mathrm{ren}}\pi_{\lambda}$ corresponds to
the coherent complex $\nu_{\ast}(V\otimes A_{\lambda})$. By Proposition
\ref{prop:costandard-filtration-key}, this has a finite filtration
with graded pieces of the form $\nu_{\ast}A_{\mu}[n]$ for some $\mu\in X^{\ast}(\hat{T})^{+}$
and $n\geq0$. Translating back to the other side, we deduce that
$T_{V}i_{b_{\lambda}!}^{\mathrm{ren}}\pi_{\lambda}$ has a finite
filtration with graded pieces of the form $i_{b_{\mu}!}^{\mathrm{ren}}\pi_{\mu}[n]$
for some $\mu\in X^{\ast}(\hat{T})^{+}$ and $n\geq0$. Since these
graded pieces are perverse connective, we get the desired result.
\end{proof}

\subsection{Two t-exactness conjectures}

At this point, we are ready to confidently pose some precise t-exactness
conjectures for the categorical equivalence with restricted variation,
localized over semisimple generic parameters. 

Set $\mathrm{Par}_{G}^{\mathrm{gen}}=\mathrm{Par}_{G}\times_{X_{G}^{\mathrm{spec}}}X_{G}^{\mathrm{spec,gen}}$,
so $\mathrm{Par}_{G}^{\mathrm{gen}}$ is a smooth algebraic stack.
In fact, $X_{G}^{\mathrm{spec,gen}}$ is the \emph{maximal} open subscheme
of $X_{G}^{\mathrm{spec}}$ with the property that its preimage in
$\mathrm{Par}_{G}$ is a smooth algebraic stack.

Let $D(\mathrm{Bun}_{G})_{\mathrm{fin}}^{\mathrm{gen}}$ be the full
subcategory of finite sheaves $A$ whose $\phi$-local summand $A_{\phi}$
vanishes for every semisimple parameter $\phi$ which is not semisimple
generic. It is clear that Conjecture \ref{conj:categorical-restr-var}
localizes to a conjectural equivalence
\[
c_{\psi}:D(\mathrm{Bun}_{G})_{\mathrm{fin}}^{\mathrm{gen}}\overset{\sim}{\to}\mathrm{Coh}(\mathrm{Par}_{G}^{\mathrm{gen}})_{\mathrm{fin}}.
\]

\begin{conjecture}
\emph{\label{conj:categorical-fin-ssgen-perverse-version}i. }The
category $D(\mathrm{Bun}_{G})_{\mathrm{fin}}^{\mathrm{gen}}$ is stable
under the perverse truncation functors.

\emph{ii. }The equivalence 
\[
D(\mathrm{Bun}_{G})_{\mathrm{fin}}^{\mathrm{gen}}\overset{\sim}{\to}\mathrm{Coh}(\mathrm{Par}_{G}^{\mathrm{gen}})_{\mathrm{fin}}
\]
is t-exact with respect to the \emph{perverse} t-structure on the
left-hand side and the \emph{standard} t-structure on the right-hand
side.
\end{conjecture}

Note that ii. is not well-posed unless one assumes i. is true. However,
we warn the reader that i. definitely fails before passing to the
localization around semisimple generic parameters. For a concrete
example, take $G=\mathrm{SL}_{2}$ and let $i_{1}:[\ast/\underline{\mathrm{SL}_{2}(E)}]\to\mathrm{Bun}_{G}$
be the inclusion of the open stratum, with closed complement $h:Z\to\mathrm{Bun}_{G}$.
Considering the distinguished triangle
\[
i_{1!}\overline{\mathbf{Q}_{\ell}}\to\overline{\mathbf{Q}_{\ell}}\to h_{\ast}\overline{\mathbf{Q}_{\ell}}\overset{[1]}{\to},
\]
it is easy to see that the constant sheaf $\overline{\mathbf{Q}_{\ell}}$
is perverse, and (by consideration of $\ast$-stalks) that $h_{\ast}\overline{\mathbf{Q}_{\ell}}$
can only have nonvanishing perverse cohomology sheaves in degrees
$\leq-2$. Therefore applying $\phantom{}^{p}H^{0}(-)$ to the first
map of this triangle induces an isomorphism
\[
\phantom{}^{p}H^{0}(i_{1!}\overline{\mathbf{Q}_{\ell}})\overset{\sim}{\to}\overline{\mathbf{Q}_{\ell}},
\]
where of course $i_{1!}\overline{\mathbf{Q}_{\ell}}$ is finite but
$\overline{\mathbf{Q}_{\ell}}$ is not. However, all of these sheaves
are $\phi$-local for the unramified parameter $\phi:W_{E}\to\mathrm{PGL}_{2}(\overline{\mathbf{Q}_{\ell}})$
sending Frobenius to $\left(\begin{array}{cc}
1\\
 & q
\end{array}\right)$, and this parameter is certainly not semisimple generic.

On the other hand, Conjecture \ref{conj:categorical-fin-ssgen-perverse-version}
is not a wild guess. Aside from the analogy with the t-exactness results
in \cite[Section 1.6.2]{FR} (which we will discuss more carefully
in section 3.2), we have the following result.
\begin{prop}
If Conjectures \ref{conj:categorical-restr-var} and \ref{conj:t-exact-Hecke-ssgen}
are true, then so is Conjecture \ref{conj:categorical-fin-ssgen-perverse-version}.
\end{prop}

\begin{proof}
If $A\in D(\mathrm{Bun}_{G})_{\mathrm{fin}}^{\mathrm{gen}}$ is perverse
connective resp. coconnective, then $T_{V}A$ is perverse connective
resp. coconnective for all $V\in\mathrm{Rep}(\hat{G})$ by Conjecture
\ref{conj:t-exact-Hecke-ssgen}, so $i_{1}^{\ast}T_{V}A$ is concentrated
in nonpositive resp. nonnegative degrees by the perverse t-exactness
of $i_{1}^{\ast}$, and then
\[
R\mathrm{Hom}(i_{1!}W_{\psi},T_{V}A)\simeq R\mathrm{Hom}(W_{\psi},i_{1}^{\ast}T_{V}A)
\]
is also concentrated in nonpositive resp. nonnegative degrees by the
projectivity of $W_{\psi}$ (see the discussion around Theorem \ref{thm:Wpsinice}).
On the other side of Conjecture \ref{conj:categorical-restr-var},
this translates into the knowledge that
\[
R\Gamma(\mathrm{Par}_{G},V\otimes c_{\psi}(A))\simeq R\mathrm{Hom}(i_{1!}W_{\psi},T_{V}A)
\]
is concentrated in nonpositive resp. nonnegative degrees for all $V\in\mathrm{Rep}(\hat{G})$.
But this implies that $c_{\psi}(A)$ is concentrated in nonpositive
resp. nonnegative degrees for the standard t-structure, since$R\Gamma(\mathrm{Par}_{G},V\otimes-)$
is a \emph{t-exact }conservative family on $\mathrm{QCoh^{qc}}(\mathrm{Par}_{G})$
for varying $V\in\mathrm{Rep}(\hat{G})$.\footnote{I sincerely thank Xinwen Zhu for explaining this last fact to me.} 

This implies the t-exactness claim in \ref{conj:categorical-fin-ssgen-perverse-version}.ii,
and then since standard truncation on $\mathrm{Coh}$ preserves finite
sheaves, the t-exactness of the equivalence 
\[
D(\mathrm{Bun}_{G})_{\mathrm{fin}}^{\mathrm{gen}}\overset{\sim}{\to}\mathrm{Coh}(\mathrm{Par}_{G}^{\mathrm{gen}})_{\mathrm{fin}}
\]
translates this into the claim that perverse truncations on $D(\mathrm{Bun}_{G})$
preserve $D(\mathrm{Bun}_{G})_{\mathrm{fin}}^{\mathrm{gen}}$. This
gives the result.
\end{proof}
Finally we formulate a t-exactness conjecture for the hadal t-structure.
This is somewhat implicit in the discussion from section 2.3, but
for completeness we spell it out fully. Here we freely reuse the notations
introduced in section 2.3. For a fixed semisimple generic parameter
$\phi$, define $^{pcoh}D^{\leq0}$ resp. $^{pcoh}D^{\geq0}$ inside
$\mathrm{Coh}(\mathrm{Par}_{G})_{\phi}$ as the full subcategory generated
under extensions by $\nu_{\phi\ast}\overline{\Delta}_{\rho}[n]$ for
$n\geq0$, resp. by $\nu_{\phi\ast}\overline{\nabla}_{\rho}[n]$ for
$n\leq0$.
\begin{conjecture}
\emph{i. }The pair $(^{pcoh}D^{\leq0},{}^{pcoh}D^{\geq0})$ defines
a perverse coherent t-structure on $\mathrm{Coh}(\mathrm{Par}_{G})_{\phi}$.
Writing $\mathrm{PCoh}(\mathrm{Par}_{G})_{\phi}$ for the heart, the
functor $\nu_{\phi\ast}:\mathrm{Coh}(\mathcal{N}_{S_{\phi}})\to\mathrm{Coh}(\mathrm{Par}_{G})_{\phi}$
should induce a faithful exact functor 
\[
\mathrm{PCoh}(\mathcal{N}_{S_{\phi}})\to\mathrm{PCoh}(\mathrm{Par}_{G})_{\phi}
\]
which is bijective on isomorphism classes of irreducible objects.

\emph{ii. }The equivalence
\[
c_{\psi}:D(\mathrm{Bun}_{G})_{\mathrm{fin},\phi}\overset{\sim}{\to}\mathrm{Coh}(\mathrm{Par}_{G})_{\phi}
\]
induced by $\phi$-localizing the equivalence of Conjecture \ref{conj:categorical-restr-var}
should be t-exact with respect to the \emph{hadal} t-structure on
the left-hand side and the \emph{perverse coherent} t-structure defined
above on the right-hand side. In particular, it should restrict to
an exact equivalence of abelian categories
\[
\mathrm{Had}(\mathrm{Bun}_{G})_{\mathrm{fin},\phi}\overset{\sim}{\to}\mathrm{PCoh}(\mathrm{Par}_{G})_{\phi}.
\]
\end{conjecture}

Part i. of this conjecture should be quite easy to verify, conditionally
on working out the speculations from section 2.3.

\section{Additional conjectures}

\subsection{ULA sheaves and generic parameters}

The reader may have noticed that general (i.e. non-finite) ULA sheaves
have been largely absent from our discussion. This is due to a psychological
complication, which we have tried to avoid confronting until now:
when translating general ULA sheaves to the spectral side, we are
forced unavoidably to reckon with $\mathrm{IndCoh}$. More precisely,
we have the following definition.
\begin{defn}
An object $A\in\mathrm{IndCoh}(\mathrm{Par}_{G})$ is \emph{admissible
}if for all $B\in\mathrm{Coh}(\mathrm{Par}_{G})$, $R\mathrm{Hom}(B,A)$
lies in $\mathrm{Perf}(\overline{\mathbf{Q}_{\ell}})$. We write $\mathrm{Adm}(\mathrm{Par}_{G})\subset\mathrm{IndCoh}(\mathrm{Par}_{G})$
for the evident stable $\infty$-category of admissible ind-coherent
sheaves.
\end{defn}

If we believe the categorical conjecture, then these are exactly the
sheaves on the spectral side which should match with ULA sheaves on
the automorphic side.
\begin{prop}
\label{prop:categorical-conj-ULA}If Conjecture \ref{conj:categorical-cpsi-unconditional}
is true, then the ind-extension of $c_{\psi}$ to an equivalence $\mathbf{L}_{\psi}:D(\mathrm{Bun}_{G})\overset{\sim}{\to}\mathrm{IndCoh}(\mathrm{Par}_{G})$
restricts to an equivalence of categories
\[
\mathbf{L}_{\psi}:D(\mathrm{Bun}_{G})^{\mathrm{ULA}}\overset{\sim}{\to}\mathrm{Adm}(\mathrm{Par}_{G})
\]
linear over the $\otimes$-action of $\mathrm{Perf}(\mathrm{Par}_{G})$.
\end{prop}

\begin{proof}
This is immediate from the fact that ULA sheaves $A$ on $\mathrm{Bun}_{G}$
are characterized by the condition that $R\mathrm{Hom}(B,A)$ lies
in $\mathrm{Perf}(\overline{\mathbf{Q}_{\ell}})$ for all $B\in D(\mathrm{Bun}_{G})^{\omega}$,
which follows from \cite[Prop. VII.7.4 and Prop. VII.7.9]{FS}.
\end{proof}
Again, we emphasize that after restricting to $D(\mathrm{Bun}_{G})_{\mathrm{fin}}\subset D(\mathrm{Bun}_{G})^{\mathrm{ULA}}$
the functors $\mathbf{L}_{\psi}$ and $c_{\psi}$ coincide, but they
do not agree on all ULA sheaves, as discussed in the warning before
Conjecture \ref{conj:dualitycpsi}. 
\begin{xca}
1. Show that there is an inclusion $\mathrm{Adm}(\mathrm{Par}_{G})\cap\mathrm{Coh}(\mathrm{Par}_{G})\subset\mathrm{Coh}(\mathrm{Par}_{G})_{\mathrm{fin}}$.

2. Show that there is an inclusion $\mathrm{Perf}(\mathrm{Par}_{G})_{\mathrm{fin}}\subset\mathrm{Adm}(\mathrm{Par}_{G})\cap\mathrm{Coh}(\mathrm{Par}_{G})$.

3. Show that if the categorical conjecture is true, then the inclusion
in 1. must be an equality.
\end{xca}

We warn the reader that most admissible sheaves are not coherent.
For instance, let $x\in\mathrm{Par}_{G}(\overline{\mathbf{Q}_{\ell}})$
be any point in the smooth locus of the stack $\mathrm{Par}_{G}$,
with associated residual gerbe $BS_{x}$, and let $i_{x}:BS_{x}\to\mathrm{Par}_{G}$
be the evident immersion. If $\rho$ is any irreducible algebraic
representation of $S_{x}$, one can show that $i_{x\ast}^{\mathrm{IndCoh}}\rho$
is always an admissible sheaf (we will prove a more general result
below). However, if $x$ is not a closed point, this sheaf will not
be coherent.

Next we formulate a duality conjecture. If $X$ is any quasismooth
algebraic stack over $\mathrm{Spec}\,\overline{\mathbf{Q}_{\ell}}$
with structure map $f_{X}$, and $A\in\mathrm{IndCoh}(X)$ is a given
object, then we get a contravariant functor
\begin{align*}
\mathrm{IndCoh}(X) & \to\mathrm{IndCoh}(\mathrm{Spec}\,\overline{\mathbf{Q}_{\ell}})=\mathrm{Vect}_{\overline{\mathbf{Q}_{\ell}}}\\
B & \mapsto R\mathrm{Hom}(f_{X\ast}^{\mathrm{IndCoh}}(A\otimes^{!}B),\overline{\mathbf{Q}_{\ell}})
\end{align*}
where the notation for pushforward and tensor product of ind-coherent
sheaves follows Gaitsgory-Rozenblyum's book. By some general nonsense
with the $\infty$-categorical adjoint functor theorem, this functor
is representable by $R\mathrm{Hom}(-,\mathbf{D}_{\mathrm{adm}}A)$
for a uniquely determined object $\mathbf{D}_{\mathrm{adm}}A\in\mathrm{IndCoh}(X)$.
The association $A\mapsto\mathbf{D}_{\mathrm{adm}}A$ is a contravariant
endofunctor of $\mathrm{IndCoh}(X)$. For general stacks and general
sheaves, this functor of ``admissible dual'' will not be well-behaved.
However, for the stack of \emph{L}-parameters, we expect the following.\footnote{A proof of Conjecture \ref{conj:ULA-duality} will appear in \cite{HM}.}
\begin{conjecture}
\emph{\label{conj:ULA-duality}i.} The functor $\mathbf{D}_{\mathrm{adm}}(-)$
defines an involutive anti-equivalence from $\mathrm{Adm}(\mathrm{Par}_{G})$
to itself.

\emph{ii. }The equivalence $\mathbf{L}_{\psi}:D(\mathrm{Bun}_{G})^{\mathrm{ULA}}\overset{\sim}{\to}\mathrm{Adm}(\mathrm{Par}_{G})$
conjectured in Proposition \ref{prop:categorical-conj-ULA} should
satisfy the duality compatibility
\[
\mathbf{D}_{\mathrm{tw.adm}}\circ\mathbf{L}_{\psi}\simeq\mathbf{L}_{\psi^{-1}}\circ\mathbf{D}_{\mathrm{Verd}},
\]
where $\mathbf{D}_{\mathrm{tw.adm}}=c^{\ast}\mathbf{D}_{\mathrm{adm}}$
is the composition of admissible duality with pullback along the Chevalley
involution.
\end{conjecture}

This suggests another perspective on the ``best hope'' discussed
at the beginning of section 2. Specifically, if we believe in Proposition
\ref{prop:categorical-conj-ULA} and Conjecture \ref{conj:ULA-duality},
we are forced to believe in the existence of an exotic t-structure
on $\mathrm{Adm}(\mathrm{Par}_{G})$ matching the perverse t-structure
on $D(\mathrm{Bun}_{G})^{\mathrm{ULA}}$, whose left and right halves
are swapped by $\mathbf{D}_{\mathrm{tw.adm}}$. Over the smooth locus
of $\mathrm{Par}_{G}$, this should just be the standard t-structure,
but it seems extremely subtle to extend the desired t-structure over
the singularities of the stack of $L$-parameters. For instance, one
can show that for most $V\in\mathrm{Rep}(\hat{G})$, the endofunctor
$V\otimes-$ corresponding to the Hecke operator $T_{V}$ is neither
left nor right t-exact for this t-structure, and its failure of t-exactness
around a given point $x$ seems to correlate with ``how singular''
the stack is at $x$. This seems to suggest that any direct definition
of this t-structure will need to use specific features of the stack
of singularities of $\mathrm{Par}_{G}$. We will discuss these ideas
in more detail elsewhere.

Next we formulate some conjectures attaching ULA Hecke eigensheaves
on $\mathrm{Bun}_{G}$ to generic \emph{L}-parameters. In some sense
this brings us back to the very origins of the entire subject in Fargues's
2014 MSRI lecture \cite{FarguesMSRI}. To simplify the discussion,
we will consider \emph{L}-parameters satisfying the following condition.

($\dagger$) $\phi$ is Frobenius-semisimple and generic, and $S_{\phi}^{\circ}$
is reductive.

We emphasize that in contrast to most of the discussion in section
2, we are no longer requiring $\phi$ to be semisimple. In fact, the
condition $(\dagger)$ is very mild. For instance, it holds for all
discrete parameters, all parameters which are $\iota$-essentially
tempered for some isomorphism $\iota:\overline{\mathbf{Q}_{\ell}}\overset{\sim}{\to}\mathbf{C}$,
all semisimple generic parameters, and all generic parameters such
that $\phi(\mathrm{Fr})$ is regular semisimple. The last two claims
here are easy, and the first two (which require some work) are proved
in \cite[Section 3]{BMIY}. Given any such \emph{L}-parameter, we
get a canonical immersion $i_{\phi}:BS_{\phi}\to\mathrm{Par}_{G}$
which factors through the smooth locus in the stack of \emph{L}-parameters
by our genericity assumption.
\begin{prop}
Let $M\in\mathrm{IndCoh}(BS_{\phi})$ be any object with the property
that for each $\rho\in\mathrm{Irr}(S_{\phi})$, the total multiplicity
$\sum_{n}\mathrm{dim}\mathrm{Hom}_{S_{\phi}}(\rho,H^{n}(M))$ is finite.
Then $i_{\phi\ast}^{\mathrm{IndCoh}}M$ is an admissible sheaf.
\end{prop}

\begin{proof}
If $B$ lies in $\mathrm{Coh}(\mathrm{Par}_{G})$, then $i_{\phi}^{\ast\mathrm{IndCoh}}B$
lies in $\mathrm{Coh}(BS_{\phi})$, using that $i_{\phi}$ has finite
tor-dimension since it is a regular immersion of stacks. We then get
that
\[
R\mathrm{Hom}(B,i_{\phi\ast}^{\mathrm{IndCoh}}M)\simeq R\mathrm{Hom}(i_{\phi}^{\ast\mathrm{IndCoh}}B,M)
\]
by adjunction, and this lies in $\mathrm{Perf}(\overline{\mathbf{Q}_{\ell}})$
by the coherence of $i_{\phi}^{\ast\mathrm{IndCoh}}B$, the reductivity
of $S_{\phi}^{\circ}$, and our assumption on $M$.
\end{proof}
Note that the condition in this proposition is trivially satisfied
if $M$ is an irreducible algebraic $S_{\phi}$-representation, but
it is also satisfied if $M=\mathcal{O}(S_{\phi})$ is the regular
representation, even though the latter is not coherent! In particular,
we get an admissible sheaf $i_{\phi\ast}^{\mathrm{IndCoh}}\mathcal{O}(S_{\phi})$
which (by our reductivity assumption) admits a canonical direct sum
decomposition with admissible pieces
\[
i_{\phi\ast}^{\mathrm{IndCoh}}\mathcal{O}(S_{\phi})\simeq\oplus_{\rho\in\mathrm{Irr}(S_{\phi})}i_{\phi\ast}^{\mathrm{IndCoh}}\rho^{\oplus\mathrm{dim}\rho}.
\]
Moreover, by \cite[Lemma 8.3.2]{AG}, we see that $i_{\phi\ast}^{\mathrm{IndCoh}}\rho\simeq i_{\phi\ast}\rho$
lies in $\mathrm{QCoh}\subset\mathrm{IndCoh}$, so we are actually
in the regime where the inverse to the equivalence $\mathbf{L}_{\psi}$
should be realized by the spectral action functor $a_{\psi}$. In
particular, we obtain a canonical sheaf $\mathscr{F}_{\phi}=a_{\psi}(i_{\phi\ast}\mathcal{O}(S_{\phi}))$
admitting a direct sum decomposition
\[
\mathscr{F}_{\phi}\simeq\oplus_{\rho\in\mathrm{Irr}(S_{\phi})}\mathscr{F}_{\phi,\rho}^{\oplus\mathrm{dim}\rho}
\]
where of course we set $\mathscr{F}_{\phi,\rho}=a_{\psi}(i_{\phi\ast}\rho)$,
which should match the above direct sum decomposition of ind-coherent
sheaves under the functor $\mathbf{L}_{\psi}$. However, we stress
that the definition of $\mathscr{F}_{\phi,\rho}$ is totally unconditional.
Moreover, we expect that these sheaves will enjoy many remarkable
properties.

1. By the usual formal argument, $\mathscr{F}_{\phi}$ is a Hecke
eigensheaf with eigenvalue $\phi$. A more refined argument shows
that for any $V\in\mathrm{Rep}(^{L}G)$, there is an isomorphism
\[
T_{V}\mathscr{F}_{\phi,\rho}\simeq\bigoplus_{\rho'\in\mathrm{Irr}(S_{\phi})}\mathscr{F}_{\phi,\rho'}\boxtimes\mathrm{Hom}_{S_{\phi}}(\rho^{\vee}\otimes\rho',V\circ\phi)
\]
as objects of $D(\mathrm{Bun}_{G})^{BW_{E}}$.

2. We expect that $\mathscr{F}_{\phi}$ is ULA and perverse, and each
$\mathscr{F}_{\phi,\rho}$ is ULA, perverse and indecomposable. Moreover,
the ``easy part'' of the BM-O algorithm begins by (unconditionally!)
attaching to the pair $(\phi,\rho)$ an element $b=b_{\phi,\rho}\in B(G)$,
and we expect that $\mathrm{supp}\mathscr{F}_{\phi,\rho}\subseteq\overline{\{b\}}$.
Additionally, we expect that $i_{b}^{\ast\mathrm{ren}}\mathscr{F}_{\phi,\rho}$
is a single irreducible $G_{b}(E)$-representation $\pi$ concentrated
in degree zero, and that this is \emph{the} $G_{b}(E)$-representation
attached to $(\phi,\rho)$ by the $B(G)$ local Langlands correspondence.
In other words, we expect that this stalk realizes the true $B(G)$
local Langlands correspondence at \emph{L}-parameters satisfying $(\dagger)$!
Since the condition $(\dagger)$ includes all tempered parameters,
and there is a standard procedure to pass from tempered local Langlands
to full local Langlands, we conclude that \emph{the spectral action
should completely pin down the long-sought set-theoretic local Langlands
correspondence in all generality!}

We also expect that the induced map
\[
i_{b!}^{\mathrm{ren}}\pi\to\mathscr{F}_{\phi,\rho}
\]
induces a surjection (in the perverse category) from $^{p}H^{0}(i_{b!}^{\mathrm{ren}}\pi)$
onto a perverse subsheaf of $\mathscr{F}_{\phi,\rho}$ containing
its socle. In general, $\mathscr{F}_{\phi,\rho}$ should \emph{not}
be the intermediate extension $i_{b!\ast}^{\mathrm{ren}}\pi$ - we
believe that this happens exactly when $\phi$ is semisimple generic.
When $\phi$ is discrete, $\mathscr{F}_{\phi,\rho}$ seems to be a
\emph{tilting} perverse sheaf \cite{BBM}.

3. We expect that $\mathbf{D}_{\mathrm{Verd}}(\mathscr{F}_{\phi,\rho})\simeq\mathscr{F}_{\phi^{\vee},c\circ\rho^{\vee}}$,
where $\mathscr{F}_{\phi^{\vee},c\circ\rho^{\vee}}=a_{\psi^{-1}}(i_{\phi^{\vee}\ast}c\circ\rho^{\vee})$
is the sheaf (conjecturally) matching $i_{\phi^{\vee}\ast}c\circ\rho^{\vee}$
under the functor $\mathbf{L}_{\psi^{-1}}$. Indeed, one can prove
unconditionally that $\mathbf{D}_{\mathrm{tw.adm}}$ exchanges $i_{\phi\ast}^{\mathrm{IndCoh}}\rho$
and $i_{\phi^{\vee}\ast}^{\mathrm{IndCoh}}c\circ\rho^{\vee}$, so
this expectation is forced upon us by Conjecture \ref{conj:ULA-duality}.ii.

4. We expect that suitable linear combinations (over varying $\rho$)
of the virtual stalks $[i_{b}^{\ast\mathrm{ren}}\mathscr{F}_{\phi,\rho}]$
can be described explicitly. For instance, if $\phi$ is discrete
and $b$ has trivial Kottwitz invariant, we expect that for every
elliptic endoscopic datum $\mathcal{H}=(H,s,\eta)$ such that $\phi$
admits a factorization $\phi=\phantom{}^{L}\eta\circ\phi^{H}$ there
should be an equality of the form
\[
\sum_{\rho\in\mathrm{Irr}(\pi_{0}(\overline{S}_{\phi}))}\mathrm{tr}\rho(s)\cdot[i_{b}^{\ast\mathrm{ren}}\mathscr{F}_{\phi,\rho}]=\mathrm{Red}_{b}^{\mathcal{H}}(S\Theta_{\phi^{H}}).
\]
Here $\mathrm{Red}_{b}^{\mathcal{H}}$ is the map from stable virtual
representations of $H(E)$ towards virtual representations of $G_{b}(E)$
defined in \cite[Definition 5.6]{BM}, and $S\Theta_{\phi^{H}}$ is
the stable virtual representation of $H(E)$ attached to $\phi^{H}$.
In general, the stalks of the individual sheaves $\mathscr{F}_{\phi,\rho}$
seem extremely hard to describe, even as virtual representations.

\subsection{Generalized coherent Springer sheaves}

Some of the phenomena predicted in section 3.1 would be neatly explained
by yet another conjecture, with many consequences.\footnote{Conjecture \ref{conj:cohspringermiracle} and some portions of Proposition
\ref{prop:springermiracleconsequences} below were independently discovered
by Koshikawa, who also independently noticed the utility of the condition
$(\dagger)$. I would also like to acknowledge that the tilting property
for Hecke eigensheaves at discrete parameters was suggested by a beautiful
example explained to me by Koshikawa. I discovered Conjecture \ref{conj:cohspringermiracle}
while trying to explain this tilting property.} To formulate this, note that for any $b$ and any open compact $K\subset G_{b}(E)$,
$i_{b!}^{\mathrm{ren}}\mathrm{ind}_{K}^{G_{b}(E)}\overline{\mathbf{Q}_{\ell}}$
is compact, so Conjecture \ref{conj:categorical-cpsi-unconditional}
predicts that
\[
\mathscr{C}_{b,K}\overset{\mathrm{def}}{=}c_{\psi}\left(i_{b!}^{\mathrm{ren}}\mathrm{ind}_{K}^{G_{b}(E)}\overline{\mathbf{Q}_{\ell}}\right)
\]
is a bounded coherent complex with quasicompact support. (Strictly
speaking, $\mathscr{C}_{b,K}$ depends also on the Whittaker datum,
but we suppress this from the notation.)
\begin{conjecture}
\label{conj:cohspringermiracle}For all $b$ and $K$ as above (and
all Whittaker data), $\mathscr{C}_{b,K}$ is a genuine coherent sheaf,
i.e. it lies in the heart $\mathrm{Coh}(\mathrm{Par}_{G})^{\heartsuit}$
of the standard t-structure on $\mathrm{Coh}$.
\end{conjecture}

When $b=1$, $G$ is unramified, and $K=I$ is an Iwahori, $\mathscr{C}_{1,I}$
should coincide with the ``coherent Springer sheaf'' studied by
Ben--Zvi-Chen-Helm-Nadler, Hellmann and Zhu. For non-basic $b$,
Conjecture \ref{conj:cohspringermiracle} is a natural $\mathrm{Bun}_{G}$
variant of the conjecture formulated in \cite[Remark 4.6.5]{Zhu}.\footnote{I do not believe Zhu's formulation is quite correct.}
This conjecture turns out to be a magic wand, both for predicting
qualitative properties of the categorical equivalence, and for generating
new conjectures purely on the automorphic side.
\begin{prop}
\label{prop:springermiracleconsequences}Assume Conjectures \ref{conj:categorical-cpsi-unconditional}
and \ref{conj:dualitycpsi}, and also assume that Conjecture \ref{conj:cohspringermiracle}
is true. Then the following hold.

\emph{i. }For all $b$ and $K$, the coherent complex $c_{\psi}\left(i_{b\sharp}^{\mathrm{ren}}\mathrm{ind}_{K}^{G_{b}(E)}\overline{\mathbf{Q}_{\ell}}\right)$
is concentrated in nonnegative degrees.

\emph{ii. }The functor $c_{\psi}$ is perverse right t-exact, i.e.
it carries $\phantom{}^{p}D^{\leq0}(\mathrm{Bun}_{G})$ into $\mathrm{QCoh}^{\leq0}(\mathrm{Par}_{G})$.
In particular, for any $b$ and any $\pi\in\Pi(G_{b})$, $c_{\psi}(i_{b!}^{\mathrm{ren}}\pi)$
is concentrated in nonpositive cohomological degrees.

\emph{iii. }For any $L$-parameter $\phi$ satisfying $(\dagger)$
as in Section 3.1 and any $\rho\in\mathrm{Irr}(S_{\phi})$, the sheaf
$\mathscr{F}_{\phi,\rho}$ defined in Section 3.1 is perverse. Moreover,
if $\phi$ is a discrete parameter then $\mathscr{F}_{\phi,\rho}$
is a \emph{tilting} perverse sheaf: for all $b$, $i_{b}^{\ast\mathrm{ren}}\mathscr{F}_{\phi,\rho}$
and $i_{b}^{!\mathrm{ren}}\mathscr{F}_{\phi,\rho}$ are concentrated
in degree zero.

\emph{iv. }For any local shtuka datum $(G,\mu,b)$ and any open compact
subgroup $K\subset G(E)$, the compactly supported intersection cohomology
\[
R\Gamma_{c}(\mathrm{Sht}(G,\mu,b)_{K},IC_{\mu})
\]
of the local shtuka space $\mathrm{Sht}(G,\mu,b)_{K}$ at level $K$
is concentrated in degrees $[\left\langle 2\rho_{G},\nu_{b}\right\rangle ,\left\langle 2\rho_{G},\mu\right\rangle ]$.

\emph{v. }For all $b$ and $K\subset G_{b}(E)$ as above, the sheaf
$i_{b!}^{\mathrm{ren}}\mathrm{ind}_{K}^{G_{b}(E)}\overline{\mathbf{Q}_{\ell}}$
is perverse, and the sheaf $i_{b\sharp}^{\mathrm{ren}}\mathrm{ind}_{K}^{G_{b}(E)}\overline{\mathbf{Q}_{\ell}}$
is hadal.
\end{prop}

Here $R\Gamma_{c}(\mathrm{Sht}(G,\mu,b)_{K},IC_{\mu})\in D(G_{b}(E),\overline{\mathbf{Q}_{\ell}})^{\omega}$
is defined as in \cite[Lemma 6.4.4]{HKW}, taking $\Lambda=\overline{\mathbf{Q}_{\ell}}$
in the notation there. A priori, the complex $R\Gamma_{c}(\mathrm{Sht}(G,\mu,b)_{K},IC_{\mu})$
lives in degrees $[-\left\langle 2\rho_{G},\mu\right\rangle ,\left\langle 2\rho_{G},\mu\right\rangle ]$.
When $\mu$ is minuscule, the Stein property formulated in \cite[Conjecture 1.10]{H2}
would imply concentration in degrees $[0,\left\langle 2\rho_{G},\mu\right\rangle ]$.
Thus for non-basic $b$, the vanishing result in iv. above goes strictly
beyond what is predicted by geometry. We also observe that i. is best
possible: it is certainly not true in general that $c_{\psi}\left(i_{b\sharp}^{\mathrm{ren}}\mathrm{ind}_{K}^{G_{b}(E)}\overline{\mathbf{Q}_{\ell}}\right)$
is concentrated in degree zero, or in any single degree.
\begin{proof}
For i., we have a duality isomorphism
\[
\mathbf{D}_{\mathrm{BZ}}\left(i_{b!}^{\mathrm{ren}}\mathrm{ind}_{K}^{G_{b}(E)}\overline{\mathbf{Q}_{\ell}}\right)\simeq i_{b\sharp}^{\mathrm{ren}}\mathrm{ind}_{K}^{G_{b}(E)}\overline{\mathbf{Q}_{\ell}}.
\]
Combining this with the duality compatibility gives
\[
c_{\psi^{-1}}\left(i_{b\sharp}^{\mathrm{ren}}\mathrm{ind}_{K}^{G_{b}(E)}\overline{\mathbf{Q}_{\ell}}\right)\simeq\mathbf{D}_{\mathrm{tw.GS}}\left(\mathscr{C}_{b,K}\right),
\]
and the right-hand side is clearly concentrated in nonnegative degrees.
For ii., the claim follows by observing that $\phantom{}^{p}D^{\leq0}(\mathrm{Bun}_{G})$
is generated under extensions and colimits by objects of the form
$i_{b!}^{\mathrm{ren}}\mathrm{ind}_{K}^{G_{b}(E)}\overline{\mathbf{Q}_{\ell}}[n]$
for arbitrary $b,K$ and arbitrary $n\geq0$.

For iii. it is enough to prove that $\mathscr{F}_{\phi,\rho}$ is
perverse coconnective, and that $i_{b}^{!\mathrm{ren}}\mathscr{F}_{\phi,\rho}$
is concentrated in degree zero when $\phi$ is discrete. These reductions
follow from the Verdier duality property of $\mathscr{F}_{\phi,\rho}$
implied by Conjecture \ref{conj:ULA-duality}.ii as in the discussion
of section 3.1. In turn, Conjecture \ref{conj:ULA-duality}.ii is
an \emph{unconditional} consequence of Conjectures \ref{conj:categorical-cpsi-unconditional}
and \ref{conj:dualitycpsi}; the argument for this will appear in
\cite{HM}. Now by design, Conjecture \ref{conj:categorical-cpsi-unconditional}
implies that
\[
R\mathrm{Hom}(\mathscr{C}_{b,K},\mathbf{L}_{\psi}(-)):D(\mathrm{Bun}_{G})^{\mathrm{}}\to D(\overline{\mathbf{Q}_{\ell}})
\]
is exactly the functor $A\mapsto(i_{b}^{!\mathrm{ren}}A)^{K}$, so
we compute that
\begin{align*}
(i_{b}^{!\mathrm{ren}}\mathscr{F}_{\phi,\rho})^{K} & \simeq R\mathrm{Hom}(\mathscr{C}_{b,K},\mathbf{L}_{\psi}(\mathscr{F}_{\phi,\rho}))\\
 & \simeq R\mathrm{Hom}(\mathscr{C}_{b,K},i_{\phi\ast}^{\mathrm{IndCoh}}\rho)\\
 & \simeq R\mathrm{Hom}(i_{\phi}^{\ast\mathrm{IndCoh}}\mathscr{C}_{b,K},\rho).
\end{align*}
Since $\mathscr{C}_{b,K}$ is in degree zero by assumption, $i_{\phi}^{\ast\mathrm{IndCoh}}\mathscr{C}_{b,K}$
is concentrated in nonpositive degrees, so we immediately get that
the above expression is concentrated in nonnegative degrees. Moreover,
if $\phi$ is discrete, one can check that $i_{\phi}^{\ast\mathrm{IndCoh}}\mathscr{C}_{b,K}$
is necessarily concentrated in degree zero, which implies that the
above expression is concentrated also in degree zero. If $G$ is semisimple
the concentration of $i_{\phi}^{\ast\mathrm{IndCoh}}\mathscr{C}_{b,K}$
in degree zero is obvious since for a discrete parameter $\phi$ the
morphism $i_{\phi}$ is an open immersion.\footnote{For semisimple groups, the morphism $i_{\phi}$ is an open immersion
exactly when $\phi$ is a discrete parameter.} For general groups one needs a small extra argument, using that the
orbit of any discrete parameter $\phi$ under unramified twisting
is open in $\mathrm{Par}_{G}$ together with the equivariance of $\mathscr{C}_{b,K}$
under unramified twisting.

For iv., concentration in degrees $\leq\left\langle 2\rho_{G},\mu\right\rangle $
follows from the general \'etale cohomology formalism. For the bound
in the other direction, the key point is the formula
\[
R\Gamma_{c}(\mathrm{Sht}(G,\mu,b)_{K},IC_{\mu})^{K'}\simeq R\mathrm{Hom}\left(\mathscr{C}_{1,K}\otimes V_{\mu},\mathscr{C}_{b,K'}\right)[-\left\langle 2\rho_{G},\nu_{b}\right\rangle ]
\]
where $K'\subset G_{b}(E)$ is any open compact subgroup.\footnote{This is a variant of \cite[Conjecture 4.7.18]{Zhu}, but again I do
not think Zhu's formulation is correct.} This formula is an unconditional consequence of Conjecture \ref{conj:categorical-cpsi-unconditional}.
Indeed, \cite[Lemma 6.4.4]{HKW} gives an isomorphism
\[
R\Gamma_{c}(\mathrm{Sht}(G,\mu,b)_{K},IC_{\mu})\simeq i_{b}^{\ast}T_{V_{\mu}}i_{1!}\mathrm{ind}_{K}^{G(E)}\overline{\mathbf{Q}_{\ell}},
\]
which easily implies an isomorphism
\[
R\Gamma_{c}(\mathrm{Sht}(G,\mu,b)_{K},IC_{\mu})^{K'}\simeq R\mathrm{Hom}(i_{b\sharp}\mathrm{ind}_{K'}^{G_{b}(E)}\overline{\mathbf{Q}_{\ell}},T_{V_{\mu}}i_{1!}\mathrm{ind}_{K}^{G(E)}\overline{\mathbf{Q}_{\ell}}).
\]
Applying BZ duality on the right-hand side, we can rewrite it as 
\[
R\mathrm{Hom}(T_{V_{\mu}}i_{1!}\mathrm{ind}_{K}^{G(E)}\overline{\mathbf{Q}_{\ell}},i_{b!}\mathrm{ind}_{K'}^{G_{b}(E)}\overline{\mathbf{Q}_{\ell}}[-2\left\langle 2\rho_{G},\nu_{b}\right\rangle ]).
\]
Now we have isomorphisms
\[
c_{\psi}\left(T_{V_{\mu}}i_{1!}\mathrm{ind}_{K}^{G(E)}\overline{\mathbf{Q}_{\ell}}\right)\simeq V_{\mu}\otimes\mathscr{C}_{1,K}
\]
and
\[
c_{\psi}\left(i_{b!}\mathrm{ind}_{K'}^{G_{b}(E)}\overline{\mathbf{Q}_{\ell}}[-2\left\langle 2\rho_{G},\nu_{b}\right\rangle ]\right)\simeq\mathscr{C}_{b,K'}[-\left\langle 2\rho_{G},\nu_{b}\right\rangle ]
\]
by unwinding the definitions and using the linearity of $c_{\psi}$
over the spectral action, so passing to the other side via Conjecture
\ref{conj:categorical-cpsi-unconditional} we arrive at the desired
formula. Then by Conjecture \ref{conj:cohspringermiracle}, $R\mathrm{Hom}\left(\mathscr{C}_{1,K}\otimes V_{\mu},\mathscr{C}_{b,K'}\right)$
is concentrated in nonnegative degrees, so accounting for the shift
and shrinking $K'$ arbitrarily gives the claim.

For v., we prove the second claim, the first being similar (and easier).
It is clear that $i_{b\sharp}^{\mathrm{ren}}\mathrm{ind}_{K}^{G_{b}(E)}\overline{\mathbf{Q}_{\ell}}$
is connective for the hadal t-structure. To show it is coconnective
is equivalent to showing that all stalks
\[
i_{b'}^{\ast\mathrm{ren}}i_{b\sharp}^{\mathrm{ren}}\mathrm{ind}_{K}^{G_{b}(E)}\overline{\mathbf{Q}_{\ell}}
\]
are concentrated in nonnegative degrees. But now for any open compact
$K'\subset G_{b'}(E)$ we have isomorphisms
\begin{align*}
\left(i_{b'}^{\ast\mathrm{ren}}i_{b\sharp}^{\mathrm{ren}}\mathrm{ind}_{K}^{G_{b}(E)}\overline{\mathbf{Q}_{\ell}}\right)^{K'} & \simeq R\mathrm{Hom}\left(i_{b'\sharp}^{\mathrm{ren}}\mathrm{ind}_{K'}^{G_{b'}(E)}\overline{\mathbf{Q}_{\ell}},i_{b\sharp}^{\mathrm{ren}}\mathrm{ind}_{K}^{G_{b}(E)}\overline{\mathbf{Q}_{\ell}}\right)\\
 & \simeq R\mathrm{Hom}\left(i_{b!}^{\mathrm{ren}}\mathrm{ind}_{K}^{G_{b}(E)}\overline{\mathbf{Q}_{\ell}},i_{b'!}^{\mathrm{ren}}\mathrm{ind}_{K'}^{G_{b'}(E)}\overline{\mathbf{Q}_{\ell}}\right)\\
 & \simeq R\mathrm{Hom}(\mathscr{C}_{b,K},\mathscr{C}_{b,K'})
\end{align*}
where the first line follows from the obvious adjunctions, the second
line follows from BZ duality, and the third line follows from Conjecture
\ref{conj:categorical-cpsi-unconditional} and the definition of $\mathscr{C}_{b,K}$.
Then $R\mathrm{Hom}(\mathscr{C}_{b,K},\mathscr{C}_{b,K'})$ is concentrated
in nonnegative degrees by Conjecture \ref{conj:cohspringermiracle},
so shrinking $K'$ arbitrarily gives the desired coconnectivity.
\end{proof}
\begin{rem}
In fact, the sheaves $i_{b\sharp}^{\mathrm{ren}}\mathrm{ind}_{K}^{G_{b}(E)}\overline{\mathbf{Q}_{\ell}}$
are hadal for all $b,K$ if and only if the sheaves $i_{b!}^{\mathrm{ren}}\mathrm{ind}_{K}^{G_{b}(E)}\overline{\mathbf{Q}_{\ell}}$
are perverse for all $b,K$. Since both are clearly connective in
the relevant t-structure, this follows from the chain of identities
\begin{align*}
\left(i_{b'}^{\ast\mathrm{ren}}i_{b\sharp}^{\mathrm{ren}}\mathrm{ind}_{K}^{G_{b}(E)}\overline{\mathbf{Q}_{\ell}}\right)^{K'} & \simeq R\mathrm{Hom}\left(i_{b'\sharp}^{\mathrm{ren}}\mathrm{ind}_{K'}^{G_{b'}(E)}\overline{\mathbf{Q}_{\ell}},i_{b\sharp}^{\mathrm{ren}}\mathrm{ind}_{K}^{G_{b}(E)}\overline{\mathbf{Q}_{\ell}}\right)\\
 & \simeq R\mathrm{Hom}\left(i_{b!}^{\mathrm{ren}}\mathrm{ind}_{K}^{G_{b}(E)}\overline{\mathbf{Q}_{\ell}},i_{b'!}^{\mathrm{ren}}\mathrm{ind}_{K'}^{G_{b'}(E)}\overline{\mathbf{Q}_{\ell}}\right)\\
 & \simeq\left(i_{b}^{!\mathrm{ren}}i_{b'!}^{\mathrm{ren}}\mathrm{ind}_{K'}^{G_{b'}(E)}\overline{\mathbf{Q}_{\ell}}\right)^{K}
\end{align*}
arguing as in the proof of \ref{prop:springermiracleconsequences}.v.
\end{rem}

The conjectural hadal property of the sheaves $i_{b\sharp}^{\mathrm{ren}}\mathrm{ind}_{K}^{G_{b}(E)}\overline{\mathbf{Q}_{\ell}}$
admits the following very concrete reinterpretation. Let $\widetilde{\mathcal{M}}_{b}\to\mathcal{M}_{b}$
be the canonical $G_{b}(E)$-torsor over the local chart $\mathcal{M}_{b}$
as constructed in Fargues-Scholze. For any $b'\preceq b$, pick a
complete algebraically closed field $C/\mathbf{F}_{p}$ and a map
$\mathrm{Spd}C\to\mathrm{Bun}_{G}$ covering the stratum $\mathrm{Bun}_{G}^{b'}$,
and define $X_{b,b'}$ by the pullback diagram
\[
\xymatrix{X_{b,b'}\ar[d]\ar[r] & \widetilde{\mathcal{M}}_{b}\ar[d]\\
\mathrm{Spd}C\ar[r] & \mathrm{Bun}_{G}
}
\]
of small v-stacks. Then $X_{b,b'}$ is a partially proper locally
spatial diamond over $\mathrm{Spd}C$ of $\ell$-cohomological dimension
$\left\langle 2\rho_{G},\nu_{b}\right\rangle $, and a priori $R\Gamma_{c}(X_{b,b'},\mathbf{F}_{\ell})$
is concentrated in degrees $[0,2\left\langle 2\rho_{G},\nu_{b}\right\rangle ]$.
\begin{conjecture}
\label{conj:hadalgenerators}For all $b'\preceq b$, $R\Gamma_{c}(X_{b,b'},\mathbf{F}_{\ell})$
is concentrated in degrees $[\left\langle 2\rho_{G},\nu_{b}+\nu_{b'}\right\rangle ,2\left\langle 2\rho_{G},\nu_{b}\right\rangle ]$.
\end{conjecture}

If this conjecture holds for $b$ fixed and all $b'\preceq b$, then
$i_{b\sharp}^{\mathrm{ren}}\mathrm{ind}_{K}^{G_{b}(E)}\overline{\mathbf{Q}_{\ell}}$
is hadal for all $K\subset G_{b}(E)$.\footnote{It is plausible that for some small primes $\ell$, Conjecture \ref{conj:hadalgenerators}
is not true as stated. However, for the intended application, it would
be enough to prove the weaker conjecture that for any fixed $i<\left\langle 2\rho_{G},\nu_{b}+\nu_{b'}\right\rangle $,
the group $H_{c}^{i}(X_{b,b'},\mathbf{Z}/\ell^{n}\mathbf{Z})$ is
killed by an integer independent of $n$. } The essential point here is the formula
\[
\mathrm{colim}_{K\to\{1\}}i_{b'}^{\ast\mathrm{ren}}i_{b\sharp}^{\mathrm{ren}}\mathrm{ind}_{K}^{G_{b}(E)}\overline{\mathbf{Q}_{\ell}}\cong\left(R\Gamma_{c}(X_{b,b'},\mathbf{Z}_{\ell})[\left\langle 2\rho_{G},\nu_{b}+\nu_{b'}\right\rangle ]^{G_{b'}(E)\times G_{b}(E)-\mathrm{sm}}\right)\otimes_{\mathbf{Z}_{\ell}}\overline{\mathbf{Q}_{\ell}}
\]
where the superscript $(-)^{G_{b'}(E)\times G_{b}(E)-\mathrm{sm}}$
indicates the (exact) functor of $G_{b'}(E)\times G_{b}(E)$-smooth
vectors. This conjecture already has nontrivial content in the degenerate
case $b=b'$, where it reduces to the fact that the $\mathbf{F}_{\ell}$-\'etale
cohomology of $\tilde{G}_{b,C}^{>0}$ (in the notation of \cite[Proposition III.5.1]{FS})
is entirely concentrated in degree $2\left\langle 2\rho_{G},\nu_{b}\right\rangle $.
When $G=\mathrm{GL}_{2},$ $\mathcal{E}_{b}=\mathcal{O}(1)\oplus\mathcal{O}(-1)$,
and $\mathcal{E}_{b'}=\mathcal{O}^{2}$, Conjecture \ref{conj:hadalgenerators}
has been checked by Miles.

We can squeeze quite a bit more out of Conjecture \ref{conj:cohspringermiracle}.
\begin{conjecture}
Assume that $b$ is basic and $\pi\in\Pi(G_{b})$ is supercuspidal,
with Fargues-Scholze parameter $\phi$. Let $\nu_{\phi}:V=q^{-1}(x_{\phi})\to\mathrm{Par}_{G}$
be the natural closed immersion of the fiber over $x_{\phi}$. Then
$c_{\psi}(i_{b!}\pi)$ is a genuine coherent sheaf, which is moreover
of the form $\nu_{\phi\ast}\mathcal{F}$ for some $\mathcal{F}\in\mathrm{Coh}(V)^{\heartsuit}$.

Assuming moreover that the fiber $V$ is smooth, then $V$ is $-\mathrm{dim}Z(\hat{G})^{\Gamma}$-dimensional
and $\mathcal{F}$ is a locally free sheaf.
\end{conjecture}

Note that $-\mathrm{dim}Z(\hat{G})^{\Gamma}$ is the maximal dimension
of any fiber of $q$, and a fiber achieves this dimension if and only
if it contains a discrete \emph{L}-parameter, which we expect is automatic
under the assumptions of this conjecture. Indeed, the fiber $V$ should
contain the true \emph{L}-parameter of $\pi$, which we expect is
always a discrete parameter. However, we do not always expect the
fiber $V$ to be smooth under the assumptions of this conjecture,
even when $G_{b}$ is split.\footnote{I thank Wee Teck Gan for showing me a counterexample on $G_{2}$.
In brief, the (Weil-Deligne incarnation of the) \emph{L}-parameter
$\phi:W_{E}\times\mathrm{SL}_{2}\to G_{2}$ which is trivial on $W_{E}$
and embeds $\left(\begin{array}{cc}
1 & 1\\
0 & 1
\end{array}\right)\in\mathrm{SL}_{2}$ as a subregular unipotent element is discrete, and the packet $\Pi_{\phi}(G_{2})$
contains a supercuspidal representation $\pi_{3}^{\varepsilon}$,
but the fiber of $q$ over $x_{\phi^{\mathrm{ss}}}$ is not smooth.
It would be very interesting to give a conjectural description of
the coherent sheaf $c_{\psi}(i_{1!}\pi_{3}^{\varepsilon})$.}

Here is a heuristic argument in support of this conjecture. For simplicity,
we assume $G$ is semisimple, so $-\mathrm{dim}Z(\hat{G})^{\Gamma}=0$
and $\mathbf{D}_{\mathrm{coh}}(\pi)\simeq\pi^{\vee}$. By some general
nonsense (using e.g. \cite[Theorem 2]{Bushnell}), any supercuspidal
$\pi$ will occur as a summand of $\mathrm{ind}_{K}^{G_{b}(E)}\overline{\mathbf{Q}_{\ell}}$
for any sufficiently small $K$, so $c_{\psi}(i_{b!}\pi)$ is a summand
of the corresponding $\mathscr{C}_{b,K}$, which is concentrated in
degree zero by Conjecture \ref{conj:cohspringermiracle}. Thus $c_{\psi}(i_{b!}\pi)$
is a genuine coherent sheaf. By the compatibility of $c_{\psi}$ with
the action of the spectral Bernstein center, $c_{\psi}(i_{b!}\pi)$
is also killed by the maximal ideal $\mathfrak{m}_{\phi}\subset\mathcal{O}(X_{G}^{\mathrm{spec}})$
whose preimage under $q$ cuts out $V$, so it is supported scheme-theoretically
on $V$. This gives the first part. For the second part, Conjecture
\ref{conj:dualitycpsi} predicts that $\mathbf{D}_{\mathrm{tw.GS}}c_{\psi}(i_{b!}\pi)\simeq c_{\psi^{-1}}(i_{b!}\pi^{\vee})$
is also concentrated in degree zero, so the (untwisted) Grothendieck-Serre
dual of $\nu_{\phi\ast}\mathcal{F}$ is also concentrated in degree
zero. Since Grothendieck-Serre duality commutes with proper pushforward,
this implies that $R\mathscr{H}\mathrm{om}(\mathcal{F},\omega_{V})$
is also concentrated in degree zero. Note that $V$ is closed in the
zero-dimensional stack $\mathrm{Par}_{G}$, so it is necessarily of
dimension $-d\leq0$. Since it is moreover smooth by assumption, the
dualizing complex is of the form $\omega_{V}\simeq\mathcal{L}[-d]$
for some line bundle $\mathcal{L}$ on $V$. But then $R\mathscr{H}\mathrm{om}(\mathcal{F},\omega_{V})\simeq R\mathscr{H}\mathrm{om}(\mathcal{F},\mathcal{L})[-d]$
is automatically concentrated in degrees $\geq d$, so we must have
$d=0$. We then see that $\mathcal{F}$ is maximal Cohen-Macaulay
on the smooth zero-dimensional algebraic stack $V$, and this implies
that $\mathcal{F}$ is locally free.

A similar heuristic also leads to the following expectation.
\begin{conjecture}
Assume that $b$ is basic and $\pi\in\Pi(G_{b})$ is supercuspidal.
Then $i_{b!}\pi$ and $i_{b\ast}\pi$ are perverse.
\end{conjecture}

Indeed, assume $G$ is semisimple for simplicity. Then $i_{b!}\pi$
is a summand of $i_{b!}\mathrm{ind}_{K}^{G_{b}(E)}\overline{\mathbf{Q}_{\ell}}$
for any sufficiently small $K$, and Proposition \ref{prop:springermiracleconsequences}.v
suggests that $i_{b!}\mathrm{ind}_{K}^{G_{b}(E)}\overline{\mathbf{Q}_{\ell}}$
is always perverse, so also $i_{b!}\pi$ should be perverse. Then
$i_{b\ast}\pi\simeq\mathbf{D}_{\mathrm{Verd}}i_{b!}\pi^{\vee}$ should
also be perverse, using the Verdier self-duality of the perverse t-structure.

We emphasize that $!$-extensions of irreducible representations from
basic strata need not be perverse in general. For instance, if $G=\mathrm{SL}_{2}$,
the discussion after Conjecture \ref{conj:categorical-fin-ssgen-perverse-version}
implies that $i_{1!}\overline{\mathbf{Q}_{\ell}}$ has a nonvanishing
perverse cohomology sheaf in some degree $\leq-1$.

Next, we sketch some consequences of Conjecture \ref{conj:cohspringermiracle}
for the categorical conjecture over the semisimple generic locus;
the full details of these arguments will appear in \cite{HM}. To
set things up, note that the equivalence predicted in Proposition
\ref{prop:categorical-conj-ULA} localizes to an equivalence
\[
\mathbf{L}_{\psi}:D(\mathrm{Bun}_{G})^{\mathrm{ULA,gen}}\overset{\sim}{\to}\mathrm{Adm}(\mathrm{Par}_{G}^{\mathrm{gen}})
\]
where both categories appearing here are the evident localizations
of the analogous category appearing in Proposition \ref{prop:categorical-conj-ULA}.
Note that $\mathrm{Par}_{G}^{\mathrm{gen}}$ is smooth, so $\mathrm{IndCoh}(\mathrm{Par}_{G}^{\mathrm{gen}})\cong\mathrm{QCoh}(\mathrm{Par}_{G}^{\mathrm{gen}})$.
\begin{lem}
\label{lem:adm-sheaves-standard-tstructure}The standard t-structure
on $\mathrm{QCoh}(\mathrm{Par}_{G}^{\mathrm{gen}})$ restricts to
a (``standard'') t-structure on $\mathrm{Adm}(\mathrm{Par}_{G}^{\mathrm{gen}})$,
whose left and right halves are exchanged by $\mathbf{D}_{\mathrm{tw.adm}}$.
\end{lem}

This is a general feature of admissible ind-coherent sheaves on algebraic
stacks of the form $[X/H]$ where $X$ is a smooth affine variety
and $H$ is a reductive group. It turns out that every connected component
of $\mathrm{Par}_{G}^{\mathrm{gen}}$ has this structure.
\begin{prop}
\label{prop:texactssgenspringer}Assume Conjectures \ref{conj:categorical-cpsi-unconditional}
and \ref{conj:dualitycpsi}, and also assume that Conjecture \ref{conj:cohspringermiracle}
is true. Then the equivalence $D(\mathrm{Bun}_{G})^{\mathrm{ULA,gen}}\overset{\sim}{\to}\mathrm{Adm}(\mathrm{Par}_{G}^{\mathrm{gen}})$
is t-exact with respect to the perverse t-structure on $D(\mathrm{Bun}_{G})^{\mathrm{ULA,gen}}$
and the standard t-structure on $\mathrm{Adm}(\mathrm{Par}_{G}^{\mathrm{gen}})$.
\end{prop}

We remark that the equivalence appearing here restricts further to
the equivalence $D(\mathrm{Bun}_{G})_{\mathrm{fin}}^{\mathrm{gen}}\overset{\sim}{\to}\mathrm{Coh}(\mathrm{Par}_{G}^{\mathrm{gen}})_{\mathrm{fin}}$
discussed in section 2.5, and the standard t-structure on $\mathrm{Adm}(\mathrm{Par}_{G}^{\mathrm{gen}})$
clearly restricts to the standard t-structure on $\mathrm{Coh}(\mathrm{Par}_{G}^{\mathrm{gen}})_{\mathrm{fin}}$.
From here, it is easy to see that the assumptions in the previous
proposition actually imply Conjecture \ref{conj:categorical-fin-ssgen-perverse-version}.
Since tensoring with a vector bundle is t-exact for the standard t-structure
on $\mathrm{Adm}(\mathrm{Par}_{G}^{\mathrm{gen}})$, we also see that
Conjecture \ref{conj:t-exact-Hecke-ssgen} is implied by the same
set of assumptions.
\begin{proof}[Sketch]
By our assumptions, Proposition \ref{prop:springermiracleconsequences}.ii
implies that $\mathbf{L}_{\psi}$ carries $\phantom{}^{p}D^{\leq0}(\mathrm{Bun}_{G})^{\mathrm{ULA,gen}}$
fully faithfully into $\mathrm{Adm}(\mathrm{Par}_{G}^{\mathrm{gen}})$
with image contained inside $\mathrm{Adm}(\mathrm{Par}_{G}^{\mathrm{gen}})\cap\mathrm{QCoh}^{\leq0}$.
Passing to right orthogonals in $D(\mathrm{Bun}_{G})^{\mathrm{ULA,gen}}$
resp. in $\mathrm{Adm}(\mathrm{Par}_{G}^{\mathrm{gen}})$ and swapping
the Whittaker data, we get a containment
\[
\mathrm{Adm}(\mathrm{Par}_{G}^{\mathrm{gen}})\cap\mathrm{QCoh}^{\geq0}\subseteq\mathbf{L}_{\psi^{-1}}\left(\phantom{}^{p}D^{\geq0}(\mathrm{Bun}_{G})^{\mathrm{ULA,gen}}\right).
\]
Applying $\mathbf{D}_{\mathrm{tw.adm}}$ on both sides and using the
previous lemma together with the admissible duality compatibility
of $\mathbf{L}_{\psi}$ (which follows from our assumptions), we get
a containment
\[
\mathrm{Adm}(\mathrm{Par}_{G}^{\mathrm{gen}})\cap\mathrm{QCoh}^{\leq0}\subseteq\mathbf{L}_{\psi}\left(\phantom{}^{p}D^{\leq0}(\mathrm{Bun}_{G})^{\mathrm{ULA,gen}}\right).
\]
But we already know that the left side contains the right, so putting
things together we see that $\mathbf{L}_{\psi}$ induces an equivalence
\[
\phantom{}^{p}D^{\leq0}(\mathrm{Bun}_{G})^{\mathrm{ULA,gen}}\overset{\sim}{\to}\mathrm{Adm}(\mathrm{Par}_{G}^{\mathrm{gen}})\cap\mathrm{QCoh}^{\leq0}.
\]
Passing to right orthogonals again we see that it also induces an
equivalence
\[
\phantom{}^{p}D^{\geq0}(\mathrm{Bun}_{G})^{\mathrm{ULA,gen}}\overset{\sim}{\to}\mathrm{Adm}(\mathrm{Par}_{G}^{\mathrm{gen}})\cap\mathrm{QCoh}^{\geq0}.
\]
Therefore the t-structures are compatible, as desired.
\end{proof}
Pushing these ideas in a slightly different direction, we can also
say something about the long-neglected functor $a_{\psi}$, which
also yields a substantial upgrade to Proposition \ref{prop:springermiracleconsequences}.ii.\footnote{I again thank Koshikawa for very interesting discussions on these
matters, and for suggesting the key trick in the next proof.}
\begin{prop}
\label{prop:temperedtexact}Assume Conjectures \ref{conj:categorical-cpsi-unconditional}
and \ref{conj:dualitycpsi}, and also assume that Conjecture \ref{conj:cohspringermiracle}
is true. Then $a_{\psi}:\mathrm{QCoh}(\mathrm{Par}_{G})\to D(\mathrm{Bun}_{G})$
is perverse right t-exact, i.e. it carries $\mathrm{QCoh}^{\leq0}$
into $\phantom{}^{p}D^{\leq0}$.

Moreover, $c_{\psi}:D(\mathrm{Bun}_{G})\to\mathrm{QCoh}(\mathrm{Par}_{G})$
is t-exact for the perverse t-structure on the source and the standard
t-structure on the target.
\end{prop}

From here it is easy to give yet another conditional justification
for Conjectures \ref{conj:t-exact-Hecke-ssgen} and \ref{conj:categorical-fin-ssgen-perverse-version}.
\begin{proof}
It suffices to prove the first part. Indeed, we already know from
Proposition \ref{prop:springermiracleconsequences}.ii that $c_{\psi}$
is perverse right t-exact under the same set of assumptions. But if
$a_{\psi}$ is also perverse right t-exact, then its right adjoint
$c_{\psi}$ is automatically perverse left t-exact as well.

For the first part, we easily reduce to the claim that $a_{\psi}$
carries any $\mathcal{F}\in\mathrm{Perf^{qc}}\cap\mathrm{QCoh}^{\leq0}$
into $\phantom{}^{p}D^{\leq0}$. Fix such an $\mathcal{F}$. Then
Conjecture \ref{conj:categorical-cpsi-unconditional} implies that
$a_{\psi}$ is fully faithful and that $a_{\psi}(\mathcal{F})$ is
compact, so the natural adjunction gives an isomorphism $\mathcal{F}\simeq c_{\psi}(a_{\psi}(\mathcal{F}))$.
Now fix $b$ and $K\subset G_{b}(E)$, and set $\mathscr{C}_{b,K}'=c_{\psi^{-1}}(i_{b!}^{\mathrm{ren}}\mathrm{ind}_{K}^{G_{b}(E)}\overline{\mathbf{Q}_{\ell}})$;
note the change in Whittaker datum. Then Conjectures \ref{conj:categorical-cpsi-unconditional}
and \ref{conj:dualitycpsi} together imply that 
\begin{align*}
(i_{b}^{\ast\mathrm{ren}}A)^{K} & \simeq R\mathrm{Hom}\left(i_{b\sharp}^{\mathrm{ren}}\mathrm{ind}_{K}^{G_{b}(E)}\overline{\mathbf{Q}_{\ell}},A\right)\\
 & \simeq R\mathrm{Hom}\left(c_{\psi}\left(i_{b\sharp}^{\mathrm{ren}}\mathrm{ind}_{K}^{G_{b}(E)}\overline{\mathbf{Q}_{\ell}}\right),c_{\psi}(A)\right)\\
 & \simeq R\mathrm{Hom}(\mathbf{D}_{\mathrm{tw.GS}}(\mathscr{C}_{b,K}'),c_{\psi}(A))
\end{align*}
for all compact sheaves $A\in D(\mathrm{Bun}_{G})$. Applying this
with $A=a_{\psi}(\mathcal{F})$ and using the full faithfulness of
$a_{\psi}$ and the compactness of $A$, we get
\begin{align*}
(i_{b}^{\ast\mathrm{ren}}a_{\psi}(\mathcal{F}))^{K} & \simeq R\mathrm{Hom}(\mathbf{D}_{\mathrm{tw.GS}}(\mathscr{C}_{b,K}'),\mathcal{F})\\
 & \simeq R\Gamma\left(\mathrm{Par}_{G},R\mathscr{H}\mathrm{om}(\mathbf{D}_{\mathrm{tw.GS}}(\mathscr{C}_{b,K}'),\mathcal{F})\right).
\end{align*}
Now the key trick, which was suggested to me by Koshikawa, is that
because $\mathcal{F}$ is perfect by assumption, we can rewrite the
internal hom appearing here as
\begin{align*}
R\mathscr{H}\mathrm{om}(\mathbf{D}_{\mathrm{tw.GS}}(\mathscr{C}_{b,K}'),\mathcal{F}) & \simeq R\mathscr{H}\mathrm{om}(\mathbf{D}_{\mathrm{tw.GS}}(\mathscr{C}_{b,K}'),\mathcal{O}_{\mathrm{Par}_{G}})\otimes\mathcal{F}.
\end{align*}
Moreover, by the involutivity of Grothendieck-Serre duality and the
definition of $\mathbf{D}_{\mathrm{twGS}}$, we easily get an isomorphism
$R\mathscr{H}\mathrm{om}(\mathbf{D}_{\mathrm{tw.GS}}(\mathscr{C}_{b,K}'),\mathcal{O}_{\mathrm{Par}_{G}})\simeq c^{\ast}\mathscr{C}_{b,K}'$.
Combining all of our observations so far, we get an isomorphism
\[
(i_{b}^{\ast\mathrm{ren}}a_{\psi}(\mathcal{F}))^{K}\simeq R\Gamma(\mathrm{Par}_{G},c^{\ast}\mathscr{C}_{b,K}'\otimes\mathcal{F}).
\]
Now we win: $\mathcal{F}$ is a perfect complex concentrated in nonpositive
degrees by assumption, and $c^{\ast}\mathscr{C}_{b,K}'$ is concentrated
in degree zero by Conjecture \ref{conj:cohspringermiracle}, so $c^{\ast}\mathscr{C}_{b,K}'\otimes\mathcal{F}$
is a bounded coherent complex with quasicompact support concentrated
in nonpositive degrees. Since each connected component of $\mathrm{Par}_{G}$
is a quotient of an affine variety by a reductive group, the functor
$R\Gamma(\mathrm{Par}_{G},-):\mathrm{Coh}(\mathrm{Par}_{G})\to D(\overline{\mathbf{Q}_{\ell}})$
is t-exact, so we conclude that
\[
(i_{b}^{\ast\mathrm{ren}}a_{\psi}(\mathcal{F}))^{K}\simeq R\Gamma(\mathrm{Par}_{G},c^{\ast}\mathscr{C}_{b,K}'\otimes\mathcal{F})
\]
is concentrated in nonpositive degrees. Varying $b$ and $K$ arbitrarily,
we conclude that $a_{\psi}(\mathcal{F})$ is perverse connective as
desired.
\end{proof}
\begin{rem}
Proposition \ref{prop:temperedtexact} is in perfect accord with the
t-exactness results proved in \cite{FR}. More precisely, we expect
that there is an intrinsically defined subcategory $D(\mathrm{Bun}_{G})^{\mathrm{temp}}\subset D(\mathrm{Bun}_{G})$
stable under colimits and Hecke operators, and coinciding with the
image of $\mathrm{QCoh}(\mathrm{Par}_{G})$ under $a_{\psi}$. Since
$a_{\psi}$ supposedly matches with the fully faithful embedding $\Xi:\mathrm{QCoh}\to\mathrm{IndCoh}$
under the categorical equivalence, passing to right adjoints implies
there should be a commutative diagram
\[
\xymatrix{D(\mathrm{Bun}_{G})\ar[d]^{\mathrm{pr}}\ar[r]^{\sim}\ar[dr]^{c_{\psi}} & \mathrm{IndCoh}(\mathrm{Par}_{G})\ar[d]^{\Psi}\\
D(\mathrm{Bun}_{G})^{\mathrm{temp}}\ar[r]^{\sim} & \mathrm{QCoh}(\mathrm{Par}_{G})
}
\]
where $\mathrm{pr}:D(\mathrm{Bun}_{G})\to D(\mathrm{Bun}_{G})^{\mathrm{temp}}$
is the right adjoint of the evident inclusion and the lower horizontal
arrow is the essential inverse of $a_{\psi}$. Transporting the t-exactness
result of the previous proposition across the lower horizontal equivalence,
we see that there should be an intrinsic t-structure on $D(\mathrm{Bun}_{G})^{\mathrm{temp}}$
compatible with the perverse t-structure on $D(\mathrm{Bun}_{G})$
via $\mathrm{pr}$. In the setting of classical geometric Langlands,
such a t-structure is constructed in \cite{FR}.
\end{rem}

We end this section by noting that while the coherence of $\mathscr{C}_{b,K}$
seems to lie very deep, there is a reasonably explicit criterion for
these gadgets to be genuine quasicoherent sheaves.
\begin{prop}
Fix $G$ and $b\in B(G)$. Then the following are equivalent.

\emph{i. }For all $K\subset G_{b}(E)$ open compact, $\mathscr{C}_{b,K}$
lies in the heart $\mathrm{QCoh}^{\heartsuit}$ of the standard t-structure
on $\mathrm{QCoh}(\mathrm{Par}_{G})$.

\emph{ii. }For all $V\in\mathrm{Rep}\,\hat{G},$ $i_{b}^{\ast}T_{V}i_{1!}W_{\psi^{-1}}$
is zero outside degree $\left\langle 2\rho_{G},\nu_{b}\right\rangle $.
\end{prop}

We leave the proof as an exercise for the interested reader. The essential
point is that the defining property of $c_{\psi}$ leads to the (unconditional!)
formula
\[
R\Gamma(\mathrm{Par}_{G},V\otimes\mathscr{C}_{b,K})\cong\left(i_{b}^{\ast\mathrm{ren}}T_{V^{\vee}}i_{1!}W_{\psi^{-1}}\right)^{K}
\]
for any $V\in\mathrm{Rep}\,\hat{G}$. One then uses the fact that
the functors 
\[
R\Gamma(\mathrm{Par}_{G},V\otimes-):\mathrm{QCoh^{qc}}(\mathrm{Par}_{G})\to D(\overline{\mathbf{Q}_{\ell}})
\]
form a \emph{t-exact} conservative family as one varies over all $V\in\mathrm{Rep}\,\hat{G}$.
\begin{xca}
Assume Conjectures \ref{conj:categorical-cpsi-unconditional} and
\ref{conj:cohspringermiracle}. Deduce that $T_{V}i_{1!}W_{\psi}$
is perverse for all $V\in\mathrm{Rep}\,\hat{G}$.
\end{xca}

\subsection{More examples}

In this section we collect some miscellaneous remarks and examples.
We also try to illustrate how complicated the situation becomes once
we allow \emph{L}-parameters with nontrivial monodromy into the discussion.
\begin{example}
It is natural to expect that $c_{\psi}\circ i_{1!}$ should coincide
with Hellmann's functor $R_{\psi}$. Taking this for granted, we can
use the calculations in \cite{Hel} as a source of examples. In particular,
take $G=\mathrm{GL}_{2}$, and let $\mathrm{St}$ be the Steinberg
representation, so we have a short exact sequence
\[
0\to\mathrm{St}\to\mathrm{Ind}_{B}^{G}\delta_{B}\to\mathbf{1}\to0.
\]
Using \cite[Theorem 4.34 and Remark 4.43]{Hel} it is not difficult
to see that $R_{\psi}$ should send this sequence to the distinguished
triangle
\[
i_{Z\ast}\mathcal{O}_{Z}\to i_{y\ast}\mathcal{O}_{y}\to i_{Z\ast}\mathcal{L}[1]\to.
\]
Here $Z\simeq\mathbf{A}^{1}/\mathbf{G}_{m}^{2}\subset\mathrm{Par}_{G}$
is the closure of the orbit of the Steinberg parameter, $y\simeq B(\mathbf{G}_{m}^{2})\in Z$
is the unique closed orbit, $i_{Z}$ and $i_{y}$ are the evident
closed immersions, and $\mathcal{L}$ is the line bundle on $Z$ of
functions which vanish at $y$. Note that this triangle is a rotation
of the more natural triangle
\[
i_{Z\ast}\mathcal{L}\to i_{Z\ast}\mathcal{O}_{Z}\to i_{y\ast}\mathcal{O}_{y}\to
\]
which is an honest short exact sequence in $\mathrm{Coh}^{\heartsuit}$.

It is very instructive to see how this sequence interacts with duality.
On the $\mathrm{Bun}_{G}$ side, it is easy to see that $\mathbf{D}_{\mathrm{coh}}$
sends $\mathbf{1}$ to $\mathrm{St}[-2]$ and $\mathrm{St}$ to $\mathbf{1}[-2]$.
It is also well-known that $\mathbf{D}_{\mathrm{coh}}$ intertwines
$\mathrm{Ind}_{B}^{G}$ and $\mathrm{Ind}_{\overline{B}}^{G}$. Putting
these observations together appropriately, we see that $\mathbf{D}_{\mathrm{coh}}$
sends our initial short exact sequence to the distinguished triangle
\[
\mathrm{St}[-2]\to\mathrm{Ind}_{\overline{B}}^{G}\delta_{\overline{B}}[-2]\to\mathbf{1}[-2]\to,
\]
which is isomorphic to the sequence we began with shifted by $-2$.
In order for this to be compatible with duality on the \emph{L-}parameter
side, we must have $\mathbf{D}_{\mathrm{tw.GS}}i_{Z\ast}\mathcal{O}_{Z}\simeq i_{Z\ast}\mathcal{L}[-1]$,
or equivalently we must have
\[
\omega_{Z}=i_{Z}^{!}\mathcal{O}_{\mathrm{Par}_{G}}\simeq\mathcal{L}[-1].
\]
This is indeed true, and follows from an explicit calculation via
the isomorphism $Z\simeq\mathbf{A}^{1}/\mathbf{G}_{m}^{2}$ noted
above.
\end{example}

Examination of this example and various other examples in \cite{Hel},
together with some optimism, leads to the following speculation. Let
$G$ be quasisplit with a fixed Whittaker datum as usual. Let $\pi\in\mathrm{Rep}(G(E))$
be irreducible and $\psi$-generic, with Fargues-Scholze parameter
$\phi$. Assume that $\phi(\mathrm{Fr})$ is \emph{regular semisimple}.
Write $Z=q^{-1}(x_{\phi})$, so this is an Artin stack which comes
with a tautological closed immersion $i:Z\hookrightarrow\mathrm{Par}_{G}$.
Our assumption on $\phi$ guarantees that $Z$ is extremely nice,
and is nothing more than the Vogan variety for the infinitesimal character
$\phi$. Explicitly, there is a presentation
\[
Z\cong\left\{ N\in\mathfrak{g}^{\mathrm{ad}\phi(I_{E})}\mid\mathrm{ad}\phi(\mathrm{Fr})\cdot N=q^{-1}N\right\} /S_{\phi}.
\]

\textbf{Question. }Is it true that $c_{\psi}(i_{1!}\pi)\simeq i_{\ast}\mathcal{O}_{Z}$?
\begin{example}
In this example we illustrate how complicated the functor $i_{b\sharp}^{\mathrm{ren}}$
can be in general. Take $E=\mathbf{Q}_{p}$ with $p>2$, and $G=\mathrm{GSp}_{4}$,
so we have an exceptional isomorphism $\hat{G}\simeq\mathrm{GSp}_{4}\subset\mathrm{GL}_{4}$.
Let $b'$ be the element such that $\mathcal{E}_{b'}\simeq\mathcal{O}(\tfrac{1}{2})\oplus\mathcal{O}(-\tfrac{1}{2})$.
Note that $b'$ is an immediate specialization of the point $b=1$.
There is a natural identification $G_{b'}(\mathbf{Q}_{p})\simeq D^{\times}\times\mathbf{Q}_{p}^{\times}$,
where $D/\mathbf{Q}_{p}$ is the usual quaternion algebra. Let $\rho$
be an irreducible $D^{\times}$-representation of dimension $>1$
with trivial central character, with \emph{L}-parameter $\phi_{\rho}:W_{\mathbf{Q}_{p}}\to\mathrm{GL}_{2}(\overline{\mathbf{Q}_{\ell}}$).
There is then a unique supercuspidal representation $\pi$ of $\mathrm{GSp}_{4}(\mathbf{Q}_{p})$
with semisimple \emph{L-}parameter $\phi_{\rho}\oplus|\cdot|^{\frac{1}{2}}\oplus|\cdot|^{-\frac{1}{2}}$.
Note that this parameter is \emph{not }semisimple generic in the sense
of section 2.3: it is the semisimplification of a discrete parameter
with nontrivial monodromy, which is the true \emph{L}-parameter of
$\pi$.

\textbf{Claim. }For $n\in\{0,1\}$, $H^{n}(i_{1}^{\ast}i_{b'\sharp}^{\mathrm{ren}}(\rho|\mathrm{Nm}|^{-\frac{1}{2}}\boxtimes|\cdot|^{-\frac{1}{2}}))$
contains $\pi$ with multiplicity one. Moreover, this is the entire
supercuspidal part of $H^{\ast}(i_{1}^{\ast}i_{b'\sharp}^{\mathrm{ren}}(\rho|\mathrm{Nm}|^{-\frac{1}{2}}\boxtimes|\cdot|^{-\frac{1}{2}}))$.

The proof of this is very indirect. The idea is to compute 
\[
A\overset{\mathrm{def}}{=}i_{1}^{\ast}T_{V}i_{b!}i_{P_{b}}^{G_{b}}(\rho|\mathrm{Nm}|^{-\frac{1}{2}}\boxtimes|\cdot|^{-\frac{1}{2}})
\]
in two different ways. Here $P=MU\subset G$ is the Siegel parabolic,
$b\in B(G)$ is (a representative in $M(\mathbf{Q}_{p})$ of) the
basic element such that $\mathcal{E}_{b}\simeq\mathcal{O}(\tfrac{1}{2})^{2}$,
$P_{b}=M_{b}U_{b}\subset G_{b}$ is the evident twist, and $V$ is
the dual standard representation of $\hat{G}$. Note that $M_{b}\cong G_{b'}$
as inner forms of $M\simeq\mathrm{GL}_{2}\times\mathbf{G}_{m}$. On
one hand, there is a short exact sequence
\[
0\to\mathrm{St}(\rho,1)\to i_{P_{b}}^{G_{b}}(\rho|\mathrm{Nm}|^{-\frac{1}{2}}\boxtimes|\cdot|^{-\frac{1}{2}})\to\mathrm{Sp}(\rho,1)\to0
\]
(cf. Proposition 5.3.(i) of \cite{GT}), and one can compute the supercuspidal
part of $i_{1}^{\ast}T_{V}i_{b!}(-)$ applied to the outer terms of
this sequence using a result of Ito-Mieda (Theorem 3.1 of \cite{M};
note that $\mathrm{St}(\rho,1)=\rho_{\mathrm{disc}}$ and $\mathrm{Sp}(\rho,1)=\rho_{\mathrm{nt}}$
in Mieda's notation). This yields a two-term filtration on the supercuspidal
part of $A$ with comprehensible graded pieces.

On the other hand, we can rewrite $A$ as $i_{1}^{\ast}T_{V}\mathrm{Eis}_{P}^{G}i_{b!}^{M}(\rho|\mathrm{Nm}|^{-\frac{1}{2}}\boxtimes|\cdot|^{-\frac{1}{2}})$.
Using the filtered commutation of $\mathrm{Eis}$ with Hecke operators
(Conjecture \ref{conj:Eisheckefiltered}), the latter expression acquires
a three-term filtration with graded pieces $i_{1}^{\ast}\mathrm{Eis}_{P}^{G}T_{W_{k}}i_{b!}^{M}(\rho|\mathrm{Nm}|^{-\frac{1}{2}}\boxtimes|\cdot|^{-\frac{1}{2}})$.
Here the $W_{k}$ are the irreducible algebraic $\hat{M}$-representations
whose inflations to the Klingen parabolic\footnote{We remind the reader that under the exceptional identification of
$\mathrm{GSp}_{4}$ with its own dual group, the Klingen and Siegel
parabolics correspond.} $\hat{P}$ yield the three irreducible subquotients of $V|\hat{P}$.
These graded pieces can be computed explicitly, and after passing
to the supercuspidal part only one of them survives and yields exactly
the supercuspidal part of $i_{1}^{\ast}i_{b'\sharp}^{\mathrm{ren}}(\rho|\mathrm{Nm}|^{-\frac{1}{2}}\boxtimes|\cdot|^{-\frac{1}{2}})$.
A careful comparison of these two calculations now yields the claim.

Note that by adjunction, the claim easily implies that $i_{b'}^{\ast\mathrm{ren}}i_{1\ast}\pi$
is nonzero, with $\rho|\mathrm{Nm}|^{-\frac{1}{2}}\boxtimes|\cdot|^{-\frac{1}{2}}$
occurring as a subquotient. In the early days of the Fargues-Scholze
program, many people (including the present author) hoped the functors
$i_{b}^{\ast}i_{1\ast}$ would have a reasonably simple explicit description
in terms of Jacquet modules and other concrete representation-theoretic
operations. Examples of this sort suggest that such hopes are woefully
misguided.
\end{example}

\begin{xca}
Give a similar analysis of the supercuspidal part of $H^{n}(i_{1}^{\ast}i_{b'\sharp}^{\mathrm{ren}}(\rho|\mathrm{Nm}|^{\frac{1}{2}}\boxtimes|\cdot|^{-\frac{1}{2}}))$
- note the change in the twisting exponent on $\rho$. (Hint: There
is a relevant variant of the short exact sequence used above where
Sp and St are swapped.)
\end{xca}

\subsection{Grothendieck groups and vanishing conjectures}

Since $D(\mathrm{Bun}_{G})$ is cocomplete, its $K_{0}$ is zero by
the usual trick. On the other hand, $D(\mathrm{Bun}_{G})_{\mathrm{fin}}$
is much smaller, and its $K_{0}$ is very interesting.

For any $b$, $i_{b}^{\ast\mathrm{ren}}$ induces a map $[i_{b}^{\ast\mathrm{ren}}]:K_{0}D(\mathrm{Bun}_{G})_{\mathrm{fin}}\to K_{0}\mathrm{Rep}(G_{b}(E))_{\mathrm{fin}}$,
as well as maps $[i_{b!}^{\mathrm{ren}}],[i_{b\sharp}^{\mathrm{ren}}]:K_{0}\mathrm{Rep}(G_{b}(E))_{\mathrm{fin}}\to K_{0}D(\mathrm{Bun}_{G})_{\mathrm{fin}}$.
Note that $[i_{b\sharp}^{\mathrm{ren}}]$ is well-defined, since we
know from the proof of Theorem \ref{thm:finiteBZduality} that $i_{b\sharp}^{\mathrm{ren}}\pi$
is a finite sheaf for $\pi$ of finite length. These clearly assemble
into maps
\[
\gamma_{!}:\bigoplus_{b\in B(G)}K_{0}\mathrm{Rep}(G_{b}(E))_{\mathrm{fin}}\overset{\sum_{b}[i_{b!}^{\mathrm{ren}}]}{\longrightarrow}K_{0}D(\mathrm{Bun}_{G})_{\mathrm{fin}},
\]
\[
\gamma_{\sharp}:\bigoplus_{b\in B(G)}K_{0}\mathrm{Rep}(G_{b}(E))_{\mathrm{fin}}\overset{\sum_{b}[i_{b\sharp}^{\mathrm{ren}}]}{\longrightarrow}K_{0}D(\mathrm{Bun}_{G})_{\mathrm{fin}}
\]
and
\[
\gamma^{\ast}:K_{0}D(\mathrm{Bun}_{G})_{\mathrm{fin}}\overset{\{[i_{b}^{\ast\mathrm{ren}}]\}_{b}}{\longrightarrow}\bigoplus_{b\in B(G)}K_{0}\mathrm{Rep}(G_{b}(E))_{\mathrm{fin}}.
\]

\begin{prop}
The maps $\gamma_{!}$, $\gamma_{\sharp}$ and $\gamma^{\ast}$ are
isomorphisms of abelian groups. The map $\gamma^{\ast}$ is a left
inverse of $\gamma_{!}$.
\end{prop}

\begin{proof}
The last part is clear, since $i_{b}^{\ast\mathrm{ren}}i_{b!}^{\mathrm{ren}}=\mathrm{id}$
and $i_{b'}^{\ast\mathrm{ren}}i_{b!}^{\mathrm{ren}}=0$ for all $b'\neq b$.
The first part is an easy consequence of the semiorthogonal decomposition
together with the fact that $i_{b\sharp}^{\mathrm{ren}}$ preserves
finite sheaves as in the proof of Theorem \ref{thm:finiteBZduality}.
\end{proof}
Of course, the unrenormalized variants of these statements are also
true.

By Theorem \ref{thm:finiteBZduality}, Bernstein-Zelevinsky duality
induces a well-defined involution on $K_{0}D(\mathrm{Bun}_{G})_{\mathrm{fin}}$,
which we denote $[\mathbf{D}_{\mathrm{BZ}}]$. Similarly, cohomological
duality induces an involution on $K_{0}\mathrm{Rep}(G_{b}(E))_{\mathrm{fin}}$
for each $b$; taking the direct sum over $b$, we get an involution
$[\mathbf{D}_{\mathrm{coh}}]$ on $\bigoplus_{b\in B(G)}K_{0}\mathrm{Rep}(G_{b}(E))_{\mathrm{fin}}$.
\begin{prop}
\label{prop:vanishingequivalents}The following statements are equivalent.

\emph{1) }The maps $\gamma_{!}$ and $\gamma_{\sharp}$ are equal.

\emph{2) }For all $b'\neq b$, the map $[i_{b'}^{\ast\mathrm{ren}}i_{b\sharp}^{\mathrm{ren}}]:K_{0}\mathrm{Rep}(G_{b}(E))_{\mathrm{fin}}\to K_{0}\mathrm{Rep}(G_{b'}(E))_{\mathrm{fin}}$
is zero.

\emph{2') }For all $b'\neq b$, the map $[i_{b'}^{\ast}i_{b\sharp}]:K_{0}\mathrm{Rep}(G_{b}(E))_{\mathrm{fin}}\to K_{0}\mathrm{Rep}(G_{b'}(E))_{\mathrm{fin}}$
is zero.

\emph{3)} There is an equality of maps $\gamma_{!}\circ[\mathbf{D}_{\mathrm{coh}}]=[\mathbf{D}_{\mathrm{BZ}}]\circ\gamma_{!}$.

\emph{4)} There is an equality of maps $\gamma^{\ast}\circ[\mathbf{D}_{\mathrm{BZ}}]=[\mathbf{D}_{\mathrm{coh}}]\circ\gamma^{\ast}$.

\emph{5) }The maps $\gamma^{\ast}$ and $\gamma_{!}$ are mutually
inverse.
\end{prop}

\begin{proof}
The equivalence of 1) and 2) is immediate from the definitions together
with the fact that $\gamma^{\ast}$ is a left inverse of $\gamma_{!}$.
The equivalence of 2) and 2') is trivial. The equivalence of 1) and
3) follows from the equality $\gamma_{\sharp}\circ[\mathbf{D}_{\mathrm{coh}}]=[\mathbf{D}_{\mathrm{BZ}}]\circ\gamma_{!}$
together with the fact that both dualities are isomorphisms on the
relevant $K_{0}$'s. The equivalence of 3) and 4) follows along similar
lines. Finally, we leave the implications 3) + 4) $\Rightarrow$5)
$\Rightarrow$1) as exercises.
\end{proof}
We now have the following key vanishing conjecture.
\begin{conjecture}
\label{conj:shriekvanishing}The equivalent statements in Proposition
\ref{prop:vanishingequivalents} are true.
\end{conjecture}

This conjecture seems to lie very deep, and it has significant consequences
for the cohomology of local Shimura varieties.

It is fruitful to study this conjecture one \emph{L}-parameter at
a time. More precisely, passing to $\phi$-localizations, one can
introduce an additional grading on the source and target of the $\gamma$'s,
indexed by semisimple \emph{L}-parameters. Conjecture \ref{conj:shriekvanishing}
is then equivalent, for instance, to the statement that for all semisimple
L-parameters $\phi$ and all $b'\neq b$, the map 
\[
[i_{b'}^{\ast\mathrm{ren}}i_{b\sharp}^{\mathrm{ren}}]:K_{0}\mathrm{Rep}(G_{b}(E))_{\mathrm{fin},\phi}\to K_{0}\mathrm{Rep}(G_{b'}(E))_{\mathrm{fin},\phi}
\]
is zero. Note that if $\phi$ is generous, the final part of Conjecture
\ref{conj:generousdeep} predicts that 
\[
i_{b'}^{\ast\mathrm{ren}}i_{b\sharp}^{\mathrm{ren}}:D(G_{b}(E),\overline{\mathbf{Q}_{\ell}})_{\mathrm{fin},\phi}\to D(G_{b'}(E),\overline{\mathbf{Q}_{\ell}})_{\mathrm{fin},\phi}
\]
is \emph{already }identically zero before passing to $K_{0}$! However,
in general this map will not vanish before passing to $K_{0}$.

A good example of this phenomenon is given by the trivial \emph{L}-parameter.
Here, to see that $\gamma_{!}$ and $\gamma_{\sharp}$ are equal on
the summand indexed by the trivial parameter, we need to see that
$[i_{b_{\lambda}!}^{\mathrm{ren}}\pi_{\lambda}]=[i_{b_{\lambda}\sharp}^{\mathrm{ren}}\pi_{\lambda}]$
for all dominant $\lambda$. On the other side of the categorical
conjecture, this corresponds to the expectation that $[\nu_{\ast}A_{\lambda}]=[\nu_{\ast}A_{w_{0}(\lambda)}]$
in $K_{0}\mathrm{Coh}(\mathrm{Par}_{G})_{\mathrm{fin}}$. But this
is true! In fact, the equality $[A_{\lambda}]=[A_{w_{0}(\lambda)}]$
already holds in $K_{0}\mathrm{Coh}(\mathcal{N}/\hat{G})$, which
follows from the results in section 4 of \cite{AH}. Indeed, Achar-Hardesty
prove that twisted Grothendieck-Serre duality induces the \emph{identity
}on $K_{0}\mathrm{Coh}(\mathcal{N}/\hat{G})$, but on the other hand
it is easy to check that $\mathbf{D}_{\mathrm{tw.GS}}A_{\lambda}\simeq A_{w_{0}(\lambda)}$.

There is another, closely related conjecture. Fix a Levi subgroup
$M\subset G$ and a parabolic $P=MU$ containing it. Recall from our
discussion of Eisenstein series that $\mathrm{Eis}_{P}^{G}$ is expected
to preserve compact objects, and also ULA objects with quasicompact
support. With $\overline{\mathbf{Q}_{\ell}}$-coefficients, we thus
expect that it will preserve finite sheaves, and in particular it
will induce a functor
\[
[\mathrm{Eis}_{P}^{G}]:K_{0}D(\mathrm{Bun}_{M})_{\mathrm{fin}}\to K_{0}D(\mathrm{Bun}_{G})_{\mathrm{fin}}.
\]

\begin{conjecture}
\label{conj:eisvanishing}The functor $[\mathrm{Eis}_{P}^{G}]$ depends
only on the Levi $M$.
\end{conjecture}

This is a geometric analogue of the classical fact that for any parabolic
$P=MU\subset G$, the map $[i_{P}^{G}]:K_{0}\mathrm{Rep}(M(E))_{\mathrm{fin}}\to K_{0}\mathrm{Rep}(G(E))_{\mathrm{fin}}$
depends only on the Levi $M$, which is an easy consequence of van
Dijk's formula for the Harish-Chandra character of a parabolic induction
\cite{vD}.
\begin{prop}
Conjecture \ref{conj:shriekvanishing} and Conjecture \ref{conj:eisvanishing}
are equivalent.
\end{prop}

\begin{proof}[Sketch]
 One implication follows immediately from Remark \ref{rem:Eisensteinpush}.
The other implication is much less obvious, and a detailed proof will
appear in \cite{HHS2}.
\end{proof}
It turns out that Conjecture \ref{conj:Eisheckefiltered} and Conjecture
\ref{conj:eisvanishing} together imply a vast generalization of the
``Harris-Viehmann conjecture'' describing the $K_{0}$-class of
the cohomology of non-basic local Shimura varieties in terms of parabolic
inductions. A detailed discussion will appear in \cite{HHS2}.
\begin{xca}
Return to the notation and setup of Section 2.2.

1) (Difficult.) For any $\lambda\in X^{*}(\hat{T})$ and any $w\in W$,
prove that $[A_{\lambda}]=[A_{w(\lambda)}]$ in $K_{0}\mathrm{Coh}(\mathcal{N}/\hat{G})$.

2) Show that 1) is consistent with Conjecture \ref{conj:eisvanishing}
and the categorical conjecture. Hint: Generalize the arguments in
Section 2.2 to show that for $\lambda$ dominant and $w$ arbitrary,
$\mathrm{Eis}_{B}(i_{w(\lambda)!}\mathbf{1})\simeq\mathrm{Eis}_{^{w}B}(i_{\lambda!}\mathbf{1})$
should match with $A_{w(\lambda)}$ under the categorical conjecture.

3) (Difficult.) Assume Conjecture \ref{conj:trivialparametermain}.
Prove that for any $\lambda\in X^{\ast}(\hat{T})^{+}$, 
\[
[\mathscr{F}_{\lambda+2\rho}]=\sum_{w\in W}(-1)^{\ell(w)}[i_{b_{\mathrm{dom}(\lambda+\rho+w\rho)}!}^{\mathrm{ren}}\pi_{\mathrm{dom}(\lambda+\rho+w\rho)}]
\]
in $K_{0}D(\mathrm{Bun}_{G})_{\mathrm{fin}}$. Here $\mathrm{dom}(\mu)$
denotes the unique dominant element in the $W$-orbit of $\mu$. For
$G=\mathrm{PGL}_{3}$ and $\lambda=0$, compute the right-hand side
explicitly.
\end{xca}

\begin{rem}
According to Conjecture \ref{conj:shriekvanishing}, the elements
$[i_{b!}^{\mathrm{ren}}\pi]=[i_{b\sharp}^{\mathrm{ren}}\pi]$ should
coincide, and should give a canonical $\mathbf{Z}$-basis for $K_{0}D(\mathrm{Bun}_{G})_{\mathrm{fin}}$
parametrized by pairs $(b,\pi)$. There is a second canonical $\mathbf{Z}$-basis
parametrized by the same set, given by the elements $[\mathscr{G}_{b,\pi}]$
where $\mathscr{G}_{b,\pi}$ is the irreducible hadal sheaf constructed
in Theorem \ref{thm:irreducible-hadal-sheaves}. In particular, we
can write any $[\mathscr{G}_{b,\pi}]$ uniquely as a finite $\mathbf{Z}$-linear
combination of $[i_{b'!}^{\mathrm{ren}}\pi']$, and vice versa. Experiments
(including the previous exercise) suggest that the coefficients in
the expansion $[\mathscr{G}_{b,\pi}]=\sum_{(b',\pi')}c_{b',\pi'}[i_{b'!}^{\mathrm{ren}}\pi']$
are somewhat complicated, with unpredictable signs. However, in every
example I have worked out, the coefficients in the expansion
\[
[i_{b!}^{\mathrm{ren}}\pi]=\sum_{(b',\pi')}\alpha_{b',\pi'}[\mathscr{G}_{b',\pi'}]
\]
are \emph{nonnegative} integers. At semisimple generic parameters,
this is predicted by the discussion in section 2.3, but this is also
true in some nontrivial cases around the Steinberg parameter for $\mathrm{GL}_{2}$.
Does this positivity phenomenon persist in general?\footnote{The answer is no; Bertoloni Meli showed me a counterexample.}
\end{rem}

\appendix

\section{Homological properties of Whittaker functions }

Fix $E/\mathbf{Q}_{p}$ finite, $G$ (the $E$-points of) a quasisplit
reductive group, $B=TU\subset G$ as usual. If $R$ is any commutative
ring, we write $\mathrm{Mod}_{R}(G)$ for the category of smooth $R[G]$-modules.

Let $\Lambda$ be any $\mathbf{Z}[\tfrac{1}{p},\zeta_{p^{\infty}}]$
algebra, and let $\psi:U\to\Lambda^{\times}$ be any nondegenerate
additive character. We are interested in the space
\[
W_{\psi}=\mathrm{ind}_{U}^{G}(\psi)
\]
of compactly supported Whittaker functions with coefficients in $\Lambda$.
In other words, $W_{\psi}\subset\mathcal{C}(G,\Lambda)$ is the space
of functions such that $f(ug)=\psi(u)f(g)$ for all $u\in U$ and
$g\in G$, $f$ is right-invariant by some open compact subgroup of
$G$, and the support of $f$ has compact image in $U\backslash G$.

In the most classical situation where $\Lambda=\mathbf{C}$, Chan-Savin
proved that $W_{\psi}$ is a projective object in $\mathrm{Mod}_{\mathbf{C}}(G)$
\cite{CS}, and Bushnell-Henniart proved that the summand $W_{\psi,\mathfrak{s}}$
corresponding to an individual Bernstein component $\mathfrak{s}$
is finitely generated \cite{BH}. The following theorem is a generalization
of these results.
\begin{thm}
\emph{\label{thm:Wpsinice}i. }For any $\Lambda$, $W_{\psi}$ is
a projective object in $\mathrm{Mod}_{\Lambda}(G)$.

\emph{ii. }Let $e_{r}:\mathrm{Mod}_{\Lambda}(G)\to\mathrm{Mod}_{\Lambda}(G)_{r}$
be the projector onto the depth $\leq r$ factor category. Then $e_{r}W_{\psi}\in\mathrm{Mod}_{\Lambda}(G)$
is finitely generated projective, and $\mathbf{D}_{\mathrm{coh}}(e_{r}W_{\psi})\simeq e_{r}W_{\psi^{-1}}$.

\emph{iii. }Assume that $\Lambda=\mathbf{C}$. Then $\mathbf{D}_{\mathrm{coh}}(W_{\psi,\mathfrak{s}})\simeq W_{\psi^{-1},\mathfrak{s}^{\vee}}$
for any Bernstein component $\mathfrak{s}$.
\end{thm}

Here as before, $\mathbf{D}_{\mathrm{coh}}(-)$ denotes the cohomological
duality functor $R\mathrm{Hom}(-,\mathcal{C}_{c}^{\infty}(G,\Lambda))$.
We note that a different proof of i. and of the finite generation
statement in ii. was also discovered by Dat-Helm-Kurinczuk-Moss.

The main tools in our analysis of $W_{\psi}$ are Bushnell-Henniart's
finiteness results over $\mathbf{C}$, together with a beautiful approximation
technique due to Rodier \cite{R}, which seems perhaps under-appreciated.
To explain this, fix $\Lambda$-valued Haar measures $dg$ on $G$
and $du$ on $U$. We may assume that the volume of any pro-$p$ open
subgroup is a unit in $\Lambda$. Following Rodier, we may choose
a sequence of pairs $(K_{n},\psi_{n})_{n\geq1}$, where $K_{n}\subset G$
is an open compact pro-$p$ subgroup and $\psi_{n}:K_{n}\to\Lambda^{\times}$
is a smooth character, with the following remarkable properties.
\begin{lyxlist}{00.00.0000}
\item [{P1}] $K_{n}\cap U$ is an increasing sequence of groups, with $U=\bigcup_{n}K_{n}\cap U$.
\item [{P2}] $K_{n}\cap\overline{B}$ is a decreasing sequence of groups,
with $\{e\}=\bigcap_{n}K_{n}\cap\overline{B}$.
\item [{P3}] For all $n$, $K_{n}=(K_{n}\cap U)\cdot(K_{n}\cap\overline{B})=(K_{n}\cap\overline{B})\cdot(K_{n}\cap U)$.
\item [{P4}] $\psi_{n}|_{K_{n}\cap\overline{B}}=1$ and $\psi_{n}|_{K_{n}\cap U}=\psi|_{K_{n}\cap U}$.
\item [{P5}] Let $\varphi_{n}\in\mathcal{C}_{c}^{\infty}(G,\Lambda)$ be
the function on $G$ obtained from $\psi_{n}$ via extension by zero.
Then there is an explicit integer $B$ such that for all $n\geq B$,
\[
\varphi_{n}\ast\varphi_{n+1}\ast\varphi_{n}=\mathrm{vol}(K_{n})\mathrm{vol}(K_{n+1}\cap K_{n})\varphi_{n}
\]
where $\ast$ denotes the usual convolution structure on $\mathcal{C}_{c}^{\infty}(G,\Lambda)$
relative to the chosen Haar measure.
\end{lyxlist}
The existence of a sequence of pairs satisfying P1-P4 is not particularly
hard, but P5 lies significantly deeper. Note that Rodier's paper only
treated the case of $G$ split and $p$ sufficiently large, in which
case P5 corresponds to the key ``Lemme 5'', whose proof occupies
the entirety of \cite[Section 5]{R}. In the present generality, the
existence of a system satisfying P1-P5 follows from Varma's paper
\cite{V}.\footnote{I am very grateful to Sandeep Varma for helpful communications regarding
his paper \cite{V}.} (Note that we write $K_{n}$ where Varma writes $G_{n}'$.) The essential
point is that P5 follows from Lemma 9 of \cite{V}, upon observing
that in the present situation, the element $Y$ is principal nilpotent
and contained in $\overline{\mathfrak{b}}$, so $G(Y)\subset\overline{B}$
and then (by P2) $G'_{n+1}\cap G(Y)\subset G'_{n+1}\cap\overline{B}\subset G'_{n}$,
so we may take $\mathcal{Y}_{n}=\{e\}$ in Lemma 9 of \cite{V}. Note
that \cite{R} and \cite{V} are written in the usual setting where
the coefficient ring is $\mathbf{C}$, but it is immediate that the
arguments work for any $\Lambda$ as above.

With these preparations in hand, consider the compact inductions $\mathrm{ind}_{K_{n}}^{G}(\psi_{n})$.
Then for all positive integers $m,n$, we have the $G$-equivariant
map
\begin{align*}
A_{n}^{m}:\mathrm{ind}_{K_{n}}^{G}(\psi_{n}) & \to\mathrm{ind}_{K_{m}}^{G}(\psi_{m})\\
f & \mapsto\frac{1}{\mathrm{vol}(K_{m})}\int_{K_{m}}\psi_{m}(k)^{-1}f(kg)dk.
\end{align*}
Note that $A_{n}^{m}$ is well-defined independently of the choice
of Haar measure $dk$ on $K_{m}$. Note also that $A_{m}^{\ell}\circ A_{n}^{m}=A_{n}^{\ell}$
for all $\ell\geq m\geq n$, so the representations $\mathrm{ind}_{K_{n}}^{G}(\psi_{n})$
form an inductive system with transition maps given by the $A_{n}^{m}$'s.
It is also easy to check that for $m\geq n$, the formula
\[
A_{n}^{m}(f)=\frac{1}{\mathrm{vol}(K_{m}\cap U)}\int_{K_{m}\cap U}\psi(u)^{-1}f(ug)du
\]
is true. On the other hand, for any $n\geq1$, we have a $G$-equivariant
map
\begin{align*}
\phi_{n}:\mathrm{ind}_{K_{n}}^{G}(\psi_{n}) & \to W_{\psi}\\
f & \mapsto\int_{U}\psi(u)^{-1}f(ug)du.
\end{align*}
Using the preceding formula for $A_{n}^{m}$, it is easy to check
that $\phi_{n}=\phi_{m}\circ A_{n}^{m}$ for all $m\geq n$, so passing
to the colimit we get a $G$-equivariant map
\[
\phi_{\infty}:\mathrm{colim}_{n}\mathrm{ind}_{K_{n}}^{G}(\psi_{n})\to W_{\psi}.
\]
It is then formal (Proposition 3 in Rodier's paper) that $\phi_{\infty}$
is an isomorphism. Note that all the arguments so far used only P1-P4. 

The key remaining observation is that for all $m\geq n\geq B$, $A_{n}^{m}$
is injective and $A_{m}^{n}$ is surjective. This follows immediately
from the claim that $A_{m}^{n}\circ A_{n}^{m}=\frac{\mathrm{vol}(K_{m}\cap K_{n})}{\mathrm{vol}(K_{m})}\mathrm{id}$
for all $m\geq n\geq B$ (note that the fraction here is a power of
$p$, hence a unit in $\Lambda$). To prove this claim, first note
(using P1-P3) that $\frac{\mathrm{vol}(K_{m}\cap K_{n})}{\mathrm{vol}(K_{m})}=\frac{\mathrm{vol}(K_{n}\cap U)}{\mathrm{vol}(K_{m}\cap U)}$,
whence the claim reduces by induction to the special case $m=n+1$.
The case $m=n+1$ in turn is a direct consequence of P5 above, upon
noting that $\mathrm{ind}_{K_{n}}^{G}(\psi_{n})\subset\mathcal{C}_{c}^{\infty}(G,\Lambda)$
is exactly the image of the map
\begin{align*}
\mathcal{C}_{c}^{\infty}(G,\Lambda) & \to\mathcal{C}_{c}^{\infty}(G,\Lambda)\\
f & \mapsto\varphi_{n}\ast f,
\end{align*}
and that the map $A_{n+1}^{n}\circ A_{n}^{n+1}$ coincides with further
left convolution by $\varphi_{n}\ast\varphi_{n+1}$ up to an explicit
scaling. For details, see the proof of \cite[Proposition 4]{R}. 

The equality $A_{m}^{n}\circ A_{n}^{m}=\frac{\mathrm{vol}(K_{m}\cap K_{n})}{\mathrm{vol}(K_{m})}\mathrm{id}$
immediately implies that for all $m\geq n\geq B$, $A_{n}^{m}$ is
a split injection, whence also $\phi_{n}=\phi_{\infty}\circ\mathrm{colim}_{m}A_{n}^{m}$
is a split injection. 
\begin{proof}[Proof of Theorem A.0.1.]
First we prove i. With the above preparations on Rodier approximation
at hand, this result proves itself. The only thing to observe is that
each $\mathrm{ind}_{K_{n}}^{G}(\psi_{n})$ is projective, so Rodier
approximation presents $W_{\psi}$ as the colimit of a directed system
of projective objects along split injective transition maps. But the
colimit of any such directed system is itself a projective object.

Next we prove iii. Note first that in the discussion of Rodier approximation,
we may replace $\psi$ resp. $\psi_{n}$ with $\psi^{-1}$ resp. $\psi_{n}^{-1}$
everywhere without changing the validity of any statements. In particular,
we may simultaneously write $W_{\psi}\simeq\mathrm{colim}_{n}\mathrm{ind}_{K_{n}}^{G}(\psi_{n})$
and $W_{\psi^{-1}}\simeq\mathrm{colim}_{n}\mathrm{ind}_{K_{n}}^{G}(\psi_{n}^{-1})$,
where for all sufficiently large $n$ the transition maps in both
systems are (split) injections. Now pick any Bernstein component $\mathfrak{s}$.
Again, we have $W_{\psi,\mathfrak{s}}\simeq\mathrm{colim}_{n}\mathrm{ind}_{K_{n}}^{G}(\psi_{n})_{\mathfrak{s}}$
and $W_{\psi^{-1},\mathfrak{s}^{\vee}}\simeq\mathrm{colim}_{n}\mathrm{ind}_{K_{n}}^{G}(\psi_{n}^{-1})_{\mathfrak{s}^{\vee}}$,
where for all sufficiently large $n$ the transition maps are injections.
Since $W_{\psi,\mathfrak{s}}$ is finitely generated, this immediately
implies that the map $\mathrm{ind}_{K_{n}}^{G}(\psi_{n})_{\mathfrak{s}}\to W_{\psi,\mathfrak{s}}$
is an isomorphism for all sufficiently large $n$. Repeating the same
argument, we also get that the map $\mathrm{ind}_{K_{n}}^{G}(\psi_{n}^{-1})_{\mathfrak{s}^{\vee}}\to W_{\psi^{-1},\mathfrak{s}^{\vee}}$
is an isomorphism for all sufficiently large $n$. But $\mathbf{D}_{\mathrm{coh}}(\mathrm{ind}_{K_{n}}^{G}(\psi_{n})_{\mathfrak{s}})\simeq\mathrm{ind}_{K_{n}}^{G}(\psi_{n}^{-1})_{\mathfrak{s}^{\vee}}$,
which gives the result.

Finally we prove ii. Observe first that the formation of the representations
$W_{\psi}$ and $\mathrm{ind}_{K_{n}}^{G}(\psi_{n})$ and the maps
$A_{n}^{m}$, $\phi_{n}$, $\phi_{\infty}$ are compatible with extension
of scalars along any ring map $\Lambda\to\Lambda'$. Since BZ duality
and the depth $\leq r$ projector $e_{r}$ are also compatible with
any extension of scalars, we are immediately reduced to the universal
case $\Lambda=\mathbf{Z}[\tfrac{1}{p},\zeta_{p^{\infty}}]$. Pick
an embedding $\Lambda\to\mathbf{C}$; we will write $(-)_{\mathbf{C}}$
for objects obtained by the evident extension of scalars along this
map. 

Since $\mathrm{Mod}_{\mathbf{C}}(G)_{r}$ is a product of finitely
many Bernstein components, the proof of iii. shows that the maps $e_{r}\phi_{n,\mathbf{C}}:e_{r}\mathrm{ind}_{K_{n}}^{G}(\psi_{n})_{\mathbf{C}}\to e_{r}W_{\psi,\mathbf{C}}$
and $e_{r}\mathrm{ind}_{K_{n}}^{G}(\psi_{n}^{-1})_{\mathbf{C}}\to e_{r}W_{\psi^{-1},\mathbf{C}}$
are isomorphisms for sufficiently large $n$. Now consider the maps
$e_{r}\phi_{n}:e_{r}\mathrm{ind}_{K_{n}}^{G}(\psi_{n})_{\mathbf{}}\to e_{r}W_{\psi}$.
Since $\phi_{n}$ is a split injection for all sufficiently large
$n$, we get also that $e_{r}\phi_{n}$ is a split injection for all
sufficiently large $n$. Moreover, the source and target of $\phi_{n}$
are projective $\Lambda$-modules,\footnote{For $\mathrm{ind}_{K_{n}}^{G}(\psi_{n})$ this is clear, and for $W_{\psi}$
it follows from the proof of i. Is it clear from first principles
that $W_{\psi}$ is projective as a $\Lambda$-module?} so also the source and target of $e_{r}\phi_{n}$ are projective
$\Lambda$-modules. In particular, $\mathrm{coker}\,e_{r}\phi_{n}$
is a projective $\Lambda$-module for sufficiently large $n$. But
we've already established that
\[
(\mathrm{coker}\,e_{r}\phi_{n})\otimes_{\Lambda}\mathbf{C}\simeq\mathrm{coker}\,e_{r}\phi_{n,\mathbf{C}}=0
\]
for all sufficiently large $n$. This implies that $\mathrm{coker}\,e_{r}\phi_{n}=0$
for all sufficiently large $n$, so $e_{r}\phi_{n}:e_{r}\mathrm{ind}_{K_{n}}^{G}(\psi_{n})\to e_{r}W_{\psi}$
is an isomorphism for all sufficiently large $n$. Since $\mathrm{ind}_{K_{n}}^{G}(\psi_{n})$
and its summand $e_{r}\mathrm{ind}_{K_{n}}^{G}(\psi_{n})$ are finitely
generated, we deduce that $e_{r}W_{\psi}$ is finitely generated.

Repeating the same argument with inverses on all characters, we also
get that $e_{r}\mathrm{ind}_{K_{n}}^{G}(\psi_{n}^{-1})\to e_{r}W_{\psi^{-1}}$
is an isomorphism for all sufficiently large $n$. Since $\mathbf{D}_{\mathrm{coh}}(\mathrm{ind}_{K_{n}}^{G}(\psi_{n}))=\mathrm{ind}_{K_{n}}^{G}(\psi_{n}^{-1})$
and $e_{r}$ commutes with $\mathbf{D}_{\mathrm{coh}}$, we compute
that 
\begin{align*}
\mathbf{D}_{\mathrm{coh}}(e_{r}W_{\psi}) & \simeq\mathbf{D}_{\mathrm{coh}}(e_{r}\mathrm{ind}_{K_{n}}^{G}(\psi_{n}))\\
 & =e_{r}\mathbf{D}_{\mathrm{coh}}(\mathrm{ind}_{K_{n}}^{G}(\psi_{n}))\\
 & =e_{r}\mathrm{ind}_{K_{n}}^{G}(\psi_{n}^{-1})\\
 & \simeq e_{r}W_{\psi^{-1}}.
\end{align*}
This concludes the proof.
\end{proof}
\begin{cor}
Suppose $\Lambda$ is an algebraically closed field of characteristic
$\neq p$ and $\pi\in\mathrm{Mod}_{\Lambda}(G)$ is irreducible. Then
$\pi$ is $\psi$-generic if and only if $\pi^{\vee}$ is $\psi^{-1}$-generic.
More generally, for any $\pi$ of finite length, there is an isomorphism
\[
\mathrm{Hom}(W_{\psi},\pi)^{\ast}\simeq\mathrm{Hom}(W_{\psi^{-1}},\pi^{\vee}).
\]
\end{cor}

When $\Lambda=\mathbf{C}$, the first part of this corollary was previously
proved by Prasad \cite[Lemma 2]{Pra}.
\begin{proof}
Choose $r$ large so $\pi$ is of depth $\leq r$, and then choose
$n$ large enough so that $e_{r}W_{\psi}\simeq e_{r}\mathrm{ind}_{K_{n}}^{G}(\psi_{n})$.
Then arguing as in the previous proof we get isomorphisms
\begin{align*}
\mathrm{Hom}(W_{\psi},\pi) & \simeq\mathrm{Hom}(e_{r}W_{\psi},\pi)\\
 & \simeq\mathrm{Hom}(e_{r}\mathrm{ind}_{K_{n}}^{G}(\psi_{n}),\pi)\\
 & \simeq\mathrm{Hom}(\mathrm{ind}_{K_{n}}^{G}(\psi_{n}),\pi)\\
 & \simeq(\pi|_{K_{n}}\otimes\psi_{n}^{-1})^{K_{n}}
\end{align*}
and similarly $\mathrm{Hom}(W_{\psi^{-1}},\pi^{\vee})\simeq(\pi^{\vee}|_{K_{n}}\otimes\psi_{n})^{K_{n}}$.
We conclude by the easy fact that $(\pi|_{K_{n}}\otimes\psi_{n}^{-1})^{K_{n}}$
and $(\pi^{\vee}|_{K_{n}}\otimes\psi_{n})^{K_{n}}$ are canonically
dual to each other as $\Lambda$-vector spaces.
\end{proof}

\subsection{A reasonable condition}

In this section we again take our coefficients to be $\Lambda=\overline{\mathbf{Q}_{\ell}}$.
Let $G/E$ be any connected reductive group.
\begin{defn}
\label{def:reasonablegroups}i. The group $G$ is \emph{reasonable}
if the Fargues-Scholze map $\Psi_{G}^{\mathrm{geom}}:X_{G}\to X_{G}^{\mathrm{spec}}$
is finite, or equivalently (by Lemma \ref{lem:psihasdiscretefibers})
if the associated map $\pi_{0}X_{G}\to\pi_{0}X_{G}^{\mathrm{spec}}$
has finite fibers.

ii. The group $G$ is \emph{very reasonable }if $G_{b}$ is reasonable
for all $b\in B(G)$.
\end{defn}

It is clear that if $G$ is very reasonable, then $G_{b}$ is very
reasonable for all $b\in B(G)$.
\begin{xca}
Prove that if $G$ is reasonable, then any Levi subgroup $M\subset G$
is reasonable.
\end{xca}

Our main reason for considering these conditions is the following
basic finiteness result, which we leave to the reader as an exercise.
(The key ingredient for i. is Bushnell-Henniart's finiteness theorem
cited above.)
\begin{prop}
\emph{\label{prop:reasonable-implications}i. }If $\mathrm{char}E=0$
and $G$ is quasisplit and reasonable, the functor $\mathcal{F}\mapsto\mathcal{F}\ast i_{1!}W_{\psi}$
sends $\mathrm{Perf^{qc}(}\mathrm{Par}_{G})$ into $D(\mathrm{Bun}_{G})^{\omega}$.

\emph{ii. }If $G$ is very reasonable, then for every $b$ and every
semisimple parameter $\phi:W_{E}\to\,^{L}G(\overline{\mathbf{Q}_{\ell}})$,
the set of irreducible smooth representations $\pi\in\Pi(G_{b})$
with Fargues-Scholze parameter $\phi$ is finite.

\emph{iii. }If $G$ is very reasonable, then every ULA sheaf in $D(\mathrm{Bun}_{G})$
is a filtered colimit of finite sheaves.
\end{prop}

Note that the condition in Proposition \ref{prop:reasonable-implications}.ii
is actually equivalent to very reasonableness. This follows from finiteness
of the map $X_{G_{b}}^{\mathrm{spec}}\to X_{G}^{\mathrm{spec}}$,
Theorem \ref{thm:ibcompatibleparameters}, and basic structure theory
a la Bernstein.

It's also quite plausible that the conclusion of Proposition \ref{prop:reasonable-implications}.iii
could be proved unconditionally for any $G$, but I didn't try very
hard to check this.\footnote{\emph{Update.} Here is a sketch of an unconditional proof for all
$G$. Let $A$ be a ULA sheaf. Since $A=\mathrm{colim}_{U}j_{U!}j_{U}^{*}A$
with $j:U\to\mathrm{Bun}_{G}$ running over inclusions of quasicompact
open substacks, we can assume $A$ is $!$-extended from a quasicompact
open substack. Then by an easy excision argument, we can assume $A=i_{b!}B$
for some fixed $b$ with $B\in D(G_{b}(E),\overline{\mathbf{Q}_{\ell}})$
ULA. Then $B\cong\mathrm{colim}_{S}\oplus_{\mathfrak{s}\in S}B_{\mathfrak{s}}$
where $S$ runs over the (filtered collection of) finite sets of Bernstein
components for $G_{b}(E)$. Now use that each $B_{\mathfrak{s}}$
is finite. This follows from the fact that $B_{\mathfrak{s}}$ is
generated by its $K$-fixed vectors for some small pro-$p$ open $K$
depending only on $\mathfrak{s}$, while $B_{\mathfrak{s}}^{K}$ is
a summand of $B^{K}\in\mathrm{Perf}(\overline{\mathbf{Q}_{\ell}})$
with the last inclusion coming from our ULA assumption, so $H^{\ast}(B_{\mathfrak{s}}^{K})$
has finite length as a module over the Hecke algebra $\mathcal{C}_{c}(K\backslash G_{b}(E)/K)$.} If $G$ is very reasonable, however, it is nearly trivial: if $A$
is ULA, then for any semisimple parameter $\phi$ and any quasicompact
open substack $j:U\to\mathrm{Bun}_{G}$, the sheaf $j_{!}j^{\ast}A_{\phi}$
is finite by the very reasonableness condition, and the natural map
\[
\mathrm{colim}_{U,S}\oplus_{\phi\in S}j_{!}j^{\ast}A_{\phi}\to A
\]
is an isomorphism (using the decomposition $A\cong\oplus_{\phi}A_{\phi}$
proved in \cite{H}), where $S$ runs over (the filtered collection
of) finite sets of semisimple \emph{L}-parameters.

Of course we expect that every group is very reasonable. Right now,
we know that $\mathrm{GL}_{n}$ is very reasonable \cite{FS,HKW},
as well as $\mathrm{GSp}_{4}$ with $E/\mathbf{Q}_{p}$ unramified
and $p>2$ \cite{Ham2}, unramified $\mathrm{(G)U}_{2n+1}/\mathbf{Q}_{p}$
\cite{BMHN}, $\mathrm{SO}_{2n+1}$ with $E/\mathbf{Q}_{p}$ unramified
and $p>2$ (H., unpublished), and groups obtained from these by passing
to derived groups, products, central isogenies, twisted Levis, etc.
To my knowledge, there is no group which is known to be reasonable
but not known to be very reasonable.
\begin{xca}
Show that if $G$ and $G'$ have isomorphic adjoint groups and $G$
is reasonable, then $G'$ is reasonable. Can you formulate a similar
statement for ``very reasonableness''?
\end{xca}

\section{A dimensional classicality criterion for derived stacks, by Adeel
Khan}

We define (co)dimension of derived schemes and stacks on classical
truncations (see \cite[Tag 04N3]{stacks-project} or \cite[0\_IV, 14.1.2, 14.2.4]{EGA}
for schemes, and \cite[Tags 0AFL and 0DRL]{stacks-project} for stacks).
\begin{prop}
\label{prop:soap} Let $f:\mathcal{X}\to\mathcal{Y}$ be a quasi-smooth
morphism of derived $1$-Artin stacks where $\mathcal{Y}_{\mathrm{cl}}$
is Cohen-Macaulay.\footnote{Equivalently, $\mathcal{Y}$ admits a smooth surjection $Y\twoheadrightarrow\mathcal{Y}$
where $Y_{\mathrm{cl}}$ is a Cohen--Macaulay scheme.} If $x\in|\mathcal{X}|$ is a point at which the relative dimension
of $f$ is equal to the relative virtual dimension of $f$, then $\sX\fibprodR_{\sY}\sY_{\cl}$
is classical in a Zariski neighbourhood of $x$. 
\end{prop}

We will make use of the following lemma from \cite[2.3.12]{KR}.
\begin{lem}
\label{lem:nonscientist} Let $Z\to X$ be a quasi-smooth closed immersion
of derived schemes where $X_{\cl}$ is Cohen--Macaulay. Then we have
$-\virdim(Z/X)\ge\codim(Z,X)$, with equality if and only if $Z\fibprodR_{X}X_{\cl}$
is classical in a Zariski neighbourhood of $x$. 
\end{lem}

\begin{proof}[Proof of Proposition \ref{prop:soap}]
 The statement is invariant under replacing $\sY$ by $\sY_{\cl}$
and $\sX$ by $\sX\fibprodR_{\sY}\sY_{\cl}$, so we may assume $\sY$
classical.

Suppose first that $\sX=X$ and $\sY=Y$ are schemes. Since $f:X\to Y$
is quasi-smooth, there exists for every $x\in\abs{X}$ over $y$ a
Zariski neighbourhood $U\subseteq X$ of $x$, a derived scheme $M$
which is smooth over $Y$, and a quasi-smooth closed immersion $U\hookrightarrow M$
over $Y$ (see \cite[Prop. 2.3.14]{KR}). We have 
\begin{align*}
\virdim_{x}(U_{y}/M_{y}) & =\virdim_{x}(U/M),\\
\codim_{x}(U_{y},M_{y}) & \le\codim_{x}(U,M).
\end{align*}
Since $\virdim_{x}(U_{y}/\kappa(y))=\dim_{x}(U_{y})$ by assumption,
we also have 
\[
-\virdim_{x}(U_{y}/M_{y})=\dim_{x}(M_{y})-\dim_{x}(U_{y})=\codim_{x}(U_{y},M_{y})
\]
where the last equality holds because $M_{y}$ is catenary (see \cite[0\_IV Cor.16.5.12]{EGA};
\cite[IV\_2 Prop. 5.1.9]{EGA}). We conclude that 
\[
-\virdim_{x}(U/M)\le\codim_{x}(U,M).
\]
Now Lemma~\ref{lem:nonscientist} implies that $U$ is classical
in a Zariski neighbourhood of $x$.

Next suppose that $\sX=X$ and $\sY=Y$ are algebraic spaces. Choose
an \'etale surjection $X_{0}\twoheadrightarrow X$ where $X_{0}$ is
a derived scheme, and let $x_{0}\in\abs{X_{0}}$ be a lift of the
given point $x\in\abs{X}$. Choose also an \'etale surjection $Y_{0}\twoheadrightarrow Y$
where $Y_{0}$ is a Cohen--Macaulay scheme and a lift $y_{0}\in\abs{Y_{0}}$
of $y$. Since $Y$ has schematic diagonal, $X_{0}\fibprod_{Y}Y_{0}$
is a derived scheme. Applying the case above to the morphism $X_{0}\fibprod_{Y}Y_{0}\to Y_{0}$,
we obtain a Zariski neighbourhood of $(x_{0},y_{0})\in X_{0}\fibprod_{Y}Y_{0}$
which is classical. Its image along the \'etale morphism $X_{0}\fibprod_{Y}Y_{0}\twoheadrightarrow X_{0}\twoheadrightarrow X$
is then a Zariski neighbourhood of $x\in X$ which is classical.

Finally we consider the general case. Choose a smooth surjection $X\twoheadrightarrow\sX$
where $X$ is a derived scheme, a lift $x_{0}\in\abs{X_{0}}$ of the
given point $x\in\abs{X}$, a smooth surjection $Y_{0}\twoheadrightarrow Y$
where $Y_{0}$ is a Cohen--Macaulay scheme, and a lift $y_{0}\in\abs{Y_{0}}$
of $y$. Since $Y$ has representable diagonal, $X\fibprod_{\sY}Y$
is a derived algebraic space. Hence the previous case applied to the
morphism $X\fibprod_{\sY}Y\to Y$ yields a Zariski neighbourhood of
$(x_{0},y_{0})\in X\fibprod_{\sY}Y$ which is classical. Its image
along the smooth morphism $X\fibprod_{\sY}Y\twoheadrightarrow X\twoheadrightarrow\sX$
is then a Zariski neighbourhood of $x\in\sX$ which is classical. 
\end{proof}
\bibliographystyle{amsalpha}
\phantomsection

\end{document}